\documentclass[12pt]{article} 
\usepackage[sectionbib]{natbib}
\usepackage{array,epsfig,fancyheadings,rotating}
\usepackage[]{hyperref}  
\usepackage{varioref} 
\usepackage{xr-hyper} 
\usepackage{booktabs}
\RequirePackage{amsthm,amsmath}
\usepackage{natbib}						
\usepackage{epsfig,amsmath,amssymb,euscript,amsfonts,hyperref}
\usepackage{array}
\usepackage{nccbbb}
\usepackage{verbatim}					
\usepackage[lined,boxed,commentsnumbered, ruled]{algorithm2e} 
\usepackage{graphicx}
\usepackage{bm}
\usepackage{bbm}						
\usepackage{dsfont}						
\usepackage{enumitem}
\usepackage{color}
\hypersetup{citecolor=blue}
\hypersetup{linkcolor=red}
\usepackage{mathabx}  
\usepackage[utf8]{inputenc}				
\usepackage{pdflscape}
\usepackage{tikz}
\usepackage{tikz-qtree}
\usetikzlibrary{arrows,shapes,trees,fit} 
\usetikzlibrary{matrix}
\usepackage[cal=boondox]{mathalfa}
\usepackage{accents}  
\usepackage{marginnote}
\usepackage{xcolor}
\usepackage{todonotes}

\usepackage{mathtools}
\makeatletter
\DeclareRobustCommand\widecheck[1]{{\mathpalette\@widecheck{#1}}}
\def\@widecheck#1#2{%
    \setbox\z@\hbox{\m@th$#1#2$}%
    \setbox\tw@\hbox{\m@th$#1%
       \widehat{%
          \vrule\@width\z@\@height\ht\z@
          \vrule\@height\z@\@width\wd\z@}$}%
    \dp\tw@-\ht\z@
    \@tempdima\ht\z@ \advance\@tempdima2\ht\tw@ \divide\@tempdima\thr@@
    \setbox\tw@\hbox{%
       \raise\@tempdima\hbox{\scalebox{1}[-1]{\lower\@tempdima\box
\tw@}}}%
    {\ooalign{\box\tw@ \cr \box\z@}}}
\makeatother

\newcolumntype{L}[1]{>{\raggedright\let\newline\\\arraybackslash\hspace{0pt}}m{#1}}
\newcolumntype{C}[1]{>{\centering\let\newline\\\arraybackslash\hspace{0pt}}m{#1}}
\newcolumntype{R}[1]{>{\raggedleft\let\newline\\\arraybackslash\hspace{0pt}}m{#1}}
\newcolumntype{H}{>{\setbox0=\hbox\bgroup}c<{\egroup}@{}}  

\renewcommand{\baselinestretch}{1.81}

\numberwithin{equation}{section}   
\theoremstyle{plain}
\newtheorem{theorem}{Theorem}
\newtheorem{lemma}{Lemma}
\newtheorem{corollary}{Corollary}
\newtheorem{proposition}{Proposition}
\newtheorem{property}{Property}
\newtheorem{assumption}{Assumption}

\theoremstyle{definition}
\newtheorem{example}{Example}[section]

\newcommand{\ignore}[1]{}

\definecolor{brown}{rgb}{0.75, 0.25, 0.0}

\DeclareMathOperator{\pr}{P}

\DeclareMathOperator{\E}{E}
\DeclareMathOperator{\Var}{\textnormal{Var}}
\DeclareMathOperator{\Bias}{\mathsf{Bias}}
\DeclareMathOperator{\Cov}{\textnormal{Cov}}

\newcommand{\inP}{  \overset{ \mathrm{pr}  }{\rightarrow} }
\newcommand{\inD}{\Rightarrow}

\newcommand{\simIID}{   \overset{ \text{iid}   }{\sim} }
\newcommand{\simApprox}{  \overset{\mathcal{D}}{\approx} }

\def\T{{ \intercal }} 

\DeclareMathOperator{\tr}{tr}
\DeclareMathOperator{\obs}{obs}
\DeclareMathOperator{\mis}{mis}
\DeclareMathOperator{\com}{com}
\DeclareMathOperator{\lrt}{L}
\DeclareMathOperator{\wt}{W}

\DeclareMathOperator*{\argmax}{arg\,max}
\DeclareMathOperator{\RC}{RC}

\DeclareMathOperator{\df}{df}

\newcommand{\loglik}{L} 

\DeclareMathOperator{\Stack}{S}
\DeclareMathOperator{\pert}{\triangle}
\DeclareMathOperator{\rob}{\lozenge}

\usepackage{sectsty, secdot}
\usepackage{todonotes}
\sectionfont{\fontsize{12}{14pt plus.8pt minus .6pt}\selectfont}
\renewcommand{\theequation}{\thesection\arabic{equation}}
\subsectionfont{\fontsize{12}{14pt plus.8pt minus .6pt}\selectfont}
\RequirePackage[normalem]{ulem}

\providecommand{\add}[1]{{\protect\color{blue}{#1}}}

\usepackage{amsmath}
\usepackage{amssymb}
\usepackage{amsfonts}
\usepackage{multirow}
\usepackage{amsthm}

\begin{document}

\renewcommand{\baselinestretch}{1}

\markright{ \hbox{\footnotesize\rm Statistica Sinica
}\hfill\\[-13pt]
\hbox{\footnotesize\rm
}\hfill }

\markboth{\hfill{\footnotesize\rm \MakeUppercase{Kin Wai Chan} AND \MakeUppercase{Xiao-Li Meng}} \hfill}
{\hfill {\footnotesize\rm MI Likelihood Ratio Tests} \hfill}

\renewcommand{\thefootnote}{}
$\ $\par


\fontsize{11}{14pt plus.8pt minus .6pt}\selectfont \vspace{0.8pc}
\centerline{\large\bf \MakeUppercase{Multiple Improvements of}}
\vspace{2pt} \centerline{\large\bf \MakeUppercase{Multiple Imputation Likelihood Ratio Tests}}
\vspace{.4cm} \centerline{Kin Wai Chan${}^1$ and Xiao-Li Meng${}^2$} \vspace{.4cm} 
\centerline{\it
Department of Statistics, The Chinese University of Hong Kong${}^1$}

\centerline{\it
Department of Statistics, Harvard University${}^2$}\vspace{.55cm} 
\fontsize{9}{11pt plus.8pt minus.6pt}\selectfont

\begin{quotation}
\noindent {\it Abstract:}
Multiple imputation (MI) inference 
handles missing data by imputing the missing values $m$ times, 
and then combining the results from the $m$ complete-data analyses.
However, the existing method for combining likelihood ratio tests (LRTs) 
has multiple defects: (i) the combined test statistic can be negative, but its null distribution is approximated by an $F$-distribution;  
(ii) it is not invariant to re-parametrization; 
(iii) it fails to ensure monotonic power owing to its use of an inconsistent estimator of the fraction of missing information (FMI) under the alternative hypothesis; and
(iv) it requires nontrivial access to the LRT statistic 
as a function of parameters instead of data sets.  
We show, using both theoretical derivations and empirical investigations,  
that essentially all of these problems can be straightforwardly addressed if we are willing to perform an additional LRT by stacking the $m$ completed data sets as one big completed data set.  
This enables users to implement the MI LRT without modifying the complete-data procedure.
A particularly intriguing finding is that the FMI can be estimated consistently by an LRT statistic for testing whether the $m$ completed data sets can be regarded effectively as samples coming from a common model. 
Practical guidelines are provided based on an extensive comparison of existing MI tests. 
Issues related to nuisance parameters are also discussed.

\vspace{9pt}
\par
\end{quotation}\par

\def\thefigure{\arabic{figure}}
\def\thetable{\arabic{table}}

\renewcommand{\theequation}{\thesection.\arabic{equation}}

\fontsize{11}{14pt plus.8pt minus .6pt}\selectfont


\lhead{\fancyplain{}MI Likelihood Ratio Tests}
\rhead[]{\thepage}

\section{Historical Successes and Failures}\label{sec:history}
\subsection{The Need for Multiple Imputation Likelihood-Ratio Tests}\label{sec:mult}
Missing-data problems are ubiquitous in practice, to the extent that the absence of any missingness 
is often a strong indication that the data have been pre-processed or manipulated in some way \citep[e.g.,][]{blocker2013potential}. 
Multiple imputation (MI) \citep{rubin1978,rubin2004multiple} 
has been a preferred method, especially by those who are ill-equipped to handle missingness on their own, owing to a lack of information or skills or resources. 
MI relies on the data collector (e.g., a census bureau) building a reliable imputation model to fill in the missing data $m (\ge 2)$ times. In this way, users can apply their preferred software or procedures
designed for complete data, and do so $m$ times. MI inference is then performed by appropriately 
combining these $m$ complete-data results. Note that in a typical analysis of public MI data, the analyst has no control over or understanding of how the imputation was done, including the choice of the model and $m$, which is often small in reality (e.g.,  $3\le m \le 10$). The analyst's job is to analyze the given $m$ completed data sets as appropriately as possible, but only using complete-data procedures or software routines.

Although MI was designed initially for 
public-use data sets, over the years, it has become a method of choice in general, because it separates 
handling the missingness from the analysis \citep[e.g.,][]{tu1993JASA,rubin1996multiple,rubin2004multiple,Schafer1999,KingHonakerJosephScheve2001,PeughEnders2004,KenwardCarpenter2007,RoseFraser2008,HolanTothFerreiraKarr2010,KimYang2017}. Software routines for performing MI are now available in 
R \citep{SuGelmanHillYajima2011}, 
Stata \citep{RoystonWhite2011},
SAS \citep{SAS_MI}, and SPSS; 
see \cite{HarelZhou2007} and \cite{HortonKleinman2007} 
for summaries.

This convenient separation, 
however, creates an issue of uncongeniality, that is, an incompatibility between the imputation model and the subsequent analysis procedures \citep{meng1994}.  This issue is examined in detail by \citet{XXMeng2017}, who show that uncongeniality is easiest to deal with when the imputer's model is more saturated than the user's model/procedure, and when the user is conducting an efficient analysis, such as a likelihood inference. 
Therefore, this study focuses on conducting MI likelihood ratio tests (LRTs), assuming the imputation model is sufficiently saturated to render the common assumptions 
made in the literature about conducting LRTs with MI valid. 

Like many hypothesis testing procedures in common practice, the exact null distributions of various MI test statistics, LRTs or not, are intractable. This intractability is not computational, but rather statistical,
owing to the well-known issue of a nuisance parameter, that is, the lack of a pivotal quantity, as highlighted by the Behrens--Fisher problem \citep{wallace1980behrens}.  Indeed, the nuisance parameter in the MI context is the so-called ``fraction of missing information" (FMI),  which is determined by the ratio of the between-imputation variance to the within-imputation variance (and its multi-variate counterparts).
Hence, the challenge we face is almost identical to the one faced by the Behrens--Fisher problem, as shown in \cite{meng1994posterior}. Currently the most successful strategy has been to reduce the number of nuisance parameters to one by assuming an equal fraction of missing information (EFMI), a strategy we 
follow as well because our simulation study indicates that it achieves a better compromise between type-I and type-II errors than other strategies we (and others) have tried.

An added challenge in the MI context is that the user's complete-data procedures can be very restrictive. Please update as follows: 
What is available to the user could vary from the entire likelihood function to point estimators (such as the MLE and Fisher information) and to a single $p$-value. 
Therefore, there have been a variety of procedures proposed in the literature, depending on what quantities we assume the user has access to, as we 
review shortly. 

Among them, a promising idea is to directly combine LRT statistics. However, the current execution of this idea \citep{mengRubin92} relies too heavily on the asymptotic equivalence (in terms of the data size, not the number of imputations, $m$) between the LRT and Wald test \textit{under the null}. Its asymptotic validity, unfortunately, does not protect it from quick deterioration for small data sizes, such as delivering a negative ``$F$ test statistic" or FMI.  Worst of all, the test can have essentially zero power because the estimator of the FMI can be badly inconsistent under some alternative hypotheses. The combining rule of \citet{mengRubin92} also requires access to the LRT as a function of parameter values, not just as a function of the data. 
The former is often unavailable from standard software packages. This defective MI LRT, however, has been adopted by textbooks \citep[e.g.,][]{vanBuurenS2012,KimShao2013} and popular software, for example, 
the function \texttt{pool.compare} in the R package \texttt{mice} \citep{vanBuurenGroothuisOudshoorn2011}, 
the function \texttt{testModels} in the R package \texttt{mitml} \citep{GrundRobitzschLuedtke2017}, 
and the function \texttt{milrtest} \citep{Medeiros2008} in the Stata module \texttt{mim} \citep{CarlinGalatiRoyston2008}. 

To minimize the negative impact of this defective LRT test, this study derives MI LRTs that are free of these defects, as detailed in Section~\ref{subsec:defect}. We achieve this mainly by switching the order of two main operators in the combining rule of \citet{mengRubin92}: we maximize the average of the $m$ log-likelihoods instead of averaging their maximizers. This switch, guided by the likelihood principle, renders positivity, invariance, and monotonic power. Other judicious uses of the likelihood functions permit us to overcome the remaining defects.

\subsection{Summary of the Major Findings}
Our major contributions are four-fold:
\begin{itemize}
	\item In terms of statistical principles, 
	        we propose switching the order of two operations, 
	        namely maximization and averaging, 
			in the existing MI LRT statistic, as suggested by
			the likelihood principle.
	        This operation  
	        retrieves the non-negativity and invariance to the re-parametrization of the MI statistic.
	\item In terms of theoretical properties, 
	        a new estimator of the fraction of missing information is proposed. 
	    	It is consistent, regardless of the validity of the null hypothesis,  
			so that the proposed test is monotonically powerful 
			with respect to the discrepancy between the null and alternative hypotheses. 
	\item In terms of computational properties, 
	        the proposed test only requires that users 
			have a standard subroutine for performing a complete-data LRT.
			Thus, unlike the existing MI LRT, 
			users do not need to modify the subroutine 
			in order to evaluate the likelihood function at arbitrary parameter values.
	\item In terms of practical impact,  the proposed test can be implemented easily to     
			replace the flawed MI LRT procedures in the aforementioned software packages and beyond. 
			It immediately resolves the issue of returning a negative $F$-test value.
			In addition, the power loss due to the flaws in the MI LRT procedure
			can be retrieved.
\end{itemize}

The remainder of Section~\ref{sec:history} provides background and notation. 
Section~\ref{sec:improved_MI_LRT} 
discusses the defects of the existing MI LRT and our remedies. 
Section~\ref{sec:comp} investigates the computational requirements,  
including theoretical considerations and comparisons. In particular, Algorithm~\ref{algo:MI_LRT_rob} of Section~\ref{sec:compFeasibleComRule}
computes our most recommended test. 
Section~\ref{sec:example} provides empirical evidence.  
Section~\ref{sec:conclusion} concludes the paper. 
Appendices \ref{supp:results} and \ref{sec:proof} provide
additional investigations, real-life data examples, and proofs.

\subsection{Notation and Complete-Data Tests}\label{sec:review}
Let $X_{\obs}$ and $X_{\mis}$
be, respectively, the observed and missing parts of an intended 
complete data set $X=X_{\com}=\{X_{\obs},X_{\mis}\}$
consisting of $n$ observations. 
Denote the sampling model --- probability or density, depending on the data type --- of $X$ by 
$f(\cdot \mid \psi)$, where $\psi\in\Psi\subseteq\mathbb{R}^h$ is a vector of parameters.
Suppose that we are interested in inferring 
$\theta=\theta(\psi) \in\Theta \subseteq\mathbb{R}^k$, 
which is expressed as a function of $\psi$. 
This definition of $\theta$ is very general. 
For example, $\theta$ can be a sub-vector of $\psi = (\theta^{\T}, \eta^{\T})^{\T}$, 
or a transformation (not necessarily one-to-one) of $\psi$; 
see Section 4.4 of \cite{serfling1980} and Section 6.4.2 of \cite{shao98}.

The goal is to test $H_0:\theta=\theta_0$ when only $X_{\obs}$ is available, 
where $\theta_0$ is a specified vector. 
For example, if $H_0$ puts a $k$-dimensional restriction $R(\psi)=\bm{0}$ on the model parameter $\psi$, 
then $\theta=R(\psi)$ and $\theta_0 = \bm{0}$.
For simplicity, we focus on a two-sided alternative, 
but our approach adapts to general 
LRTs.
Here, we assume $X_{\obs}$ is rich enough that the missing data mechanism is ignorable \citep{rubin1976}, 
or it has been properly incorporated 
by the imputer, 
who may have access to additional confidential data. 

Let $\widehat{\theta}=\widehat{\theta}(X)$, $\widehat{\psi}=\widehat{\psi}(X)$, and 
$\widehat{\psi}_0=\widehat{\psi}_0(X)$ be 
the complete-data MLE of $\theta$, 
complete-data MLE of $\psi$, and
$H_0$-constrained complete-data MLE of $\psi$, respectively.
Furthermore, let  
$U = U_{\theta}=U_{\theta}(X)$ and  $U_{\psi}=U_{\psi}(X)$ be
efficient estimators of $\Var(\widehat{\theta})$ and $\Var(\widehat{\psi})$, respectively,
for example, the inverse of the observed Fisher information. 
Common test statistics for $H_0$
include the Wald statistic $D_{\wt} = d_{\wt}(\widehat{\theta},U)/k$ 
and the LRT statistic $D_{\lrt} = d_{\lrt}(\widehat{\psi}_0,\widehat{\psi}\mid X)/k$, 
where
\begin{eqnarray*}
	d_{\wt}(\widehat{\theta},U) 
		= (\widehat{\theta}-\theta_0)^\T U^{-1} (\widehat{\theta}-\theta_0), 
	\qquad
	d_{\lrt}(\widehat{\psi}_0,\widehat{\psi}\mid X) 
		= 2 \log \frac{ f(X\mid\widehat{\psi})}{f(X\mid\widehat{\psi}_0)}.
\end{eqnarray*}
Under regularity conditions, such as those in Section~4.2.2 and Section~4.4.2 of \citet*{serfling1980},
we have the following classical results.

\begin{property}\label{prop:asyEquiv_LRT_Wald}  
Under $H_0$, 
(i) $D_{\wt} \inD \chi^2_k/k$ and $D_{\lrt}\inD \chi^2_k/k$; and 
(ii) $n(D_{\wt}-D_{\lrt}) \inP 0$ as $n\rightarrow\infty$,  
where ``$\inD$'' and ``$\inP$'' denote  
convergence in distribution and in probability, respectively. 
\end{property}

Testing $H_0$ based on 
$X_{\obs}$ is more involved. 
For MI, let $X^{(\ell)} = \{X_{\obs},X_{\mis}^{(\ell)}\}$, for $\ell=1,\ldots,m$,  
be the $m$ completed data sets,  where 
$X_{\mis}^{(\ell)}$
are drawn 
from a proper imputation model \citep{rubin2004multiple}.  
We then carry out a complete-data estimation or testing procedure on $X^{(\ell)}$, 
for $\ell =1, \ldots, m$, resulting in a set of $m$ quantities. 
The so-called MI inference combines them to obtain a single answer. 
Note that the setting of MI is such that the user is unable or unwilling to carry out the test based directly on the observed data $X_{\obs}$.

\subsection{MI Wald Test and Fraction of Missing Information}
Let $d^{(\ell)}_{\wt} = d_{\wt}(\widehat{\theta}^{(\ell)},U^{(\ell)})$,
$\widehat{\theta}^{(\ell)} = \widehat{\theta}(X^{(\ell)})$,
and $U^{(\ell)} = U(X^{(\ell)})$
be the imputed counterparts of $d_{\wt}(\widehat{\theta},U)$, $\widehat{\theta}$, and $U$,
respectively, for each $\ell$.
In addition, let
\begin{eqnarray}\label{eq:ddd}
	\overline{d}_{\wt} = \frac{1}{m} \sum_{\ell=1}^m d^{(\ell)}_{\wt} ,
	\qquad
	\overline{\theta} = \frac{1}{m} \sum_{\ell=1}^{m}\widehat{\theta}^{(\ell)} ,
	\qquad
	\overline{U} = \frac{1}{m} \sum_{\ell=1}^m U^{(\ell)} .
\end{eqnarray}
Under congeniality \citep{meng1994}, one can show that asymptotically \citep*{rubinSchenker1986} $\Var(\overline{\theta})$ can be consistently estimated by 
\begin{eqnarray}
T = \overline{U} + (1+1/m)B,\quad 
\text{where}\quad  B = \frac{1}{m-1}\sum_{\ell=1}^m (\widehat{\theta}^{(\ell)}-\overline{\theta})(\widehat{\theta}^{(\ell)}-\overline{\theta})^\T
\end{eqnarray}
is known as the \textit{between-imputation variance}, in contrast to $\bar U$ in (\ref{eq:ddd}), which measures the \textit{within-imputation variance}.  
Intriguingly, $2T$ serves as a universal (estimated) upper bound of $\Var(\overline{\theta})$ under uncongeniality \citep{XXMeng2017}. 
Under regularity conditions, we have that, as $m,n\rightarrow\infty$, 
\begin{eqnarray*}\label{eqt:lowerOrderVarUTB}
	n(\overline{U} - \mathcal{U}_\theta) \inP \mathbf{0} ,\qquad 
	n(T - \mathcal{T}_\theta) \inP \mathbf{0} ,\qquad 
	n(B - \mathcal{B}_\theta) \inP \mathbf{0},
\end{eqnarray*}
for some deterministic matrices $\mathcal{U}_\theta$, $\mathcal{T}_\theta$, and 
$\mathcal{B}_\theta= \mathcal{T}_\theta-\mathcal{U}_\theta$, 
where 
$\mathbf{0}$ denotes a matrix of zeros, and the subscript $\theta$ highlights that these matrices are 
for estimating $\theta$, 
because there are also corresponding $\mathcal{T}_\psi$, $\mathcal{B}_\psi$, and $\mathcal{U}_\psi$ for the entire parameter $\psi$. 
Similar to $\overline{U}$, $T$, and $B$,
we define $\overline{U}_{\psi}$, $T_{\psi}$, and $B_{\psi}$ for the parameter $\psi$.
If $\widehat{\theta}_{\com}$ and $\widehat{\theta}_{\obs}$ are the MLEs of $\theta$ 
based on $X_{\com}$ and $X_{\obs}$ (under congeniality), respectively, 
then $\mathcal{U}_{\theta} \bumpeq \Var(\widehat{\theta}_{\com})$
and $\mathcal{T}_{\theta} \bumpeq \Var(\widehat{\theta}_{\obs})$ 
as $n\rightarrow\infty$,
where $A_n \bumpeq B_n$ means that $A_n-B_n=o_p\{\min(A_n,B_n)\}$.
Note that the relation $A_n\bumpeq B_n$ means that the difference between $A_n$ and $B_n$
is of a smaller order than $A_n$ or $B_n$,
when both $A_n\geq 0$ and $B_n\geq 0$ approach zero. 
This notation (or its variants) is also used in, for example, 
\cite{mengRubin92}, \cite{li91JASA}, and \cite{KimShao2013}.

The straightforward MI Wald test  
$D_{\wt}(T) = d_{\wt}(\overline{\theta},T)/k$
is not practical because 
$T$ is singular when $m<k$ (usually $3\leq m \leq 10$). 
Even when it is not singular, 
it is usually not a very stable estimator of $\mathcal{T}_{\theta}$
because $m$ is small. 
To circumvent this problem, \cite{rubin1978} adopted the following assumption of 
an EFMI.

\begin{assumption}[EFMI of $\theta$]\label{ass:EFMI}
There is $\mathcal{r}\geq 0$ such that 
$\mathcal{T}_\theta = (1+\mathcal{r}) \mathcal{U}_\theta$.
\end{assumption}

EFMI is a strong assumption, implying that the missing data have caused an equal loss of information for estimating every component of $\theta$. However, as we shall see shortly, adopting 
this assumption \textit{for the purpose of hypothesis testing} is essentially the same as summarizing 
the impact of  (at least) $k$ nuisance parameters due to FMI by a single nuisance parameter, this is, 
the average FMI across different components. How well this reduction strategy works has a greater effect on the power of the test than on its validity, as long as we can construct an approximate null distribution that is more robust to the EFMI assumption. 
The issue of power turns out to be a rather tricky one, because without the reduction strategy, 
we also lose power when $m/k$ is small or even modest. 
This is because we simply do not have enough degrees of freedom to estimate all the nuisance parameters well or at all. We illustrate this point in Section~\ref{sec:MCexpUFMI}. 
(To clarify some confusion in literature, $\mathcal{r}$ in Assumption~\ref{ass:EFMI} is the \textit{odds of the missing information}, not the FMI, which is
$\mathcal{f}=\mathcal{r}/(1+\mathcal{r})$.)
We also denote $\mathcal{r}_m = (1+1/m)\mathcal{r}$ as the 
finite-$m$ adjusted value of $\mathcal{r}$.

Under EFMI, \cite{rubin2004multiple} replaced $T$ by $(1+\widetilde{r}_{\wt}^{\prime})\overline{U}$, where 
\begin{eqnarray}
	\widetilde{r}_{\wt}^{\prime}
		= \frac{(m+1)}{k(m-1)}( \overline{d}_{\wt}^{\prime}- \widetilde{d}_{\wt}^{\prime}) 
	\label{eqt:def_of_rm}, \qquad 
	\overline{d}_{\wt}^{\prime}
		= \frac{1}{m}\sum_{\ell=1}^m d_{\wt}(\widehat{\theta}^{(\ell)},\overline{U}),
\end{eqnarray}
$\widetilde{d}_{\wt}^{\prime}
		=  d_{\wt}(\overline{\theta},\overline{U})$, and the prime ``$\prime$'' 
indicates that $\overline{U}$ is used instead of individual $\{U^{(\ell)}\}_{\ell=1}^m$.
Then, a simple MI Wald test statistic \citep{rubin2004multiple} is
\begin{eqnarray}\label{eqt:def_of_DmApprox}
	\widetilde{D}_{\wt}^{\prime}
		= \frac{\widetilde{d}_{\wt}^{\prime}}{k(1+\widetilde{r}_{\wt}^{\prime})} . 
\end{eqnarray}
The intuition behind (\ref{eqt:def_of_rm})--(\ref{eqt:def_of_DmApprox}) is important 
because it forms the building blocks for virtually all the subsequent developments. 
The ``obvious'' Wald statistic $\widetilde{d}_{\wt}^{\prime}/k$ is too large 
(compared to the usual $\chi^2_k/k$), because it fails to take into account the missing information. 
The $(1+\widetilde{r}_{\wt}^{\prime})$ factor attempts to correct this, 
with the amount of correction determined by the  
amount of between-imputation variance relative to the within-imputation variance. 
This relative amount can be estimated by contrasting 
the average of individual Wald statistics and the Wald statistic based on 
an average of individual estimates,
as in (\ref{eqt:def_of_rm}). 
Using the difference between the ``average of functions'' and the ``function of average,'' namely,
\begin{eqnarray}\label{eq:agga}
	{\rm Ave}\{G(x)\} - G({\rm Ave}\{x\}),
\end{eqnarray}
is a common practice, for example, $G(x)=x^2$ for variance; see \citet{meng2002}.  

Because the exact null distribution of $\widetilde{D}_{\wt}^{\prime}$ 
is intractable,
\cite{li91JASA} proposed 
approximating it by 
$F_{k,\widetilde{\df}(\widetilde{r}_{\wt}^{\prime},k)}$, 
the $F$ distribution with degrees of freedom $k$ and 
$\widetilde{\df}(\widetilde{r}_{\wt}^{\prime},k)$,
where, denoting $K_m = k(m-1)$,   
\begin{eqnarray}\label{eqt:def_w_r}
	\widetilde{\df}(\mathcal{r}_m,k) 
		= \left\{ \begin{array}{ll}
				4+(K_m-4)\{1+(1-2/K_m)/\mathcal{r}_m\}^2, & \text{if $K_m>4$;} \\
				(m-1)(1+1/\mathcal{r}_m)^2(k+1)/2, & \text{otherwise}.
			\end{array}\right. 
\end{eqnarray}
In (\ref{eqt:def_w_r}), $n$ is assumed to be sufficiently large
so that the asymptotic $\chi^2$ distribution 
in Property \ref{prop:asyEquiv_LRT_Wald} can be used. 
If $n$ is small, the small sample degree of freedom in \cite{BarnardRubin1999} should be used.

\subsection{The Current MI Likelihood Ratio Test and Its Defect}\label{subsec:defect}
Let $d_{\lrt}^{(\ell)} = d_{\lrt}(\widehat{\psi}_{0}^{(\ell)},\widehat{\psi}^{(\ell)}\mid X^{(\ell)})$,
$\widehat{\psi}_{0}^{(\ell)} = \widehat{\psi}_0(X^{(\ell)})$ and 
$\widehat{\psi}^{(\ell)} = \widehat{\psi}(X^{(\ell)})$
be the imputed counterparts of 
$d_{\lrt}(\widehat{\psi}_0,\widehat{\psi}\mid X)$, $\widehat{\psi}_0$ and $\widehat{\psi}$, respectively, for each $\ell$.
Define
\begin{eqnarray}\label{eqt:def_psiHat_bardd}
	\qquad
	\overline{d}_{\lrt} = \frac{1}{m}\sum_{\ell=1}^m d_{\lrt}^{(\ell)}, \qquad
	\overline{\psi}_{0} = \frac{1}{m}\sum_{\ell=1}^m\widehat{\psi}_{0}^{(\ell)}, \qquad
	\overline{\psi} = \frac{1}{m}\sum_{\ell=1}^m\widehat{\psi}^{(\ell)}.
\end{eqnarray}
Similar to $\widetilde{r}_{\wt}^{\prime}$,
\cite{mengRubin92} proposed estimating $\mathcal{r}_m$ by 
\begin{eqnarray}\label{eqt:rL}
	\qquad
	\widetilde{r}_{\lrt} 
		= \frac{m+1}{k(m-1)} ( \overline{d}_{\lrt} - \widetilde{d}_{\lrt} ), \quad {\rm where} \quad
	\widetilde{d}_{\lrt} 
		= \frac{1}{m}\sum_{\ell=1}^m d_{\lrt}(\overline{\psi}_{0} ,\overline{\psi}\mid X^{(\ell)}),
\end{eqnarray}
and hence it is again in the form of (\ref{eq:agga}). The computation of $\widetilde{r}_{\lrt}$ requires 
that users have access to 
(i) a subroutine for $(X,\psi_0,\psi)\mapsto d_{\lrt}(\psi_0,\psi\mid X)$, and 
(ii) the estimates $\widehat{\psi}_{0}^{(\ell)}$ and $\widehat{\psi}^{(\ell)}$, 
rather than the matrices $\overline{U}$ and $B$.
Therefore, computing $\widetilde{r}_{\lrt}$ is easier than 
computing $\widetilde{r}_{\wt}^{\prime}$.
The resulting MI LRT is 
\begin{equation}\label{eq:mrlrt}
	\widetilde{D}_{\lrt} = \frac{\widetilde{d}_{\lrt}}{k(1+\widetilde{r}_{\lrt})} ,
\end{equation}
the null distribution of which can be approximated by $F_{k,\widetilde{\df}(\widetilde{r}_{\lrt},k)}$.
Its main theoretical justification (and motivation) 
is the asymptotic equivalence between the complete-data Wald test statistic and the LRT statistic \textit{under the null}, as stated in Property~\ref{prop:asyEquiv_LRT_Wald}.  
This equivalence permitted the replacement of $\overline{d}_{\wt}^{\prime}$ and $\widetilde{d}_{\wt}^{\prime}$ in (\ref{eqt:def_of_rm}) by $\overline{d}_{\lrt}$ and $\widetilde{d}_{\lrt}$, respectively, 
in (\ref{eqt:rL}). However, this is also where the problems lie.

First, with finite samples, $0\leq \widetilde{d}_{\lrt} \leq \overline{d}_{\lrt}$ is not guaranteed;
consequently, nor 
is $\widetilde{D}_{\lrt} \geq 0$ or $\widetilde{r}_{\lrt}\geq 0$. 
Because $\widetilde{D}_{\lrt}$ is referred to as an $F$ distribution
and $\widetilde{r}_{\lrt}$ estimates $\mathcal{r}_m\geq 0$,
clearly, negative values of $\widetilde{D}_{\lrt}$ or $\widetilde{r}_{\lrt}$ will cause trouble. 
Second, 
the MI LRT statistic $\widetilde{D}_{\lrt}$ is not invariant to 
re-parameterization of $\psi$, 
although invariance is a natural property of the standard LRT; see, for example, 
\citet{DagenaisDufour91}. 
This invariance principle 
is an appealing property because 
it requires that 
problems with the same formal structure should produce the same statistical results; 
see Chapter 6 of \citet{berger85} and Chapter 3.2 of \citet{LehmannCasella98}.
Formally, we say that $\varphi = g(\psi)$ is a re-parametrization of $\psi$
if $g$ is a bijective map. 
The classical LRT statistic is invariant to re-parametrization because
\[
	d_{\lrt}(\widehat{\psi}_{0},\widehat{\psi}\mid X)
		= d_{\lrt}(g^{-1}(\widehat{\varphi}_{0}),g^{-1}(\widehat{\varphi})\mid X),
\]
where $\widehat{\varphi}_{0}$ and $\widehat{\varphi}$ 
are the constrained and unconstrained MLEs, respectively, of $\varphi$ based on $X$.
However, the MI (pooled) LRT statistic $\widetilde{d}_{\lrt}$
no longer has this property because 
\[
	\sum_{\ell=1}^m d_{\lrt}(\overline{\psi}_{0},\overline{\psi}\mid X^{(\ell)})
		\not{=}\sum_{\ell=1}^m d_{\lrt}(g^{-1}(\overline{\varphi}_{0}),g^{-1}(\overline{\varphi})\mid X^{(\ell)}) 
\]
in general,
where $\widehat{\varphi}_{0}^{(\ell)}$ and $\widehat{\varphi}^{(\ell)}$
are the constrained and unconstrained MLEs, respectively, of $\varphi$ based on $X^{(\ell)}$, 
and 
$\overline{\varphi}_{0} = m^{-1} \sum_{\ell=1}^m \widehat{\varphi}_{0}^{(\ell)}$
and $\overline{\varphi} = m^{-1} \sum_{\ell=1}^m \widehat{\varphi}^{(\ell)}$.
Section \ref{sec:example} shows how the 
MI LRT results vary dramatically with parametrizations in finite samples.

Third, the estimator $\widetilde{r}_{\lrt}$ involves 
the estimators of $\psi$ under $H_0$, this is, $\widehat{\psi}_0^{(\ell)}$ and $\overline{\psi}_0$.
When $H_0$ fails, they may be inconsistent for $\psi$.
Thus, $\widetilde{r}_{\lrt}$ is no longer consistent for $\mathcal{r}_m$. 
A serious consequence is that the power of the test statistic $\widetilde{D}_{\lrt}$ is not guaranteed to monotonically increase as $H_1$ moves away from $H_0$. Indeed, 
our simulations (see Section~\ref{sec:comparison}) show that 
under certain parametrizations, 
the power may nearly vanish for obviously false $H_0$. 
Fourth, computing $\widetilde{d}_{\lrt}$ in (\ref{eqt:rL})
requires that users have access to  $\widetilde{\mathcal{D}}_{\lrt}$, 
a function of both data and parameters. 
However, in most software,  
the available function is $\mathcal{D}_{\lrt}$, a function of data only; that is,
\begin{eqnarray}
	\widetilde{\mathcal{D}}_{\lrt}:(X,\psi_0,\psi)\mapsto d_{\lrt}(\psi_0,\psi\mid X), \quad
	\mathcal{D}_{\lrt}:X\mapsto d_{\lrt}(\widehat{\psi}_0(X),\widehat{\psi}(X)\mid X). \label{eqt:computeCode}
\end{eqnarray}
It is not always feasible for users to  
write themselves a subroutine for computing 
$\widetilde{\mathcal{D}}_{\lrt}$.

In short, four problems need to be resolved: (i) the lack of non-negativity, 
(ii) the lack of invariance, (iii) the lack of consistency and power, 
and (iv) the lack of a feasible algorithm. 
Problems (i)--(iii) are resolved in Section~\ref{sec:improved_MI_LRT}; 
(iv) is resolved in Section~\ref{sec:comp}.

\section{Improved MI Likelihood Ratio Tests}\label{sec:improved_MI_LRT}
\kern -0.1in
\subsection{Invariant Combining Rule and Estimator of $\mathcal{r}_m$}\label{sec:correctInv} 
To derive a parametrization-invariant MI LRT, we replace 
$\widetilde{d}_{\lrt}$ by 
an asymptotically equivalent version  
that behaves like a standard LRT statistic.  
Let
\begin{equation}\label{eqt:jointDensity}
	\overline{\loglik}(\psi)
		= \frac{1}{m}\sum_{\ell=1}^m \loglik^{(\ell)}(\psi),
	\qquad \text{where} \quad
	\loglik^{(\ell)}(\psi)	
		= \log f(X^{(\ell)}\mid \psi) .
\end{equation}
Here, $\overline{\loglik}(\psi)$ is \textit{not} a real log-likelihood, 
because it does not properly model the completed data sets: $\mathbb{X}=\{X^1, \ldots,  X^m\}$ (e.g., all $X^\ell$ share the same $X_{\obs}$).  Nevertheless, $\overline{\loglik}(\psi)$ can be treated as a log-likelihood for computational purposes. In particular, we can maximize it to obtain 
\begin{equation}\label{eqt:MLEpsi_wrt_combinedLike}
	\widehat{\psi}_{0}^{*} = \widehat{\psi}_0^{*}(\mathbb{X})= \argmax_{\psi\in\Psi\;:\;\theta(\psi)=\theta_0} \overline{\loglik}(\psi) ,
	\qquad 
	\widehat{\psi}^{*} =\widehat{\psi}^{*}(\mathbb{X})= \argmax_{\psi\in\Psi} \overline{\loglik}(\psi).
\end{equation}
The corresponding log-likelihood ratio test statistic
is given by 
\begin{eqnarray}\label{eqt:bardd_combinedLike}
	\widehat{d}_{\lrt}=2 \left\{ \overline{\loglik}(\widehat\psi^*)-\overline{\loglik}(\widehat\psi_0^*) \right\}=
    \frac{1}{m} \sum_{\ell=1}^m d_{\lrt}(\widehat\psi_{0}^*,\widehat\psi^* \mid X^{(\ell)}).
\end{eqnarray}
Thus, in contrast to 
$\widetilde{d}_{\lrt}$ of (\ref{eqt:rL}),
$\widehat{d}_{\lrt}$ aggregates MI data sets by 
averaging the MI LRT functions, as in (\ref{eqt:jointDensity}), rather than 
averaging the MI test statistics and moments, as in (\ref{eqt:def_psiHat_bardd}).
Although $\sqrt{n}(\widehat{\psi}_{0}^{*} - \overline{\psi}_{0}) \inP \textbf{0}$ 
and $\sqrt{n}(\widehat{\psi}^{*} - \overline{\psi}) \inP \textbf{0}$
as $n\rightarrow\infty$ for each $m$, 
only $\widehat{d}_{\lrt}$, not $\widetilde{d}_{\lrt}$,
is guaranteed to be non-negative and invariant to parametrization of $\psi$ for all $m,n$.
Indeed, the likelihood principle guides us to consider averaging individual log-likelihoods 
rather than individual MLEs, because the former has a much better chance 
of capturing the functional features of the real log-likelihood than any of their (local) maximizers can. 

To derive the properties of $\widehat{d}_{\lrt}$, we need the usual regularity conditions on the MLE and MI.  

\begin{assumption}\label{ass:likelihood}
The sampling model $f(X\mid\psi)$ satisfies the following: 
\begin{enumerate}[noitemsep]
	\item[(a)] The map $\psi\mapsto\underline{\loglik}(\psi)= n^{-1}\log f(X\mid \psi)$ 
					is twice continuously differentiable. 
	\item[(b)] The complete-data MLE $\widehat{\psi}(X)$ is the unique solution of 
					$\partial \underline{L}(\psi)/\partial \psi = \bm{0}$.
	\item[(c)] Let $\underline{I}(\psi) = -\partial^2 \underline{\loglik}(\psi) / \partial \psi \partial \psi^{\T}$; 
					then, for each $\psi$, there exists a positive-definite matrix 
					$\underline{\mathcal{I}}(\psi) = \mathcal{U}_{\psi}^{-1}$ such that 
					$\underline{I}(\psi) \inP \underline{\mathcal{I}}(\psi)$
					as $n\rightarrow\infty$. 
	\item[(d)] The observed-data MLE $\widehat{\psi}_{\obs}$ of $\psi$ obeys
					\begin{equation}\label{eqt:psiObsHatNormal}
						\left[ \mathcal{T}_{\psi}^{-1/2}\left(\widehat{\psi}_{\obs} 
								-\psi \right) \bigg| \psi  \right]  \inD \mathcal{N}_h(\bm{0},I_h)
					\end{equation}
					as $n\rightarrow\infty$,
					where $I_h$ is the $h\times h$ identity matrix.
\end{enumerate}
\end{assumption}

\begin{assumption}\label{ass:properImpModel}
The imputation model is proper \citep*{rubin2004multiple}:
\begin{gather}
	\left[\mathcal{B}_{\psi}^{-1/2}\left( \widehat{\psi}^{(\ell)} - \widehat{\psi}_{\obs} \right) \bigg| X_{\obs} \right]
	\inD \mathcal{N}_h(\bm{0} ,I_h), \label{eqt:psiellHatNormal} \\
	\left[ \mathcal{T}_{\psi}^{-1}\left(U^{(\ell)}_{\psi} - \mathcal{U}_{\psi}\right) \bigg| X_{\obs} \right]\inP \bm{0} , \qquad
	\left[ \mathcal{T}_{\psi}^{-1}\left(B_{\psi} - \mathcal{B}_{\psi}\right) \bigg| X_{\obs} \right]\inP \bm{0} 
	\label{eqt:lowerOrderVarAss}
\end{gather}
independently for each $\ell$, as $n\rightarrow\infty$, 
provided that $\mathcal{B}_{\psi}^{-1}$ is well defined. 
\end{assumption}

Assumption \ref{ass:likelihood} holds
under the usual regularity conditions that guarantee the normality and consistency of MLEs.
When $X^{(1)}_{\mis},\ldots,X^{(m)}_{\mis}$
are drawn independently from a (correctly 
specified) posterior predictive distribution $f(X_{\mis}\mid X_{\obs})$,
Assumption \ref{ass:properImpModel} is typically satisfied.
Clearly,
we can replace $\psi$ by its sub-vector $\theta$ in Assumptions \ref{ass:likelihood} 
and \ref{ass:properImpModel}. These $\theta$-version assumptions are sufficient to guarantee the validity of
Theorem~\ref{thm:robEstofR} and Corollary~\ref{coro:finitenessOfr}.   
For simplicity, Assumption~\ref{ass:EFMI}, the
$\theta$-version of Assumptions~\ref{ass:likelihood} and \ref{ass:properImpModel}, 
and the conditions that guarantees  
Property \ref{prop:asyEquiv_LRT_Wald}
are collectively written as $\RC_\theta$ (RC denotes ``regularity conditions''),
which are commonly assumed for MI inference.

\begin{theorem}\label{thm:AsyEq_dBar_LL}
Assume $\RC_\theta$.
Under $H_0$, we have 
(i) $\widehat{d}_{\lrt}\geq 0$ for all $m,n$; 
(ii) $\widehat{d}_{\lrt}$ is invariant to parametrization of $\psi$ for all $m,n$; and 
(iii) $\widehat{d}_{\lrt}\bumpeq \widetilde{d}_{\lrt}$
as $n\rightarrow\infty$ for each $m$.
\end{theorem}

Consequently, an improved combining rule is defined as
\begin{equation}\label{eqt:final_proposal}
	\widehat{D}_{\lrt}(\mathcal{r}_m) = \frac{\widehat{d}_{\lrt}}{k(1+\mathcal{r}_m)},
\end{equation}
for a given value of $\mathcal{r}_m$.
The forms of (\ref{eqt:def_of_DmApprox}) and (\ref{eq:mrlrt}) follow. 
Using $\widehat{d}_{\lrt}$ in (\ref{eqt:bardd_combinedLike}),
we can modify $\widetilde{r}_{\lrt}$ in (\ref{eqt:rL}) to provide 
a potentially better estimator: 
\begin{equation}\label{eqt:rLLpos}
	\widehat{r}_{\lrt} = \frac{m+1}{k(m-1)} ( \overline{d}_{\lrt} - \widehat{d}_{\lrt} ).
\end{equation}
Although $\widehat{d}_{\lrt}\geq 0$ is guaranteed by our construction, $\widehat{r}_{\lrt}\geq 0$ does not hold in general for a finite $m$. However, it is guaranteed in the following situation. 

\begin{proposition}\label{prop:condForPositiveR}
Write $\psi=(\theta^{\T}, \eta^{\T})^{\T}$, 
where $\eta$ represents a nuisance parameter that is distinct from $\theta$. If there exist functions $\loglik_{\dag}$ and $\loglik_{\ddag}$ 
such that, for all $X$, 
the log-likelihood function 
$\loglik(\psi\mid X) = \log f(X\mid\psi)$
is of the form $\loglik(\psi\mid X)= \loglik_{\dag}(\theta \mid X) + \loglik_{\ddag}(\eta\mid X)$,
then $\widehat{r}_{\lrt}\geq 0$ for all $m,n$.
\end{proposition}

The condition in Proposition \ref{prop:condForPositiveR}
means that the likelihood function of $\psi$
is separable, which ensures that 
the profile likelihood estimator of $\eta$ given $\theta$, this is,
$\widehat{\eta}_{\theta} = \argmax_{\eta} \loglik(\theta,\eta\mid X)$,
is free of $\theta$.  
Clearly, in the absence of the nuisance parameter $\eta$, 
the separation condition holds trivially. 
More generally, we have the following. 

\begin{corollary}\label{coro:finitenessOfr}
Assume $\RC_\theta$.
We have 
(i) under $H_0$, $\widehat{r}_{\lrt} \inP \mathcal{r}$ as $m,n\rightarrow\infty$; and 
(ii) under $H_1$,
$\widehat{r}_{\lrt} \inP \mathcal{r}_{0}$ as $m,n\rightarrow\infty$, 
where $\mathcal{r}_{0} \geq 0$ is some finite value depending on $\theta_0$ and the true value of $\theta$.
\end{corollary}

Corollary \ref{coro:finitenessOfr} ensures that, 
under $H_0$, $\widehat{r}_{\lrt}$ is non-negative asymptotically
and converges in probability to the true $\mathcal{r}$.
However, it also reveals another fundamental defect of $\widehat{r}_{\lrt}$:
under $H_1$, the limit $\mathcal{r}_{0}$ may not equal $\mathcal{r}$,
a problem we address in 
Section~\ref{sec:rob}. 
Fortunately, 
because $\widehat{d}_{\lrt} \inP \infty$ under $H_1$, 
the LRT statistic $\widehat{D}_{\lrt}(\widehat{r}_{\lrt})$ is still powerful, albeit the power may be reduced. 
Similarly,  $\widetilde{r}_{\lrt}$ of (\ref{eqt:rL})
has the same asymptotic properties and defects, but $\widehat{r}_{\lrt}$ behaves more nicely than $\widetilde{r}_{\lrt}$ for finite $m$.
This hinges closely on the high sensitivity of 
$\widetilde{r}_{\lrt}$ to the parametrization of $\psi$;  
for example, $\widetilde{r}_{\lrt}$ may become more negative as $H_1$ moves away from $H_0$; 
see Section~\ref{sec:eg_lrt_normal}.

Whereas we can fix the occasional negativeness of $\widehat{r}_{\lrt}$ by 
using $\widehat{r}_{\lrt}^+=\max( 0, \widehat{r}_{\lrt})$, such an ad hoc fix misses the opportunity to improve upon $\widehat{r}_{\lrt}$, and indeed it cannot fix the inconsistency of $\widehat{r}_{\lrt}$ under $H_1$.

\subsection{A Consistent and Non-Negative Estimator of $\mathcal{r}_m$}\label{sec:rob}
Proposition \ref{prop:condForPositiveR} already hinted that 
the source of the negativity and inconsistency of $\widehat{r}_{\lrt}$
is related to the existence of the nuisance parameter $\eta$. 
By definition, $\overline{d}_{\lrt}$ and $\widehat{d}_{\lrt}$ depend 
on the specification of $\theta_0$. 
In general, the effect of $\theta_0$ may not be cancelled out by their difference 
$\overline{d}_{\lrt}-\widehat{d}_{\lrt}$, 
unless a certain type of orthogonality assumption is made on $\eta$ and $\theta$; see 
Proposition \ref{prop:condForPositiveR} for an example. 
Consequently, the validity of the estimator $\widehat{r}_{\lrt}$ depends on the correctness of $H_0$.
A more elaborate discussion can be found in Appendix \ref{subsec:nusiancePara}.
In order to principally resolve the aforementioned problem, 
we need to eliminate the dependence on $\theta_0$ in our estimator for the odds of missing information, $\mathcal{r}_m$.
We achieve this goal by estimating these odds for the entire $\psi$, resulting 
in the following estimator for  $\mathcal{r}_m$:
\begin{gather}
	\widehat{r}_{\lrt}^{\rob}
		= \frac{m+1}{h(m-1)} (\overline{\delta}_{\lrt} - \widehat{\delta}_{\lrt}), 
	\quad \text{where} \label{eqt:def_hatr_rob} \\  
	\overline{\delta}_{\lrt} 
		= 2 \overline{L}(\widehat{\psi}^{(1)}, \ldots, \widehat{\psi}^{(m)}),
	\qquad 
	\widehat{\delta}_{\lrt} 
		= 2 \overline{L}(\widehat{\psi}^*, \ldots, \widehat{\psi}^*),
		\label{eqt:def_deltabarhatL}
\end{gather}
and $h$ is the dimension of $\psi$. 
In (\ref{eqt:def_deltabarhatL}), the rhombus ``${\rob}$'' symbolizes a robust estimator. It is robust because it is consistent under either $H_0$ or $H_1$, 
as long as we are willing to impose the EFMI assumption on $\psi$,
this is, Assumption \ref{ass:EFMIpsi}.
This expansion from $\theta$ to $\psi$ is inevitable because the  LRT must handle the entire $\psi$, not just $\theta$.
The collection of Assumptions~\ref{ass:likelihood}--\ref{ass:EFMIpsi} are referred to as $\RC_\psi$.

\begin{assumption}[EFMI of $\psi$]\label{ass:EFMIpsi}
There is $\mathcal{r}\geq 0$ such that 
$\mathcal{T}_\psi =(1+\mathcal{r}) \mathcal{U}_{\psi}$.
\end{assumption}

\begin{theorem}\label{thm:robEstofR}
Assume $\RC_\psi$. 
For any value of $\psi$, we have 
(i) $\widehat{r}_{\lrt}^{\rob}\geq 0$ for all $m,n$; 
(ii) $\widehat{r}_{\lrt}^{\rob}$ is invariant to parametrization of $\psi$ for all $m,n$; and 
(iii) $\widehat{r}_{\lrt}^{\rob}  \inP \mathcal{r} $ as $m,n\rightarrow\infty$, 
where $\mathcal{r}$ is given in Assumption~\ref{ass:EFMIpsi}.
\end{theorem}

With the improved combining rule $\widehat{D}_{\lrt}(\mathcal{r}_m)$
of (\ref{eqt:final_proposal}) and improved estimators for $\mathcal{r}_m$,
we are ready to propose two MI LRT statistics:
\begin{equation}\label{eqt:newstat}
	\widehat{D}^{+}_{\lrt}=\widehat{D}_{\lrt}(\widehat{r}^{+}_{\lrt})
	\qquad\text{and}\qquad
	\widehat{D}^{\rob}_{\lrt}=\widehat{D}_{\lrt}(\widehat{r}^{\rob}_{\lrt}).
\end{equation}
For comparison, we also study the test statistic 
$\widehat{D}_{\lrt} =\widehat{D}_{\lrt}(\widehat{r}_{\lrt})$.

\subsection{Reference Null Distributions}\label{sec:refNullDist}
The estimators $\widehat{r}_{\lrt}^+$ and $\widetilde{r}_{\lrt}$
have the same functional form asymptotically ($n\rightarrow\infty$). 
Hence, they have the same asymptotic distribution.

\begin{lemma}\label{thm:distHatrPos}
Suppose $\RC_\theta$ and $m>1$.
Under $H_0$, we have, jointly, 
\begin{equation}\label{eqt:asyDistrPos}
	 \frac{\widehat{r}_{\lrt}^+}{\mathcal{r}_m} \inD M_2
	\qquad \text{and} \qquad
	\widehat{D}_{\lrt}^+ \inD \frac{\left( 1+\mathcal{r}_m \right) M_1 }{ 1+\mathcal{r}_m M_2} 
\end{equation}
as $n\rightarrow\infty$, where 
$M_1\sim \chi^2_k/k$ and $M_2\sim\chi^2_{k(m-1)}/\{k(m-1)\}$ are independent.
\end{lemma}
\noindent
Consequently,  
$\widehat{D}^{+}_{\lrt} = \widehat{D}_{\lrt}(\widehat{r}^{+}_{\lrt})$
approximately follows $F_{k,\widetilde{\df}(\widehat{r}^{+}_{\lrt},k)}$ under $H_0$,
but a better approximation is provided shortly. 
For the other proposal, 
although $\widehat{r}_{\lrt}^+ - \widehat{r}_{\lrt}^{\rob} \inP 0$
as $n\rightarrow\infty$ under $H_0$, 
their non-degenerated limiting distributions  
are different
because $\widehat{r}_{\lrt}^{\rob}$ and $\widehat{r}_{\lrt}^+$
rely on an average FMI in $\psi$ and $\theta$, respectively.

\begin{theorem}\label{thm:distHatrRob}
Suppose $\RC_\psi$ and $m>1$. Then, for any value of $\psi$, 
\begin{equation}\label{eqt:asyDistrRob}
	  \frac{ \widehat{r}_{\lrt}^{\rob}}{\mathcal{r}_m} \inD  M_3 \sim \frac{\chi^2_{h(m-1)}}{h(m-1)} 
\end{equation}
as $n\rightarrow\infty$,
where $M_3$ is independent of the $M_1$ defined in (\ref{eqt:asyDistrPos}).
\end{theorem}

Theorem \ref{thm:distHatrRob} implies that, 
if $n$ can be regarded as infinity
and $\widehat{r}_{\lrt}^{\rob}$ is uniformly integrable in $\mathcal{L}^2$, 
then  
$
	\Bias(\widehat{r}_{\lrt}^{\rob}) 
		\!\!=\!\!\E(\widehat{r}_{\lrt}^{\rob})\!\!-\!\!\mathcal{r}_m\!\! 
		=\!\!0 
$
and
$
	\Var(\widehat{r}_{\lrt}^{\rob})\!\!=\!\! 2\mathcal{r}_m^2/\{h(m-1)\} = O(m^{-1}) 
$
as $m\rightarrow\infty$.
Hence,
$\widehat{r}_{\lrt}^{\rob}$ is a $\sqrt{m}$-consistent estimator of $\mathcal{r}$ in $\mathcal{L}^2$.
Moreover, for each $m>1$ and as $n\rightarrow\infty$, 
we have ${\Bias(\widehat{r}_{\lrt}^{+})}/{\Bias(\widehat{r}_{\lrt}^{\rob})} \rightarrow 1$
and 
${\Var(\widehat{r}_{\lrt}^{+})}/{\Var(\widehat{r}_{\lrt}^{\rob})}
		\rightarrow {h}/{k} \geq 1$,
which imply that $\widehat{r}_{\lrt}^{\rob}$ is no less  
efficient than $\widehat{r}_{\lrt}^{+}$ when $\RC_\psi$ holds. 
This is not surprising because of 
the extra information 
brought in by the stronger Assumption \ref{ass:EFMIpsi}.
Result (\ref{eqt:asyDistrRob}) also gives us  
the exact (i.e., for any $m>1$, but assuming $n\rightarrow\infty$) 
reference null distribution 
of $\widehat{D}^{\rob}_{\lrt}$, as given below. 

\begin{theorem}\label{thm:exactNullDistRrob}
Assume $\RC_\psi$ and $m>1$. Under $H_0$, we have
\begin{equation}\label{eqt:exactNullDistRrob}
	\widehat{D}_{\lrt}^{\rob}
		\inD \frac{\left( 1 + \mathcal{r}_m \right) M_1 }{ 1+\mathcal{r}_m M_3} \equiv D 
\end{equation}
as $n\rightarrow\infty$,
where $M_1\sim \chi^2_k/k$ and $M_3\sim\chi^2_{h(m-1)}/\{h(m-1)\}$ are independent. 
\end{theorem}

The impact of the nuisance parameter $\mathcal{r}_m$ on $D$  
diminishes with $m$ because 
$\widehat{D}_{\lrt}^{\rob}$ and $\widehat{D}_{\lrt}^{+}$
converge in distribution to $M_1 = \chi^2_k/k$ as $m,n\rightarrow\infty$. 
Because $M_3\inP 1$ faster than $M_2\inP 1$, $\widehat{D}_{\lrt}^{\rob}$ is expected to be more robust to $\mathcal{r}_m$. 
Nevertheless, $m$ typically is small in practice (e.g., $m\le 10$), so we cannot ignore the impact of $\mathcal{r}_m$. 
This issue has been largely dealt with in the literature by seeking an $F_{k, {\rm df}}$ distribution to approximate $D$, as in \cite{li91JASA}. However, directly adopting their $\widetilde{\rm df}$ of (\ref{eqt:def_w_r}) leads to a poorer approximation for our purposes; see below. 
A better approximation is to match the first two moments of the denominator of (\ref{eqt:exactNullDistRrob}), 
$1+\mathcal{r}_m M_3$, with that of a scaled $\chi^2$: $a \chi^2_b/b$. 
This yields $a = 1+\mathcal{r}_m$ and $b=(1+\mathcal{r}_m^{-1})^2h(m-1)$, and the approximated $F_{k, \widehat{\df}(\mathcal{r}_m,h)}$,
where 
\begin{equation}\label{eqt:df2New}
	\widehat{\df}(\mathcal{r}_m ,h ) 
		= \left\{ \frac{1+\mathcal{r}_m}{\mathcal{r}_m} \right\}^2 h(m-1)
		= \frac{h(m-1)}{\mathcal{f}_m^2},
\end{equation}
which is appealing because it simply inflates the denominator degrees of freedom $h(m-1)$ by dividing it by 
the square of the finite-$m$ corrected FMI $\mathcal{f}_m = \mathcal{r}_m/(1+\mathcal{r}_m)$. 
The less missing information, the closer $F_{k, \widehat{\df}(\mathcal{r}_m ,h )}$ is to $\chi^2_k/k$, the usual large-$n$ $\chi^2$ test; as mentioned earlier, for small $n$, see \cite{BarnardRubin1999}.

\begin{figure}[t!]
\begin{center}
\includegraphics[width=.95\textwidth]{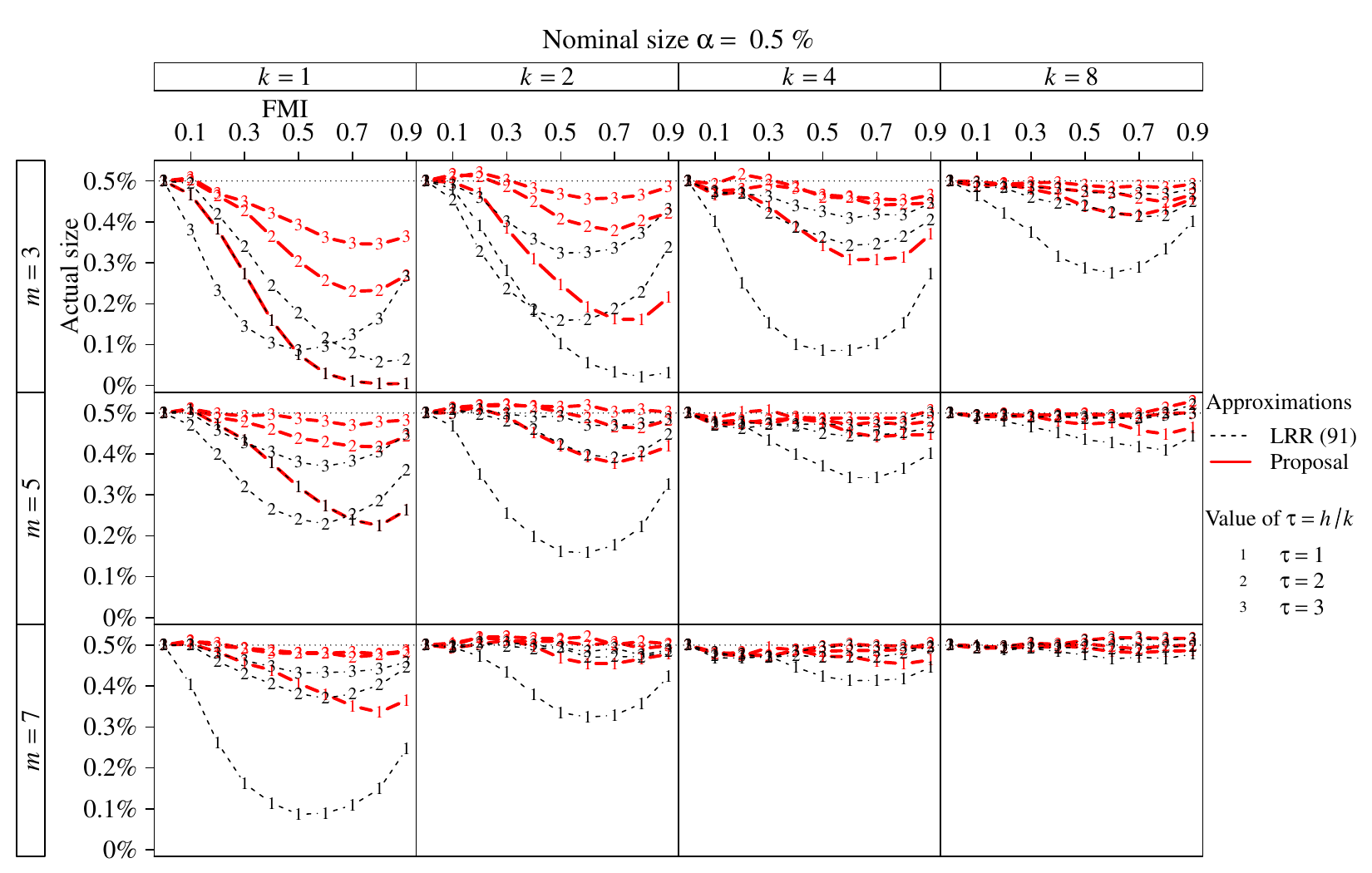}
\vspace{-0.4cm}
\caption{\small 
The performance of two approximated null distributions
when the nominal size is $\alpha=0.5\%$. 
The vertical axis denotes $\widehat{\alpha}$ or $\widetilde{\alpha}$, 
and the horizontal axis denotes the value of $\mathcal{f}_m$.
The number attached to each line denotes the value of $\tau = h/k$.
}
\label{fig:nullDist05pc}
\end{center} 
\end{figure}

To compare the performance of $F_{k,\widehat{\df}(\mathcal{r}_m,h)}$ in (\ref{eqt:df2New}) 
with the existing best approximation $F_{k,\widetilde{\df}(\mathcal{r}_m,h)}$,
as approximations to the limiting distribution of $D$ given in (\ref{eqt:exactNullDistRrob}), 
we compute via simulations
\begin{equation*}\label{eqt:alphaHatTilde}
	\widetilde{\alpha} = \pr\left\{ D > F^{-1}_{k, \widetilde{\df}(\mathcal{r}_m,h)}(1-\alpha)\right\} 
	\quad \text{and} \quad
	\widehat{\alpha}  = \pr\left\{D > F^{-1}_{k, \widehat{\df}(\mathcal{r}_m,h)}(1-\alpha)\right\},
\end{equation*}
where $F^{-1}_{k,\df}(q)$ denotes the $q$-quantile of $F_{k,\df}$.
Note that the experiments assess solely the performance of the finite-$m$ approximation instead of the 
performance of the large-$n$ $\chi^2$-approximation of the asymptotic LRT statistics.
We draw $N=2^{18}$ independent copies $D$ for each of the following possible combinations: $m\in\{3,5,7\}$, $k\in\{1,2,4,8\}$, $\tau = h/k\in\{1,2,3\}$,
$\mathcal{f}_m \in\{ 0, 0.1, \ldots, 0.9 \}$, and following the recommendation of \citet{benjamin2017redefine}, we use both $\alpha\in\{0.5\%, 5\%\}$.
The results for $\alpha=0.5\%$ and $\alpha=5\%$ are shown in Figure \ref{fig:nullDist05pc} and Figure \ref{fig:nullDist5pc} of the Appendix, respectively.
In general, $\widehat{\alpha}$ approximates $\alpha$ much better than $\widetilde{\alpha}$ does, 
especially when $m,k,h$ are small. 
When $m,h$ are larger, they perform similarly   
because both $F_{k, \widetilde{\df}(\mathcal{r}_m,h)}$ and $F_{k, \widehat{\df}(\mathcal{r}_m,h)}$ 
get closer to $\chi^2_k/k$. 
However, the performance of $\widetilde{\alpha}$ and $\widehat{\alpha}$
is not monotonic in $\mathcal{f}_m$. 
The performance of $F_{k,\widehat{\df}(\mathcal{r}_m,h)}$ 
is particularly good for $0\%\lesssim \mathcal{f}_m \lesssim 30\%$. 
Consequently, we recommend using $F_{k,\widehat{\df}(\widehat{r}_{\lrt}^{\rob},h)}$ 
as an approximate null distribution for 
$\widehat{D}_{\lrt}^{\rob}$,  
and $F_{k,\widehat{\df}(\widehat{r}_{\lrt}^{+},k)}$ for $\widehat{D}_{\lrt}^{+}$, 
as employed in the rest of this paper. However, these approximations obviously suffer from the usual ``plug-in problem" by ignoring the uncertainty in estimating $\mathcal{r}_m$. 
Because $F_{k, {\rm df}}$ is not  too sensitive to the value of ${\rm df}$ once it is reasonably large ($\df\ge 20$), the ``plug-in problem" is less an issue here than in many other contexts, 
leading to acceptable approximations, as empirically demonstrated in Section~\ref{sec:example}. Nevertheless, further improvements 
should be sought, especially for dealing with the violation of the EFMI assumption, which would likely make the performance of our tests deteriorate with large $k$ or $h$, in contrast to the results shown in Figure~\ref{fig:nullDist05pc}; 
see \cite{chanSMI} for a possible remedy.

\section{Computational Considerations and Comparisons}\label{sec:comp}

\subsection{Computationally Feasible Combining Rule}\label{sec:compFeasibleComRule}
For many real-world data sets, 
$X$ is an $n\times p$ matrix, with rows indicating subjects and columns indicating attributes. 
We write $X=(X_1,\ldots,X_n)^{\T}$, and its sampling model by $f_n(X\mid\psi)$. 
Correspondingly, the $\ell$th imputed data set is
$X^{(\ell)} = (X^{(\ell)}_{1},\ldots,X^{(\ell)}_{n})^{\T}$.
Define the stacked data set by
$X^{(1:m)} = [(X^{(1)})^\T,\ldots,(X^{(m)})^\T]^\T$, 
a $mn\times p$ matrix, which is conceptually different from the collection of data sets 
$\{X^{(1)},\ldots,X^{(m)}\}$. 
Assuming that the rows of $X$ are independent, 
we can compute (\ref{eqt:jointDensity}) as 
\begin{equation}\label{eqt:prod_joint_density}
	\overline{\loglik}(\psi)
		= \frac{1}{m} \log f_{mn}(X^{(1:m)}\mid \psi).
\end{equation}
Consequently, as long as the user's complete-data procedures can handle size $mn$ instead of $n$, the user can apply them to $X^{(1:m)}$ to obtain 
$\widehat{D}^{+}_{\lrt}$ and $\widehat{D}^{\rob}_{\lrt}$ in (\ref{eqt:newstat}).

In many applications, the rows correspond to individual subjects. 
Thus, the row-independence assumption typically holds for \textit{arbitrary} $n$. Hence, we can extend from $n$ to $mn$, assuming the user's complete-data procedure is not size-limited. Even if this is not true, 
(\ref{eqt:prod_joint_density}) can still hold approximately
under some regularity conditions;
see Appendix~\ref{supp:results}, where we also reveal a subtle, but important difference 
between the computation formulae (\ref{eqt:jointDensity}) and (\ref{eqt:prod_joint_density}).

Similar to $\mathcal{D}_{\lrt}$ in (\ref{eqt:computeCode}), we define complete-data functions 
\begin{gather}
	\mathcal{D}_{\lrt,0}(X)
		= 2\log f(X\mid \widehat{\psi}_0(X)),  \quad
	\mathcal{D}_{\lrt,1}(X)
		= 2\log f(X\mid \widehat{\psi}(X)), \label{eqt:computeCode01}
\end{gather}
the only input of which is the data set $X$.
Clearly, $\mathcal{D}_{\lrt}(X)=\mathcal{D}_{\lrt,1}(X)-\mathcal{D}_{\lrt,0}(X)$.
The subroutine for evaluating the complete-data LRT function 
$X\mapsto\mathcal{D}_{\lrt}(X)$ is usually available, as is the subroutine for 
$X\mapsto\mathcal{D}_{\lrt,1}(X)$, for example, the function \texttt{logLik} in R extracts 
the maximum of the complete data log-likelihood for objects belonging to 
classes \texttt{"glm"}, \texttt{"lm"}, \texttt{"nls"}, and \texttt{"Arima"}.

Algorithms \ref{algo:MI_LRT_rob} and \ref{algo:MI_LRT_pos} 
compute $\widehat{D}^{\rob}_{\lrt}$ and $\widehat{D}_{\lrt}^{+}$, respectively.  
We recommend using the robust MI LRT in Algorithm \ref{algo:MI_LRT_rob}, 
because it has the best theoretical guarantee.  
The second test can be useful when $\mathcal{D}_{\lrt}$ is available but $\mathcal{D}_{\lrt,1}$ is not.

\begin{algorithm}[t!] 
\small
\caption{(Robust) MI LRT statistic $\widehat{D}^{\rob}_{\lrt}$}\label{algo:MI_LRT_rob}
\SetAlgoVlined
\DontPrintSemicolon
\textbf{Input}:  
{Data sets $X^{(1)},\ldots,X^{(m)}$; $h,k$; 
functions $\mathcal{D}_{\lrt,1}$, $\mathcal{D}_{\lrt}$ in (\ref{eqt:computeCode01}), (\ref{eqt:computeCode}).} \;
\Begin{
    Stack the data sets to form $X^{(1:m)} = [(X^{(1)})^\T,\ldots,(X^{(m)})^\T]^\T$. \\ 
    Find $\overline{\delta}_{\lrt} =\sum_{\ell=1}^m\mathcal{D}_{\lrt,1}(X^{(\ell)})/m$,
			$\widehat{\delta}_{\lrt} = \mathcal{D}_{\lrt,1}(X^{(1:m)})/m$, 
			$\widehat{d}_{\lrt} = \mathcal{D}_{\lrt}(X^{(1:m)})/m$. \\
	Calculate 
	$\widehat{r}_{\lrt}^{\rob}$ according to (\ref{eqt:def_hatr_rob}), and 
	$\widehat{D}^{\rob}_{\lrt}$ according to (\ref{eqt:final_proposal}) and (\ref{eqt:newstat}). \\
	Calculate $\widehat{\df}(\widehat{r}^{\rob}_{\lrt},h)$ according to (\ref{eqt:df2New}).\\
	Compute the $p$-value as $1-F_{k,\widehat{\df}(\widehat{r}^{\rob}_{\lrt},h)}(\widehat{D}_{\lrt}^{\rob})$. 
}	
\end{algorithm}

\begin{algorithm}[t!] 
\caption{MI LRT statistic $\widehat{D}_{\lrt}^{+}$}\label{algo:MI_LRT_pos}
\small
\SetAlgoVlined
\DontPrintSemicolon
\textbf{Input}:  
{Data sets $X^{(1)},\ldots,X^{(m)}$; 
 $k$;  
function $\mathcal{D}_{\lrt}$ in (\ref{eqt:computeCode}).} \;  
\Begin{
	Stack the data sets to form $X^{(1:m)} = [(X^{(1)})^\T,\ldots,(X^{(m)})^\T]^\T$.\\ 
	Find $\overline{d}_{\lrt} = \sum_{\ell=1}^m \mathcal{D}_{\lrt}(X^{(\ell)})/m$ and $\widehat{d}_{\lrt} = m^{-1}\mathcal{D}_{\lrt}(X^{(1:m)})$.\\
	Calculate 
	$\widehat{r}_{\lrt}^{+}$ according to (\ref{eqt:rLLpos}), and 
	$\widehat{D}_{\lrt}^{+}$ according to (\ref{eqt:final_proposal}) and (\ref{eqt:newstat}).  \\
	Calculate $\widehat{\df}(\widehat{r}^+_{\lrt},k)$ according to (\ref{eqt:df2New}).\\
	Compute the $p$-value as $1- F_{k,\widehat{\df}(\widehat{r}^+_{\lrt},k)}(\widehat{D}_{\lrt}^{+}).$
}	
\end{algorithm}

\subsection{Computational Comparison with Existing Tests}\label{sec:comparison}
Different MI tests require different computing subroutines, 
for example, $\mathcal{D}_{\lrt}$, $\widetilde{\mathcal{D}}_{\lrt}$, $\mathcal{D}_{\lrt,1}$,   
\begin{gather*}		
	\mathcal{M}_{\wt}(X) = \left\{ \widehat{\theta}(X), U(X) \right\} \qquad\text{and}\qquad
	\mathcal{M}_{\lrt}(X) = \left\{ \widehat{\psi}(X) , \widehat{\psi}_0(X) \right\} , 
\end{gather*}
where the unnormalized density can be used in $\mathcal{D}_{\lrt,1}$. 
We summarize the computing requirement in Table \ref{table:prosCons}. 
We also compare the following statistical and computational properties
of various MI test statistics and various estimators of $\mathcal{r}_m$:
\begin{itemize}[noitemsep]
\vspace{-0.2cm}
	\item (Inv) The MI test is invariant to re-parametrization of $\psi$. 
	\item (Con) The estimator of $\mathcal{r}_m$ is consistent, regardless of whether or not $H_0$ is true.
	\item ($\geq 0$) The test statistic and estimator of $\mathcal{r}_m$ 
			are always non-negative.  
	\item (Pow) The MI test has high power to reject $H_0$ under $H_1$. 
	\item (Def) The MI test statistic is well defined and numerically well conditioned. 
	\item (Sca) The MI procedure requires that users deal with scalars only.
	\item (EFMI) Whether EFMI is assumed for $\theta$ or for $\psi$. 
\end{itemize}

\begin{table}
\caption{\small Computational requirements and statistical properties of 
MI tests.
The symbol ``$+$'' (resp.~``$-$'') means that a test
has (resp.~does not have) the indicated property;
see Section~\ref{sec:comparison} for detailed descriptions.
WT-1 \citep{rubin2004multiple,li91} and 
LRT-1 \citep{mengRubin92} are existing tests. 
LRT-2 and LRT-3 are the proposed tests, which can be computed by Algorithms \ref{algo:MI_LRT_pos} and \ref{algo:MI_LRT_rob}, respectively. LRT-3 is recommended.}
\setlength{\tabcolsep}{3.5pt}
\renewcommand{\arraystretch}{0.5}
\begin{center}
\begin{tabular}{cccc ccccccc} \toprule
&&&&\multicolumn{7}{c}{Properties}\\
\cmidrule(r){5-11}
Test & Statistic & Distribution & Routine & Inv & Con & $\geq 0$ & Pow & Def & Sca & EFMI \\ 
\cmidrule(r){1-11}
WT-1  & $D_{\wt}(T)$ & $\approx F_{k, \widetilde{\df}(\widetilde{r}_{\wt}',k)}$ & $\mathcal{M}_{\wt}$ & $-$ & $+$ & $+$ & $-$ & $-$ & $-$ & $\theta$ \\
LRT-1 & $\widetilde{D}_{\lrt}(\widetilde{r}_{\lrt})$ & $\approx F_{k, \widetilde{\df}(\widetilde{r}_{\lrt},k)}$ & $\mathcal{M}_{\lrt},\widetilde{\mathcal{D}}_{\lrt}$ & $-$ & $-$ & $-$ & $-$ & $+$ & $-$ & $\theta$\\
LRT-2 & $\widehat{D}_{\lrt}(\widehat{r}^+_{\lrt})$ & $\approx F_{k, \widehat{\df}(\widehat{r}_{\lrt}^+,k)}$ & $\mathcal{D}_{\lrt}$ & $+$ & $-$ & $+$ & $-$ & $+$ & $+$ & $\theta$ \\
LRT-3 & $\widehat{D}_{\lrt}(\widehat{r}^{\rob}_{\lrt})$ & $\approx F_{k, \widehat{\df}(\widehat{r}_{\lrt}^{\rob},k)}$ & $\mathcal{D}_{\lrt},\mathcal{D}_{\lrt,1}$ & $+$ & $+$ & $+$ & $+$ & $+$ & $+$ & $\psi$ \\
\bottomrule
\end{tabular}
\end{center}
\label{table:prosCons}
\end{table}

In summary, our proposed LRT-2
is the most attractive computationally. 
If the user is willing to make stronger assumptions, 
our proposed LRT-3 has better statistical properties, and is still computationally feasible. 
In practice, we recommend using LRT-3. 
We also present other existing MI tests and compare our proposals with them
in Appendix \ref{sec:otherMItests}.

\subsection{Summary of Notation}
For ease of referencing, we summarize all major notation used in the paper.
Recall that 
$\psi\in\mathbb{R}^h$ is the model parameter, and 
$\theta$ is the parameter of interest.
We would like to test $H_0:\theta=\theta_0$ against $H_1:\theta\neq \theta_0$. 

\begin{itemize}[noitemsep]
	\item Complete-data Estimators and Test Statistics:
			\begin{itemize}[noitemsep]
				\item $\widehat{\theta}(X)$ and $U(X)$: MLE of $\theta$ and its variance estimator.
				\item $\widehat{\psi}(X)$ and $\widehat{\psi}_0(X)$:
						the unrestricted and $H_0$-restricted MLEs of $\psi$.
				\item $d_{\wt}(\widehat{\theta},U) = (\widehat{\theta}-\theta_0)^\T U^{-1} (\widehat{\theta}-\theta_0)$: the Wald test statistic.
				\item $d_{\lrt}(\widehat{\psi}_0,\widehat{\psi}\mid X) 
						= 2 \log \{ f(X\mid\widehat{\psi}) / f(X\mid\widehat{\psi}_0)\}$:
						the LRT statistic.
			\end{itemize}
	\item Complete-data Functions (or Software Routines): 
			\begin{itemize}[noitemsep]
				\item $\mathcal{M}_{\wt}(X)= \{ \widehat{\theta}(X), U(X) \}$
					and $\mathcal{M}_{\lrt}(X)= \{ \widehat{\psi}(X) , \widehat{\psi}_0(X) \}$.
				\item $\widetilde{\mathcal{D}}_{\lrt}(X,\psi_0,\psi)= d_{\lrt}(\psi_0,\psi\mid X)$:
						a nonstandard LRT function/routine. 
				\item $\mathcal{D}_{\lrt}(X)= d_{\lrt}(\widehat{\psi}_0(X),\widehat{\psi}(X)\mid X)$:
						the standard LRT function/routine.
				\item $\mathcal{D}_{\lrt,1}(X ) = 2\log f(X\mid \widehat{\psi}(X))$:
						the (scaled) maximum log-likelihood.
			\end{itemize}	
	\item MI Statistics:
			\begin{itemize}[noitemsep]
				\item $\widehat{\theta}^{(\ell)}$, $U^{(\ell)}$, 
						$\widehat{\psi}_{0}^{(\ell)}$, $\widehat{\psi}^{(\ell)}$, 
						$d^{(\ell)}_{\wt}$, $d_{\lrt}^{(\ell)}$:
						the imputed values of $\widehat{\theta}$, $U$,
						$\widehat{\psi}_0$, $\widehat{\psi}$, 
						$d_{\wt}(\widehat{\theta},U)$, $d_{\lrt}(\widehat{\psi}_0,\widehat{\psi}\mid X)$
						using the imputed data set $X^{(\ell)}$ for each $\ell$.
				\item $\overline{\theta}$, $\overline{U}$, $\overline{\psi}_0$, $\overline{\psi}$,
						$\overline{d}_{\wt}$, $\overline{d}_{\lrt}$:
						the averages (over $\ell$)
						of $\widehat{\theta}^{(\ell)}$, $U^{(\ell)}$, 
						$\widehat{\psi}_{0}^{(\ell)}$, $\widehat{\psi}^{(\ell)}$, 
						$d^{(\ell)}_{\wt}$, $d_{\lrt}^{(\ell)}$.
				\item $T = \overline{U} + (1+1/m)B$, where
						$B = \sum_{\ell=1}^m (\widehat{\theta}^{(\ell)}-\overline{\theta})(\widehat{\theta}^{(\ell)}-\overline{\theta})^\T/(m-1)$.
				\item $\overline{d}_{\wt}^{\prime}= \sum_{\ell=1}^m d_{\wt}(\widehat{\theta}^{(\ell)},\overline{U})/m$ and $\widetilde{d}_{\wt}^{\prime}=  d_{\wt}(\overline{\theta},\overline{U})$.
				\item $\widetilde{d}_{\lrt} = \sum_{\ell=1}^m \widetilde{\mathcal{D}}_{\lrt}(X^{(\ell)},\overline{\psi}_{0} ,\overline{\psi})/m$: an existing pooled LRT statistic. 
				\item $\widehat{d}_{\lrt}= \mathcal{D}_{\lrt}(X^{(1:m)})/m$: 
						the proposed pooled LRT statistic. 
				\item $\overline{\delta}_{\lrt} =\sum_{\ell=1}^m\mathcal{D}_{\lrt,1}(X^{(\ell)})/m$
						and $\widehat{\delta}_{\lrt} = \mathcal{D}_{\lrt,1}(X^{(1:m)})/m$:
						two proposed ways for pooling maximized log-likelihood.
			\end{itemize}		
	\item Estimators of $\mathcal{r}_m$:
			\begin{itemize}[noitemsep]
				\item $\widetilde{r}_{\wt}^{\prime} = (m+1)( \overline{d}_{\wt}^{\prime}- \widetilde{d}_{\wt}^{\prime})/\{k(m-1)\}$
				\add{\citep{rubin2004multiple}}.
				\item $\widetilde{r}_{\lrt} = (m+1) ( \overline{d}_{\lrt} - \widetilde{d}_{\lrt} ) /\{k(m-1)\}$ 
				\add{\citep{mengRubin92}}.
				\item $\widehat{r}_{\lrt}^+ = \max[0,(m+1) ( \overline{d}_{\lrt} - \widehat{d}_{\lrt} )/\{k(m-1)\}]$:
						our first proposal.
				\item $\widehat{r}_{\lrt}^{\rob}
						= (m+1) (\overline{\delta}_{\lrt} - \widehat{\delta}_{\lrt})/\{h(m-1)\}$:
						our second proposal.
			\end{itemize}		
	\item MI Test Statistics for Testing $H_0$ against $H_1$:
			\begin{itemize}[noitemsep]
				\item (WT-1) $D_{\wt}(T) = d_{\wt}(\overline{\theta},T)/k$:
						the classical MI Wald test. 
				\item (LRT-1) $\widetilde{D}_{\lrt}(\widetilde{r}_{\lrt}) = {\widetilde{d}_{\lrt}}/\{k(1+\widetilde{r}_{\lrt})\}$: 
						the existing MI LRT. 
				\item (LRT-2) $\widehat{D}_{\lrt}(\widehat{r}_{\lrt}^+) = {\widehat{d}_{\lrt}}/\{k(1+\widehat{r}_{\lrt}^+)\}$:
						our first proposal.
 				\item (LRT-3) $\widehat{D}_{\lrt}(\widehat{r}_{\lrt}^{\rob}) = {\widehat{d}_{\lrt}}/\{k(1+\widehat{r}_{\lrt}^{\rob})\}$:
						our second proposal.
			\end{itemize}
\end{itemize}

\vspace{-0.3cm}
\section{Empirical Investigation and Findings}\label{sec:example}
\subsection{Monte Carlo Experiments With EFMI}\label{sec:eg_lrt_normal}
Let $X_1,\ldots,X_n\sim\mathcal{N}_p(\mu,\Sigma)$ independently, 
where $\mu=(\mu_1,\ldots,\mu_p)^{\T}$.
Assume that only $n_{\obs} = \lfloor (1-\mathcal{f})n \rfloor$ data points are observed. 
Let $X_{\obs} = \{X_i: i=1,\ldots, n_{\obs}\}$ and
$X_{\mis} = \{X_i: i=n_{\obs}+1, \ldots, n\}$.
We want to test $H_0: \mu_1 = \cdots = \mu_p$.

Obviously,  one may directly use the observed data set to construct the LRT statistic $D_{\lrt}$ without MI.
Thus, it is regarded as a benchmark (denoted by LRT-0).  
The tests 
WT-1 and LRT-1,2,3 listed in Table \ref{table:prosCons} are investigated.
We perform MI using a Bayesian model with a multivariate Jeffreys prior on $(\mu,\Sigma)$, 
this is, $f(\mu,\Sigma)\propto |\Sigma|^{-(p+1)/2}$. 
The imputation procedure is detailed in Appendix \ref{sec:eg_lrt_normal_MI}.
We study the impact of the parametrization on
different test statistics.

\begin{itemize}[noitemsep]
\vspace{-0.3cm}
\item Parametrizations of $\theta$ for the Wald tests: 
(i) $\theta = \left(\mu_2-\mu_1, \ldots, \mu_p-\mu_{p-1}\right)^{\T}$;
(ii) $\theta = \left(\mu_2/\mu_1-1, \ldots, \mu_p/\mu_{p-1}-1\right)^{\T}$; and
(iii) $\theta = \left(\mu_2^3-\mu_1^3, \ldots, \mu_p^3-\mu_{p-1}^3\right)^{\T}$. 
For any case above, $H_0$ can be expressed as $\theta=(0,\ldots,0)^{\T}$. 
\item Parametrizations of $\psi$ for LRTs:
(i) $\psi = \left\{ \mu;\Sigma \right\}$; 
(ii) $\psi = \{\sqrt{\sigma_{ii}}/\mu_i, 1\le i \le p; \Sigma\}$; and 
(iii) $\psi=\left\{\mu^{\T}\Sigma^{-1/2}; \Sigma^{-1}\right\}$,
where $\Sigma=(\sigma_{ij})$  
and $\Sigma^{1/2}$ is the square root of $\Sigma$ via the spectral method. 
The dimension of $\psi$ is $h=(p^2+3p)/2$.
\end{itemize}
We set $\Sigma = \sigma^2 \{(1-\rho)I_p + \rho \mathbf{1}_p\mathbf{1}_p^{\T}\}$,
$\mathcal{f}=0.5$, $p=2$, $\rho=0.8$, $\sigma^2=5$, and $\mu=(-2+\delta,-2+2\delta)^{\T}$
for different values of $m\in\{3,10,30\}$, $n\in\{100,400,1600\}$,
and $\delta=\mu_2-\mu_1\in[0,4]$.
All simulations are repeated $2^{12}$ times. 
The empirical power functions 
for $\alpha=0.5\%$ tests  
are plotted in Figure \ref{fig:power_05pc}.
The results for $\alpha=5\%$ tests are deferred to Table~\ref{fig:power_5pc} of the Appendix.

\begin{figure}[t!]
\begin{center}
\includegraphics[width=.95\textwidth]{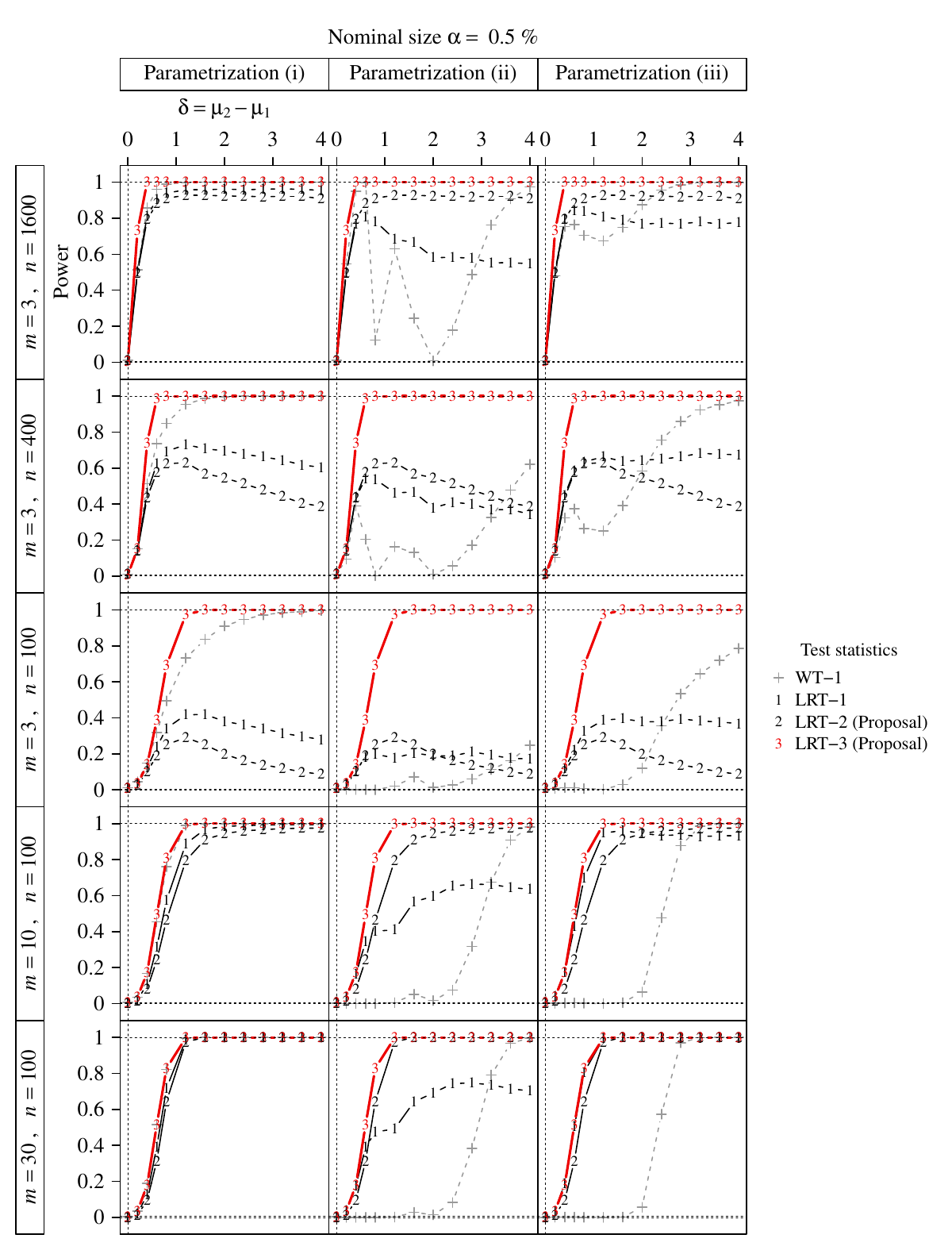} 
\vspace{-0.6cm}
\caption{\small 
The power curves under nominal size $\alpha=0.5\%$. 
In each plot, the vertical axis denotes the power,
and the horizontal axis denotes the value of $\delta=\mu_2-\mu_1$.
}
\label{fig:power_05pc}
\end{center} 
\end{figure}

In general, WT-1 exhibits monotonically increasing power as $\delta$ increases, and its performance is affected significantly by parametrization.  
Indeed, the power can be as low as zero 
when $1\lesssim \delta \lesssim 2$ under parametrizations (ii) and (iii).
Under parametrization (ii), 
LRT-1 is not powerful, even for large $\delta$. 
On the other hand, 
our first proposed test statistic LRT-2 performs
better than LRT-1, at least for large $m$;
however, they also lose a significant amount of power when $m$ is small.
Our recommended proposal LRT-3 performs best in all cases. 
The superiority of LRT-3 is particularly striking when $m$ is small, this is, $m=3$.

We also investigate
(a) the distribution of the $p$-value,
(b) the empirical size
$\widehat{\alpha}$ in comparison to the nominal type-I error $\alpha$, 
(c) the empirical size-adjusted power \citep{BBBS2015},
(d) the robustness of our proposed estimators of $\mathcal{r}_m$, and 
(e) the performance of other existing MI tests.  
The results are shown in Appendix \ref{sec:eg_lrt_normal_MI}, 
all of which indicate that our proposed tests perform best.

\subsection{Monte Carlo Experiments Without EFMI}\label{sec:MCexpUFMI}
To check how robust various tests are to the assumption of EFMI, we simulate $X_{i}  = (X_{i1}, \ldots, X_{ip})^{\T} \sim \mathcal{N}_p(\mu, \Sigma)$ independently 
for $i=1,\ldots, n$. 
Let $R_{ij}$ be defined by $R_{ij}=1$ if $X_{ij}$ is observed, and $R_{ij}=0$ otherwise. 
Suppose that the first variable $X_{\cdot 1}$ is always observed, and the rest form a monotone missing pattern, as defined by a logistic model on the missing propensity:
$\pr\left( R_{ij} = 0 \mid R_{i,j-1} = a \right) = 
\left[1+ \exp(  \alpha_0 + \alpha_1 X_{i,j-1}) \right]^{-1}$
(for $j=2,\ldots,p$) when $a=1$. This probability is zero when $a=0$ (i.e., nothing is missing).   
If $\alpha_1=0$, the data are missing completely at random (MCAR); otherwise they are missing at random (MAR); see \cite{rubin1976}. The imputation procedure is given in Appendix \ref{sec:MCexpUFMI_MI}.

We test $H_0: \mu = \bm{0}_p$ against $H_1: \mu \not{=} \bm{0}_p$.
We set 
$\mu = \delta \bm{1}_p$, where $\delta \in [0,0.6]$; 
$\Sigma_{ij} = 0.5^{|i-j|}$, for $i,j=1,\ldots,p$; 
$n=500$; 
$m\in\{3,5\}$; 
$p=5$; and
$(\alpha_0, \alpha_1) \in \{ (2,-1) , (1,0)\}$.
Our model treats $\Sigma$ as unknown, and hence $k=p$ and $h=(3p+p^2)/2$. 	
Under $H_0$ and MAR, 
the FMI, i.e., 
the eigenvalues of $\mathcal{B}_{\theta}\mathcal{T}_{\theta}^{-1}$, 
are $(0, 19\%, 34\%, 45\%, 55\%)$. 
Thus, the assumption of EFMI does not hold.

In this experiment, we also compare the performance of WT-1 and LRT-1,2,3.
For reference, 
the complete-case (asymptotic) LRT using $\left\{ X_{i} : R_{i1}=\cdots=R_{ip}=1 \right\}$, denoted by LRT-0, 
is also computed. 
The results are shown in Figure \ref{fig:jmi_3Jun2018_eg4a2}.
The size of LRT-3 is accurate 
when the nominal size is small. 
If the data are MCAR, LRT-0 is valid, 
but with slightly less power.
(LRT-0 is typically invalid without MCAR.) 
The test LRT-3 has the best power-to-size ratio
among all other tests. 
The power-to-size ratio of LRT-2 and LRT-3
become closer to the nominal value $1/0.5\%=200$ as $m$ increases. 
These results indicate that our proposed tests perform well and best,
despite the serious violation of the EFMI assumption.

\begin{figure}[t!]
\begin{center}
\includegraphics[width=1\textwidth]{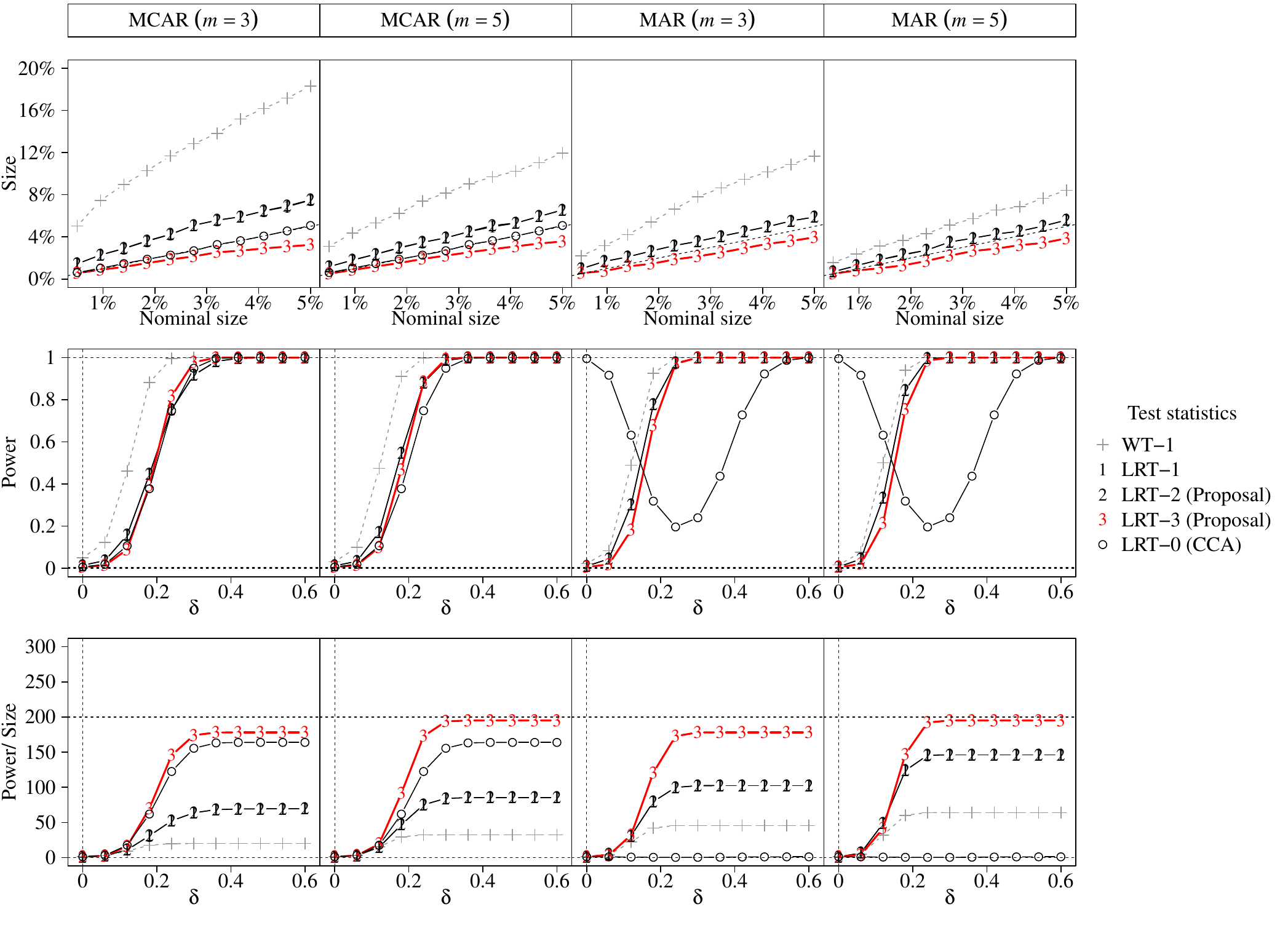} 
\vspace{-1cm}
\caption{\small The empirical size, empirical power, and their ratio. 
The first row of plots show the empirical sizes.
The size of the complete-case test (C2) under MAR is off the chart (always equal to one) because it is invalid. 
The second and third rows of plots show the powers 
and the power-to-size ratios, respectively. The nominal size is $0.5\%$.}
\label{fig:jmi_3Jun2018_eg4a2}
\end{center}
\end{figure}

\section{Conclusion, Limitation and Future Work}\label{sec:conclusion} 
In addition to conducting a general comparative study of MI tests, we have proposed two particularly promising MI LRTs based on $\widehat{D}_{\lrt}^{\rob} = \widehat{D}_{\lrt}(\widehat{r}^{\rob}_{\lrt})$ and
$\widehat{D}_{\lrt}^{+} = \widehat{D}_{\lrt}(\widehat{r}^{+}_{\lrt})$. Both test statistics are non-negative, invariant to parametrizations, and powerful to reject a false $H_0$ (at least for large enough $m$). 
The test $\widehat{D}_{\lrt}^{\rob}$
is the most principled, and has desirable monotonically increasing power as $H_1$ departs from $H_0$. 
However, it is derived under the stronger assumption of EFMI for $\psi$, not just for $\theta$. 
Furthermore, row independence of $X_{\com}$ is needed for ease of computation. 
(With a slightly more computationally demanding requirement, 
$\widehat{D}_{\lrt}(\widehat{r}^{\rob}_{\lrt})$ can be used without the independence assumption.) 
The main advantage of $\widehat{D}_{\lrt}^{+}$ is that it is  easier to compute, because it requires only standard complete-data computer subroutines for LRTs. One drawback is that the ad hoc fix $\widehat{r}^{+}_{\lrt} = \max(0,\widehat{r}_{\lrt})$
is inconsistent, in general. However, the inconsistency does not significantly affect the asymptotic power,  at least in our experiments. 
Although $\widehat{D}_{\lrt}^{+}$
and $\widehat{D}_{\lrt}^{\rob}$
offer significant improvements over existing options, 
more research is needed, for the reasons listed below:
\begin{itemize}[noitemsep]
\vspace{-0.3cm}
\item When the missing-data mechanism is not ignorable, but the imputers fail to fully take that into account, the issue of uncongeniality becomes critical \citep{meng1994}. \citet*{XXMeng2017} provide theoretical tools to address this issue in the context of estimation, and research is needed to extend their findings to the setting of hypothesis testing. 
	\item  Violating the EFMI assumption may not  
	invalidate a test, but it will affect its power. 
	Thus, it is desirable to explore MI tests without assuming EFMI. 
	\item The robust $\widehat{D}_{\lrt}^{\rob}$
			relies on a stronger assumption of EFMI on $\psi$. 
			We can modify it so only EFMI on $\theta$ is required, 
			but the modification may be very difficult to compute, 
			and may require that users have access to nontrivial complete-data procedures.
			Hence, a computationally feasible robust test that only assumes EFMI on $\theta$
			needs to be developed. 
	\item Because the FMI is a fundamental nuisance parameter and there is no (known) pivotal quantity, all MI tests are just approximations.
	If FMI is large or $m$ is small, 
	they may perform poorly.
	Thus,  
	seeking powerful MI tests that are least affected by FMI
	is of both theoretical and practical interest.
\end{itemize}

{\small
\vskip 5pt
\noindent {\large\bf Supplementary Material}\\
Appendix \ref{supp:results} contains additional theoretical results
and details of numerical examples.
Appendix \ref{sec:proof} contains proofs of the main results.
The R code is provided online.

\par
\vskip 5pt
\noindent {\large\bf Acknowledgments}\\
Meng thanks the NSF and JTF for their partial financial support. 
He is also grateful for 
Keith's (Kin Wai's) creativity and diligence, 
which led to the remedies presented here, and which are also a part of Keith's thesis. 
Chan thanks the University Grant Committee of HKSAR for its partial financial support. 
\par
}
\markboth{\hfill{\footnotesize\rm Kin Wai Chan AND Xiao-Li Meng} \hfill}
{\hfill {\footnotesize\rm Multiple imputation LR tests} \hfill}

\vspace{-0.3cm}
\bibhang=1.7pc
\bibsep=2pt
\fontsize{9}{14pt plus.8pt minus .6pt}\selectfont
\renewcommand\bibname{\large \bf References}
\expandafter\ifx\csname
natexlab\endcsname\relax\def\natexlab#1{#1}\fi
\expandafter\ifx\csname url\endcsname\relax
  \def\url#1{\texttt{#1}}\fi
\expandafter\ifx\csname urlprefix\endcsname\relax\def\urlprefix{URL}\fi

\lhead[\footnotesize\thepage\fancyplain{}\leftmark]{}\rhead[]{\fancyplain{}\rightmark\footnotesize{} }
\bibliographystyle{rss}
{\footnotesize
\bibliography{myRef.bib}
}

\vskip .65cm
\noindent
Department of Statistics, The Chinese University of Hong Kong.
\vskip 2pt
\noindent
E-mail: {\ttfamily kinwaichan@cuhk.edu.hk}
\vskip 2pt
\noindent
Department of Statistics, Harvard University.
\vskip 2pt
\noindent
E-mail: {\ttfamily meng@stat.harvard.edu}

\newpage
\begin{center}
\Large Supplementary Material
\end{center}

\appendix
\fontsize{12}{14pt plus.8pt minus .6pt}\selectfont
\section{Supplementary Results}\label{supp:results}
\subsection{A Complication Caused by Nuisance Parameter}\label{subsec:nusiancePara}
This section supplement the discussion of Section \ref{sec:rob} in the main article.
Recall that the likelihood function $\loglik^{(\ell)}(\cdot)$
is based on both observed data $X_{\obs}$ and imputed data $X_{\mis}^{(\ell)}$, 
which varies across $\ell$.
Hence, each imputed likelihood $\loglik^{(\ell)}(\cdot)$ is associated with 
a (imputation-specific) pseudo parameter $\psi^{(\ell)}$, 
may vary across $\ell=1,\ldots, m$.

To see the source of the negativity of $\widehat{r}_{\lrt}$,
we extend $\overline{\loglik}(\psi)$ in (\ref{eqt:jointDensity}) to 
\begin{equation}\label{eqt:general_Lbar}
	\overline{\loglik}(\psi^{(1)}, \ldots, \psi^{(m)}) = \frac{1}{m}\sum_{\ell=1}^m \loglik^{(\ell)}(\psi^{(\ell)}).
\end{equation}
Using the ``log-likelihood'' $\overline{\loglik}(\psi^{(1)}, \ldots, \psi^{(m)})$, 
we can construct, at least conceptually, 
four hypotheses $H_0^{0}$, $H_0^{1}$, $H_1^{0}$, $H_1^{1}$ 
defined in Table \ref{table:4hypotheses}.
Each of them consists of zero, one or two of the constraints 
$\mathcal{C}_0: \theta^{(1)} = \cdots = \theta^{(m)} = \theta_0$
and 
$\mathcal{C}^0: \psi^{(1)} = \cdots = \psi^{(m)}$, 
where $\theta^{(\ell)} = \theta(\psi^{(\ell)})$ is the interested part of $\psi^{(\ell)}$ for each $\ell$.
The constraint $\mathcal{C}_0$ is equivalent to $H_0$, and the constraint $\mathcal{C}^0$ means that all $\psi^{(\ell)}$s are equal, and hence it is effectively equivalent to $\mathcal{r}=0$, 
i.e., no missing information. The relationships among $H_0^{0}$, $H_0^{1}$, $H_1^{0}$, $H_1^{1}$ can be visualized in Figure \ref{fig:4hypotheses}.
Define the maximized value of $\overline{\loglik}(\psi^{(1)}, \ldots, \psi^{(m)})$
under hypothesis $H\in\{H_0^{0},H_0^{1},H_1^{0},H_1^{1}\}$ by $\mathbb{L}(H)$.
Then we can re-express $(\overline{d}_{\lrt} - \widehat{d}_{\lrt})/2$ as
\begin{equation}\label{eqt:diff_of_d_function_of_4_liklihood}
	(\overline{d}_{\lrt} - \widehat{d}_{\lrt} ) /2
	= \left\{\mathbb{L}(H_1^{1}) - \mathbb{L}(H_1^{0})\right\} - \left\{\mathbb{L}(H_0^{1}) - \mathbb{L}(H_0^{0})\right\}.
\end{equation}
Whereas the two bracketed terms in (\ref{eqt:diff_of_d_function_of_4_liklihood}) 
are non-negative as they correspond to two LRT statistics, 
their difference can be negative.

A simple example illustrates this well. 
For the regression model $[Y\mid X_1, X_2] \sim \mathcal{N}(\beta_0 + \beta_1 X_1 + \beta_2 X_2, \sigma^2)$, the LRT statistic for testing $H_1^0: \beta_1=0, \beta_2\in\mathbb{R}$ 
against $H_1^1: \beta_1,\beta_2\in\mathbb{R}$ is not necessarily larger (or smaller) than  
that for testing $H_0^0: \beta_1=\beta_2=0$ against $H_0^1: \beta_1\in\mathbb{R}, \beta_2=0$;
see Figure \ref{fig:contour} for a schematic illustration.

\begin{figure}[t!]
\begin{center}
\includegraphics[width=80mm]{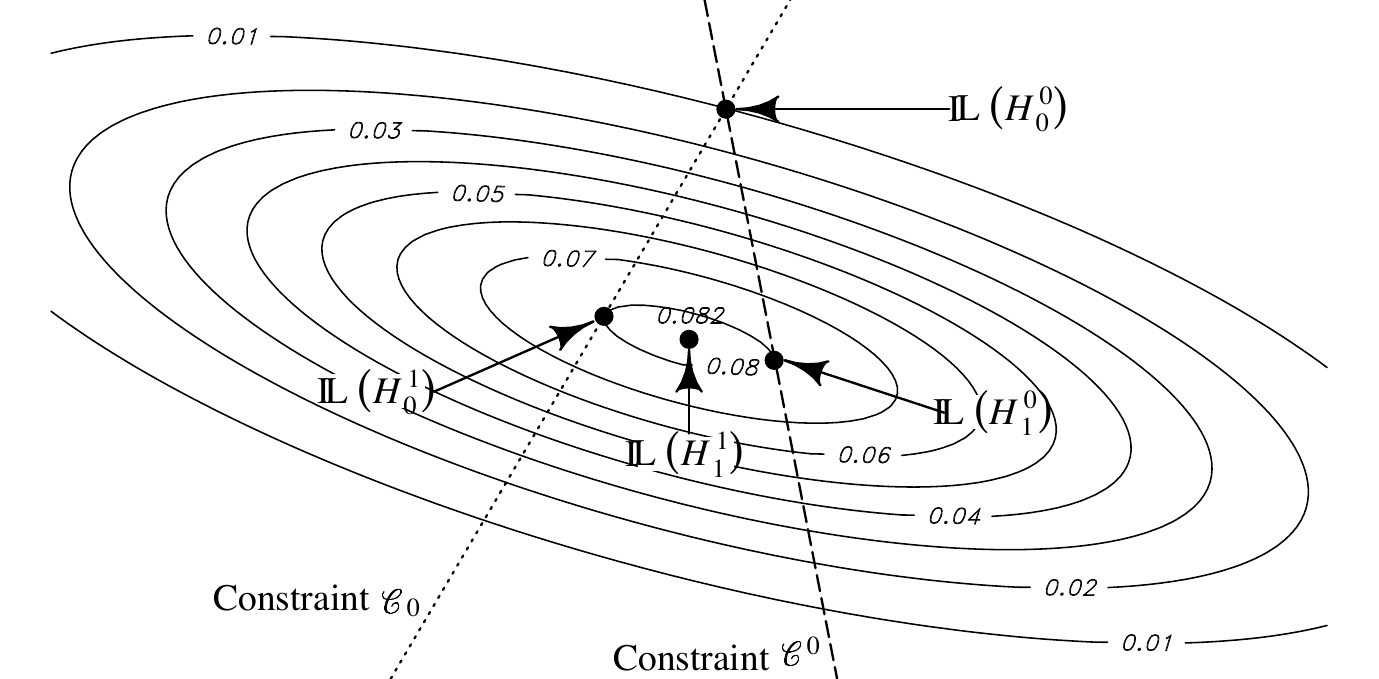} 
\vspace{-0.3cm}
\caption{\small 
A schematic illustration of the sign of (\ref{eqt:diff_of_d_function_of_4_liklihood}). 
The contour lines of $\overline{L}(\psi^{(1)}, \ldots, \psi^{(m)})$ are plotted.
The two straight lines refer to constraints $\mathcal{C}_0$ and $\mathcal{C}^0$.
Since
$\mathbb{L}(H_1^{1}) = 0.082$, $\mathbb{L}(H_0^{1})=\mathbb{L}(H_1^{0}) = 0.08$, 
and $\mathbb{L}(H_0^{0}) = 0.01$, 
we have $\left\{\mathbb{L}(H_1^{1})- \mathbb{L}(H_1^{0})\right\}- \left\{\mathbb{L}(H_0^{1})-\mathbb{L}(H_0^{0})\right\}= 0.002 - 0.007<0$.
Note that the function $\overline{L}(\psi^{(1)}, \ldots, \psi^{(m)})$ in (\ref{eqt:general_Lbar})
is at least $4$-dimensional (i.e., $\theta^{(1)}, \theta^{(2)}, \eta^{(1)}, \eta^{(2)}$) generally, 
so this illustration in a $2$-dimension space is just conceptual.
}
\label{fig:contour}
\end{center} 
\end{figure}

\begin{table*}[t!]
\begin{center}
\setlength{\tabcolsep}{2pt}
\renewcommand{\arraystretch}{0.5}
\caption{\small The definitions of hypotheses $H_0^{0}$, $H_0^{1}$, $H_1^{0}$, $H_1^{1}$.}
\begin{tabular}{c|c|c}\toprule
& $\begin{array}{c}	\mathcal{C}^0: \psi^{(1)} = \cdots= \psi^{(m)} \in\Psi 
						\\\text{(i.e., $\mathcal{r}=0$)}\end{array}$ & 
$\begin{array}{c}	\mathcal{C}^1: \psi^{(1)} , \ldots, \psi^{(m)} \in\Psi 
						\\\text{(i.e., $\mathcal{r}\geq 0$)}\end{array}$ \\
\cmidrule(r){1-3}
$\begin{array}{c}	\mathcal{C}_0: \theta^{(1)} = \cdots= \theta^{(m)} =\theta_0\in\Theta \\ 
						\text{(i.e., $H_0$-constrained)}\end{array}$ 
& $H_{0}^{0}= \mathcal{C}_{0} \cap \mathcal{C}^{0}$ 
& $H_{0}^{1}= \mathcal{C}_{0} \cap \mathcal{C}^{1}$  \\[0.5ex]
\cmidrule(r){1-3}
$\begin{array}{c} \mathcal{C}_1: \theta^{(1)} , \ldots , \theta^{(m)} \in\Theta\\ 
						\text{(i.e., not $H_0$-constrained)}\end{array}$	
& $H_{1}^{0}= \mathcal{C}_{1} \cap \mathcal{C}^{0}$  
& $H_{1}^{1}= \mathcal{C}_{1} \cap \mathcal{C}^{1}$  \\[0.5ex]
\bottomrule
\end{tabular}
\label{table:4hypotheses}
\end{center}
\end{table*}

\begin{figure}[t!]
\begin{center}
\tikzstyle{arrow} = [draw, thick, -latex']
\begin{tikzpicture}
	\node (P1) at (0, 0)   {$H_1^0$};    
	\node (P2) at (3, 0)   {$H_1^1$};    
	\node (P3) at (0, 1.5)   {$H_0^0$};    
	\node (P4) at (3, 1.5)   {$H_1^1$};    
	\path [arrow] (P1) [double] -- (P2);
	\path [arrow] (P3) [double] -- (P1);
	\path [arrow] (P3) [double] -- (P4);
	\path [arrow] (P4) [double] -- (P2);
	\path [arrow] (P3) [double] -- (P2);
\end{tikzpicture}
\end{center}
\vspace{-1cm}
\caption{\small The relationships between the four hypotheses 
$H_0^{0}$, $H_0^{1}$, $H_1^{0}$, $H_1^{1}$.
Each arrow denotes an implication, e.g., 
$H_0^{0} \Rightarrow H_0^{1}$ means that $H_0^{0}$ implies $H_0^{1}$.}\label{fig:4hypotheses}
\end{figure}

The decomposition (\ref{eqt:diff_of_d_function_of_4_liklihood}) 
provides another interpretation of $\widehat{r}_{\lrt}$.
The test statistic $\mathbb{L}(H_1^{1})-\mathbb{L}(H_1^{0})$ 
seeks evidence for detecting 
the falsity of $\mathcal{r}=0$ in both $\theta$ and $\eta$, whereas
$\mathbb{L}(H_0^{1})-\mathbb{L}(H_0^{0})$ seeks evidence 
only in $\eta$.
For cases where $\theta$ and $\eta$ are orthogonal (at least locally), 
the left-hand side of (\ref{eqt:diff_of_d_function_of_4_liklihood}) can be viewed as a measure of evidence against $\mathcal{r}=0$ 
solely from $\theta$;  Proposition~\ref{prop:condForPositiveR} already hinted this possibility. 
However, the ``test statistic'' (\ref{eqt:diff_of_d_function_of_4_liklihood}) 
has a fatal flaw. Because $\mathcal{C}_0$ requires all $\theta^{(\ell)}$s to coincide with a specific $\theta_0$,
$\mathcal{C}_0$ is not nested within $\mathcal{C}^0$, 
i.e., $\mathcal{C}^0 \nRightarrow \mathcal{C}_0$.
Hence $\widehat{r}_{\lrt}$ is guaranteed to consistently estimate $\mathcal{r}_m$ only under $H_0$.
This explains Corollary \ref{coro:finitenessOfr}, 
and leads to an improvement in Section \ref{sec:rob}.
In it not hard to see that our new estimator $\widehat{r}_{\lrt}^{\rob}$ simply 
drops the second term in (\ref{eqt:diff_of_d_function_of_4_liklihood}).

\subsection{Another Motivation for $\widehat{r}_{\lrt}^{\rob}$}\label{appendix:perb}
The definition of $\widehat{r}_{\lrt}^{\rob}$
can also be motivated by the following observation. 
First, observe that one simple method to construct an always non-negative estimator of $\mathcal{r}_m$
is to perturb $\widehat{\psi}_{0}^{*}$ and $\widehat{\psi}_{0}^{(\ell)}$ 
by a suitable amount, say $\Delta$, so that 
the perturbed version of $\widehat{r}_{\lrt}$ 
is always non-negative, and is still asymptotically equivalent to the original $\widehat{r}_{\lrt}$.
We show, in Theorem \ref{thm:AsyEq_r_LL_Delta} below, that the right amount of $\Delta$ is 
$\Delta = \widehat{\psi}^{*} - \widehat{\psi}_{0}^{*}$.
Using the perturbed version of $\widehat{r}_{\lrt}$, we obtain 
\begin{equation*}
	\widehat{r}_{\lrt}^{\pert}
        = \frac{m+1}{k(m-1)} \widehat{\delta}_{\lrt}^{\pert}, 
\end{equation*}
where
\begin{equation}\label{eqt:alternativeDeltad}
    \widehat{\delta}_{\lrt}^{\pert}
       = \frac{2}{m}\sum_{\ell=1}^m 
        \log\left\{\frac{f(X^{(\ell)}\mid\widehat{\psi}^{(\ell)})}{f(X^{(\ell)}\mid\widehat{\psi}^{*})}
        \frac{f(X^{(\ell)}\mid\widehat{\psi}_{0}^{*}+\Delta)}{f(X^{(\ell)}\mid\widehat{\psi}_{0}^{(\ell)}+\Delta)}
        \right\} \nonumber
        = \frac{1}{m} \sum_{\ell=1}^m d_{\lrt}(\widehat{\psi}^{(\ell)}_0+ \Delta, \widehat{\psi}^{(\ell)} \mid X^{(\ell)}) .
\end{equation}
Then we have the following result.

\begin{theorem}\label{thm:AsyEq_r_LL_Delta}
Suppose $\RC_\theta$.
Under $H_0$, we have 
(i) $\widehat{r}_{\lrt}^{\pert} \geq 0$ for all $m,n$; and 
(ii) $\widehat{r}_{\lrt}^{\pert} \bumpeq \widehat{r}_{\lrt}$
as $n\rightarrow\infty$ for each $m$.
\end{theorem}

Although $\widehat{r}_{\lrt}^{\pert}\geq 0$,
it is only invariant to affine transformations, 
and not robust against $\theta_0$, 
and less computational feasible than $\widehat{r}_{\lrt}$; 
see Section~\ref{sec:comp}.
However, it gives us some insights on how to construct a potentially better estimator. 
Note that, in (\ref{eqt:alternativeDeltad}), 
the constrained MLE is not used in $d_{\lrt}(\cdot,\cdot\mid X^{(\ell)})$, 
but it is still always non-negative. 
We call this a ``pseudo'' LRT statistics.
Then, $\widehat{\delta}_{\lrt}^{\pert}$ 
is just a multiple of an average of many ``pseudo'' LRT statistics.
In order to find a good estimator of $\mathcal{r}_m$, 
we may seek for an estimator 
which admits this form.
Indeed, our estimator $\widehat{r}^{\rob}_{\lrt}$ 
also takes the same form:
\[
	\widehat{r}_{\lrt}^{\rob}
		= \frac{m+1}{h (m-1)} \frac{1}{m} \sum_{\ell=1}^m d_{\lrt}(\widehat{\psi}^{*} , \widehat{\psi}^{(\ell)} \mid X^{(\ell)}) .
\]

\subsection{Additional result for Section~\ref{sec:refNullDist}}
This section presents the additional simulation result 
for Section~\ref{sec:refNullDist}.
The performance of different approximations to the reference null distribution
when $\alpha=5\%$ is shown in Figure \ref{fig:nullDist5pc}.

\begin{figure}[h!]
\begin{center}
\includegraphics[width=.95\textwidth]{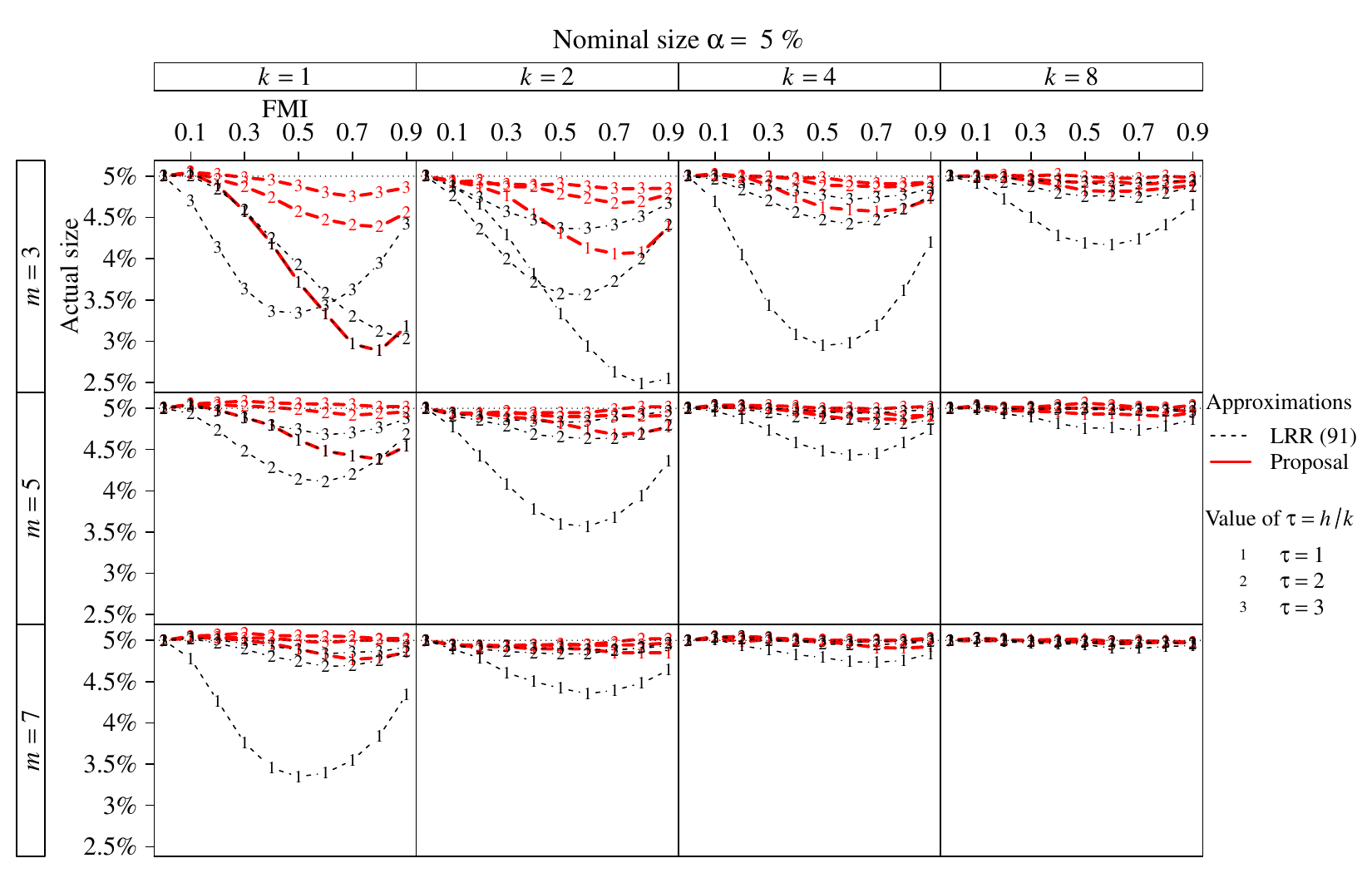} 
\vspace{-0.5cm}
\caption{\small 
The performance of two approximate null distributions
when the nominal size is $\alpha=5\%$. 
The vertical axis denotes $\widehat{\alpha}$ or $\widetilde{\alpha}$, 
and the horizontal axis denotes the value of $\mathcal{f}_m$.
The number attached to each line denotes the value of $\tau = h/k$.
The proposed approximation $\widehat{\alpha}$ is denoted by thick solid lines with triangles, and the existing approximation $\widetilde{\alpha}$ is denoted by thin dashed lines with circles. 
}
\label{fig:nullDist5pc}
\end{center} 
\end{figure}

\subsection{Results for Dependent Data}\label{appendix-timeSeries}
This is a supplement for Section~\ref{sec:compFeasibleComRule}.
If the data are not independent, then (\ref{eqt:prod_joint_density}) is no longer true.
In other words, $\overline{\loglik}(\psi) \not\equiv \overline{\loglik}^{\Stack}(\psi)$, 
where $\overline{\loglik}(\psi)=\sum_{\ell=1}^m \loglik^{(\ell)}(\psi)/m$ is defined in (\ref{eqt:jointDensity}), and 
\begin{equation}\label{eqt:prod_joint_density}
	\overline{\loglik}^{\Stack}(\psi)
		= \frac{1}{m} \log f_{mn}(X^{(1:m)}\mid \psi).
\end{equation}
In principle, $\overline{\loglik}(\psi)$ should be used instead of the ``stacked version'' $\overline{\loglik}^{\Stack}(\psi)$, 
however, the stacked one is much easier to compute. 
Because of this reason, it is of interest to see whether the stacked version can be used 
generally.

To begin with, we define the stacked version of all MI statistics when
$\overline{\loglik}^{\Stack}(\psi)$ is used instead of $\overline{\loglik}(\psi)$.
Let 
\begin{align}
	\widehat{\psi}_{0}^{\Stack}
		&= \argmax_{\psi\in\Psi\;:\; \theta(\psi)=\theta_0} \overline{\loglik}^{\Stack}(\psi),
	&
	\widehat{\psi}^{\Stack} 
		&= \argmax_{\psi\in\Psi} \overline{\loglik}^{\Stack}(\psi); &&\label{eqt:stackedMLE}  \\
	\widehat{\delta}_{0,\Stack} &= 2 \overline{\loglik}^{\Stack}(\widehat{\psi}_{0}^{\Stack}),
	   &
	\widehat{\delta}_{\Stack} &= 2 \overline{\loglik}^{\Stack}(\widehat{\psi}^{\Stack}).&& \label{eqt:def_deltaHat01S}
\end{align} 
and
\begin{align}
	\widehat{D}_{\Stack}(\mathcal{r}_m) 
		&= \frac{\widehat{d}_{\Stack}}{k(1+\mathcal{r}_m)}, 
		&& \text{with $\widehat{d}_{\Stack}=\widehat{\delta}_{\Stack}-\widehat{\delta}_{0,\Stack}$ 
			of (\ref{eqt:def_deltaHat01S})}; & \label{eqt:hatbarDm_harrcirc} \\
	\widehat{r}_{\Stack} 
		&= \frac{m+1}{k(m-1)}
			( \overline{d}_{\Stack} - \widehat{d}_{\Stack} ),
		&& \text{with $\overline{d}_{\Stack} = \overline{d}_{\lrt}$ 
			of (\ref{eqt:def_psiHat_bardd})}; & \label{eqt:def_hatrs} \\
	\widehat{r}_{\Stack}^{\rob}
		&= \frac{m+1}{h(m-1)} 
				( \overline{\delta}_{\Stack} - \widehat{\delta}_{\Stack}),  
		&&\text{with $\overline{\delta}_{\Stack}=\overline{\delta}_{\lrt}$ 
			of (\ref{eqt:def_deltabarhatL})}; & \label{eqt:def_rharrobCompeasy}
\end{align}
and $\widehat{r}^{+}_{\Stack} = \max(0,\widehat{r}_{\Stack})$. 
The stacked counterparts of $\widehat{D}^{\rob}_{\lrt}$ and 
its existing counterparts $\widehat{D}_{\lrt}$ and $\widehat{D}^{+}_{\lrt}$ 
(see (\ref{eqt:newstat})) then are given by 
\begin{equation}\label{eqt:newstat_compeasy}
	\widehat{D}_{\Stack}^{\rob}	= \widehat{D}_{\Stack}(\widehat{r}^{\rob}_{\Stack}), \qquad
	\widehat{D}_{\Stack} 			= \widehat{D}_{\Stack}(\widehat{r}_{\Stack}), \qquad
	\widehat{D}_{\Stack}^{+}	= \widehat{D}_{\Stack}(\widehat{r}^{+}_{\Stack}).
\end{equation}

The approximation 
$\widehat{d}_{\lrt} \bumpeq \widehat{d}_{\Stack}$ is still true
under the following conditions.

\begin{assumption}\label{ass:for_propApproxMLE}
\begin{enumerate}[noitemsep]
	\item[(a)] 
			Define $R(\psi) = \underline{\overline{\loglik}}^{\Stack}(\psi) - \underline{\overline{\loglik}}(\psi)$, where 
			$$\underline{\overline{\loglik}}(\psi) = (mn)^{-1} \sum_{\ell=1}^m \log f(X^{(\ell)}\mid \psi)
			\quad{\rm and}\quad \underline{\overline{\loglik}}^{\Stack}(\psi) = (mn)^{-1} \log f(X^{\Stack}\mid \psi).$$
			For each $m$, as $n\rightarrow\infty$, 
			\[
				\sup_{\psi\in\Psi} \left\vert R(\psi) \right\vert =O_p(1/n),
				\qquad 
				\sup_{\psi\in\Psi} \left\vert \frac{\partial}{\partial \psi} R(\psi) \right\vert =O_p(1/n).
			\]
	\item[(b)] 
			For each $m$, there exists a continuous function 
			$\psi\mapsto\overline{\underline{\mathcal{L}}}(\psi)$, which is free of $n$
			but may depend on $m$, 
			such that, as $n\rightarrow\infty$,
			\[
				\sup_{\psi\in\Psi} \left\vert \underline{\overline{\loglik}}(\psi) 
					- \overline{\underline{\mathcal{L}}}(\psi) \right\vert = o_p(1).
			\]
	\item[(c)] 
			Let $\psi_{0}^{*} = \argmax_{\psi\in\Psi\;:\;\psi(\theta)=\theta_0}\overline{\underline{\mathcal{L}}}(\psi)$
			and $\psi_{}^{*} = \argmax_{\psi\in\Psi}\overline{\underline{\mathcal{L}}}(\psi)$.
			For any fixed $m$, and for all $\varepsilon>0$, there exists $\delta>0$ such that 
			\begin{equation*}
				\sup_{\substack{\psi\in\Psi\,:\,|\psi_{0}^{*}-\psi|>\varepsilon\\ \theta(\psi)=\theta_0}} 
				\left\{ \overline{\underline{\mathcal{L}}}(\psi_{0}^{*})
					 - \overline{\underline{\mathcal{L}}}(\psi) \right\} \geq \delta, 
				\quad 
				\sup_{\psi\in\Psi\,:\,|\psi_{}^{*}-\psi|>\varepsilon} 
				\left\{ \overline{\underline{\mathcal{L}}}(\psi_{}^{*})
					 - \overline{\underline{\mathcal{L}}}(\psi) \right\} \geq \delta.
			\end{equation*}
\end{enumerate}
\end{assumption}

Conditions (b) and (c) in Assumption \ref{ass:for_propApproxMLE} 
are standard RCs that are usually assumed for M-estimators 
(see Section~5 of \cite{vanDerVaart2000}); 
whereas condition (a) is satisfied by many models 
(see Example \ref{eg:timeSeries} below).

\begin{theorem}\label{prop:approxMLE}
Suppose $\RC_\theta$ and Assumption \ref{ass:for_propApproxMLE}.
Under both $H_0$ and $H_1$, 
we have 
(i) $\widehat{d}_{\Stack},\widehat{r}_{\Stack}\geq 0$ for all $m,n$; 
(ii) $\widehat{d}_{\Stack},\widehat{r}_{\Stack}$ are 
invariant to the parametrization of $\psi$ for all $m,n$; and
(iii) $\widehat{d}_{\lrt} \bumpeq \widehat{d}_{\Stack}$
and $\widehat{r}_{\lrt} \bumpeq \widehat{r}_{\Stack}$
as $n\rightarrow\infty$ for each $m$.
\end{theorem}

Theorem \ref{prop:approxMLE} implies that 
the handy test statistics $\widehat{D}_{\Stack}$ and $\widehat{D}_{\Stack}^{+}$
approximate $\widehat{D}_{\lrt}$ and $\widehat{D}_{\lrt}^{+}$
for dependent data, provided that Assumption \ref{ass:for_propApproxMLE} holds.

\begin{example}\label{eg:timeSeries}
Consider a stationary autoregressive model of order one.
Suppose the complete data $X=(X_1,\ldots,X_n)^{\T}$
is generated as following: 
$X_1\sim\mathcal{N}(0,v^2)$ 
and $[ X_i | X_{i-1} ]\sim\mathcal{N}(\phi X_{i-1}, \sigma^2)$
for $i\geq 2$, 
where $v^2 = {\sigma^2(1+\phi)}/{(1-\phi)}$.
Then $\psi=(\phi,\sigma^2)^{\T}$, and 
\begin{eqnarray*}
	\underline{\overline{\loglik}}(\psi)
		&=&  -\frac{1}{2}\log(2\pi)
				- \frac{1}{2n}\log v^2
				- \frac{1}{mn}\sum_{\ell=1}^m \frac{X_1^{(\ell)}}{2v^2}
				- \frac{n-1}{2n}\log \sigma^2 \\
		&&		- \frac{1}{mn}\sum_{\ell=1}^m\sum_{i=2}^n \frac{(X_i^{(\ell)}-\phi X_{i-1}^{(\ell)})^2}{2\sigma^2} , \\
	\underline{\overline{\loglik}}^{\Stack}(\psi)	
		&=&  -\frac{1}{2}\log(2\pi)
			- \frac{1}{2mn} \log v^2
			- \frac{(X_1^{(1)})^2}{2mnv^2}
			- \frac{mn-1}{2mn}\log \sigma^2 \\ 
		&&	- 
			\frac{1}{mn}\sum_{\ell=1}^m\sum_{i=2}^n \frac{(X_i^{(\ell)}-\phi X_{i-1}^{(\ell)})^2}{2\sigma^2}
			- \frac{1}{mn}\sum_{\ell=2}^m \frac{(X_1^{(\ell)}-\phi X_{n}^{(\ell-1)})^2}{2\sigma^2}.
\end{eqnarray*}
Then, it is easy to see that 
condition (a) of Assumption \ref{ass:for_propApproxMLE} is satisfied.
\end{example}

\subsection{Other existing MI tests}\label{sec:otherMItests}

First, we list some existing estimators of $\mathcal{r}_m$.
Let $s_{\wt,a}^2$ be the sample variances of 
$\{ ( d^{(\ell)}_{\wt})^a \}_{\ell=1}^m$ for $a>0$.
\cite{rubin2004multiple} and \cite{li91} proposed 
\begin{eqnarray}
	\widetilde{r}_{\wt,1} 
		&=& \frac{(1+1/m)s_{\wt,1}^2}
			{2\overline{d}_{\wt} + \sqrt{ \max\left\{ 0, 4 \overline{d}^2_{\wt} - 2 k s_{\wt,1}^2\right\}}},\label{eqt:otherEst_of_r_1}\\
	\widetilde{r}_{\wt,1/2}  
		&=& (1+1/m) s_{\wt,1/2}^2 ,\label{eqt:otherEst_of_r_2}
\end{eqnarray}
respectively.
When $k$ is large and $m$ is small, 
using (\ref{eqt:otherEst_of_r_1}) or (\ref{eqt:otherEst_of_r_2}) may lead to power loss. 
A trivial modification of $\widetilde{r}_{\lrt}$ of (\ref{eqt:rL}), i.e., $\widetilde{r}^{+}_{\lrt} = \max(0,\widetilde{r}_{\lrt})$, is a better alternative.

Second, we list some alternative MI combining rules. 
Having the above estimators of $\mathcal{r}_m$, 
we can insert them into the following combining rules:
\begin{equation}\label{eqt:def_D_as_function_of_r}
	\widetilde{D}_{\wt}^{\prime}(\mathcal{r}_m) = \frac{\widetilde{d}_{\wt}^{\prime}}{k(1+\mathcal{r}_m)}, \;
	\widetilde{D}_{\lrt}(\mathcal{r}_m) = \frac{\widetilde{d}_{\lrt}}{k(1+\mathcal{r}_m)} , \;
	\widetilde{D}^{+}_{\lrt}(\mathcal{r}_m) = \left\{ \widetilde{D}_{\lrt}(\mathcal{r}_m)\right\}^+.
\end{equation}
Using (\ref{eqt:def_of_rm}) and (\ref{eqt:rL}), 
we can also define the following combining rules:
\begin{equation}\label{eqt:barD_combineRule}
	\overline{D}_{\wt}^{\prime}(\mathcal{r}_m) 
		= \frac{\overline{d}_{\wt}^{\prime}-\frac{k(m-1)}{m+1}\mathcal{r}_m}{k(1+\mathcal{r}_m)},
	\qquad 
	\overline{D}_{\lrt}(\mathcal{r}_m) 	
		= \frac{\overline{d}_{\lrt}-\frac{k(m-1)}{m+1}\mathcal{r}_m}{k(1+\mathcal{r}_m)};
\end{equation}
see, e.g., \cite{li91}.
The combining rule $\overline{D}_{\wt}^{\prime}(\mathcal{r}_m)$ is useful
when  
computing $\overline{d}_{\wt}^{\prime}$ and estimating $\mathcal{r}_m$ are simple,
but the resulting power may deteriorate.
If $\widetilde{r}_{\wt,1}$ or $\widetilde{r}_{\wt,1/2}$ is used
for estimating $\mathcal{r}_m$,
the null distribution of (\ref{eqt:def_D_as_function_of_r}) and (\ref{eqt:barD_combineRule})
can be approximated by $F_{k,\widetilde{\df}'(\mathcal{r}_m,k)}$,
where $\widetilde{\df}'(\mathcal{r}_m, k)=(m-1)(1+\mathcal{r}_m^{-1})^2k^{-3/m}$; see \cite{li91}.

Next, we introduce and recall some notation:  
(a) standard complete-data moments estimation ($\mathcal{M}_{\wt}$, $\mathcal{M}_{\lrt}$) and testing procedures ($\mathcal{D}_{\wt}$, $\mathcal{D}_{\lrt}$), and 
(b) non-standard complete-data procedures ($\widetilde{\mathcal{D}}_{\lrt}$, $\overline{\mathcal{D}}_{\lrt}$, $\mathcal{D}_{\lrt,1}$, $\overline{\mathcal{D}}_{\lrt,1}$), where 
\begin{gather*}		
	\mathcal{M}_{\wt}(X) = \left\{ \widehat{\theta}(X), U(X) \right\}, \quad
	\mathcal{M}_{\lrt}(X) = \left\{ \widehat{\psi}(X) , \widehat{\psi}_0(X) \right\} ,  \\
	\mathcal{D}_{\wt}(X) = d_{\wt}(\widehat{\theta}(X),U(X)) , \quad
	\overline{\mathcal{D}}_{\lrt,1}(\mathbb{X}) 
		= \frac{2}{m} \sum_{\ell=1}^m \log f(X^{(\ell)}\mid\widehat\psi^*(\mathbb{X})), \\
	\overline{\loglik}(\psi)
		= \frac{1}{m} \log f_{mn}(X^{(1:m)}\mid \psi).\label{eqt:prod_joint_density}
\end{gather*}

Table \ref{table:prosCons_long} is the full version of Table \ref{table:prosCons} in the main text. 
It summarizes the statistical and computational properties of different MI tests; 
see Section \ref{sec:comparison} for details.

\begin{landscape} 
\begin{table}
\caption{\small Computational requirements and statistical properties of 
MI test statistics,
their associated combining rules and estimators of FMI $\mathcal{r}_m$.
The symbols ``$+$'' and ``$-$'' mean that the test statistic (or estimator)
is equipped and not equipped with the indicated property, respectively;
see the end of Section~\ref{sec:comparison} for heading descriptions.
The reference papers/book are abbreviated as follows: 
\cite{rubin2004multiple} (R04), \cite{li91} (LMRR91) and \cite{mengRubin92} (MR92).}
\begin{minipage}{20cm}
\setlength{\tabcolsep}{3.5pt}
\renewcommand{\arraystretch}{0.5}
\scriptsize
\begin{tabular}{cccc cccc c cccc ccHc} \toprule
\multicolumn{2}{c}{} & 
\multicolumn{2}{c}{Combining Rule} & 
\multicolumn{2}{c}{Estimator of $\mathcal{r}_m$} & 
\multicolumn{2}{c}{Approx. null distribution\footnote{In actual computation, the $\mathcal{r}_m$ in the denominator degree of freedom of $F$ is replaced by its corresponding estimator.}} 
&
\multicolumn{1}{c}{} & 
\multicolumn{8}{c}{Properties} \\
\cmidrule(r){3-4}\cmidrule(r){5-6}\cmidrule(r){7-8}
\cmidrule(r){10-17}
Test & No. & Formula & Routine & Formula & Routine & Original & Proposed & Reference 
& Inv & Con & $\geq 0$ & Pow & Def & Sca & Dep & EFMI \\
\cmidrule(r){1-17}
WT & WT-1 & 
 $D_{\wt}(T)$\footnote{Computing the test statistic $D_{\wt}(T) = d_{\wt}(\overline{\theta},T)/k$ does not require estimating $\mathcal{r}_m$.} & $\mathcal{M}_{\wt}$ & 
 $\widetilde{r}'_{\wt}$ & $\mathcal{D}_{\wt}$ &
 $F_{k, \widetilde{\df}(\mathcal{r}_m,k)}$ & $F_{k, \widehat{\df}(\mathcal{r}_m,k)}$ &
R04 & 
$-$ & $+$ & $+$ & $-$ & $-$ & $-$ & $+$ & $\theta$\footnote{EFMI is not required for the test statistic $D_{\wt}(T)$, but it is required for its approximate null distribution.} \\
 & WT-2 & 
$\widetilde{D}'_{\wt}(\mathcal{r}_m)$ & $\mathcal{M}_{\wt}$ & 
$\widetilde{r}'_{\wt}$ & $\mathcal{M}_{\wt}$ & 
$F_{k, \widetilde{\df}(\mathcal{r}_m,k)}$\footnote{The approximate null distribution documented in \cite{rubin2004multiple} was modified by \cite{li91}. This also applies to WT-2,4,5.} & $F_{k, \widehat{\df}(\mathcal{r}_m,k)}$ &
R04 & 
$-$ & $+$\footnote{The estimator $\widetilde{r}'_{\wt}$ does not depend on $\theta_0$, but its MSE may be inflated under $H_1$ if a bad parametrization of $\theta$ is used.} & $+$ & $-$ & $-$ & $-$ & $+$ & $\theta$ \\
 & WT-3 &  
 $\widetilde{D}'_{\wt}(\mathcal{r}_m)\footnote{The originally proposed combining rule is $\overline{D}'_{\wt}(\mathcal{r}_m)$; see (\ref{eqt:barD_combineRule}). Although $\overline{D}'_{\wt}(\mathcal{r}_m)$ is more computational feasible, the power loss is more significant than $\widetilde{D}'_{\wt}(\mathcal{r}_m)$ after inserting an inefficient estimator $\widetilde{r}'_{\wt,1}$ for $\mathcal{r}_m$. This footnote also applies to WT-3.}$ & $\mathcal{M}_{\wt}$  & 
 $\widetilde{r}'_{\wt,1}$ & $\mathcal{D}_{\wt}$ &
$F_{k, \widetilde{\df}'(\mathcal{r}_m,k)}$ & NA &
R04 & 
$-$ & $-$ & $+$ & $-$ & $-$ & $-$ & $+$ & $\theta$ \\  
 & WT-4 &  
 $\widetilde{D}'_{\wt}(\mathcal{r}_m)$ & $\mathcal{M}_{\wt}$  & 
 $\widetilde{r}'_{\wt,1/2}$ & $\mathcal{D}_{\wt}$ &
 $F_{k, \widetilde{\df}'(\mathcal{r}_m,k)}$ & NA &
LMRR91 & 
$-$ & $-$ & $+$ & $-$ & $-$ & $-$ & $+$ & $\theta$ \\
 & WT-5 &  
 $\overline{D}'_{\wt}(\mathcal{r}_m)$ & $\mathcal{D}_{\wt}$  & 
 $\widetilde{r}'_{\wt,1}$ & $\mathcal{D}_{\wt}$ &
 $F_{k, \widetilde{\df}'(\mathcal{r}_m,k)}$ & NA &
R04 & 
$-$ & $-$ & $-$ & $-$ & $-$ & $+$ & $+$ & $\theta$ \\
 & WT-6 &  
 $\overline{D}'_{\wt}(\mathcal{r}_m)$ & $\mathcal{D}_{\wt}$  & 
 $\widetilde{r}'_{\wt,1/2}$ & $\mathcal{D}_{\wt}$ &
 $F_{k, \widetilde{\df}'(\mathcal{r}_m,k)}$ & NA &
LMRR91 & 
$-$ & $-$ & $-$ & $-$ & $-$ & $+$ & $+$ & $\theta$ \\
\cmidrule(r){1-17}
LRT & LRT-1 & 
$\widetilde{D}_{\lrt}(\mathcal{r}_m)$ & $\mathcal{M}_{\lrt},\widetilde{\mathcal{D}}_{\lrt}$ & 
$\widetilde{r}_{\lrt}$ & $\mathcal{M}_{\lrt},\widetilde{\mathcal{D}}_{\lrt}$ &
$F_{k, \widetilde{\df}(\mathcal{r}_m,k)}$ & $F_{k, \widehat{\df}(\mathcal{r}_m,k)}$ &
MR92 & 
$-$ & $-$ & $-$ & $-$ & $+$ & $-$\footnote{Averaging and processing vector estimators of $\psi$, but not their covariance matrixes, is needed. This footnote also applies to LRT-2.} & $+$ & $\theta$ \\
 & LRT-2 & 
 $\widehat{D}_{\lrt}(\mathcal{r}_m)$ & $\mathcal{D}_{\lrt}$ & 
 $\widehat{r}^+_{\lrt}$ & $\mathcal{D}_{\lrt}$ &
 $F_{k, \widetilde{\df}(\mathcal{r}_m,k)}$ & $F_{k, \widehat{\df}(\mathcal{r}_m,k)}$ &
Proposal & 
$+$ & $-$ & $+$ & $-$ & $+$ & $+$ & $+$ & $\theta$ \\
 & LRT-3 & 
 $\widehat{D}_{\lrt}(\mathcal{r}_m)$ & $\mathcal{D}_{\lrt}$ & 
 $\widehat{r}^{\rob}_{\lrt}$ & $\mathcal{D}_{\lrt,1}$ &
 $F_{k, \widetilde{\df}(\mathcal{r}_m,h)}$ & $F_{k, \widehat{\df}(\mathcal{r}_m,h)}$ &
Proposal & 
$+$ & $+$ & $+$ & $+$ & $+$ & $+$ & $+$ & $\psi$ \\
 & LRT-4 & 
 $\widetilde{D}^+_{\lrt}(\mathcal{r}_m)$ & $\mathcal{M}_{\lrt},\widetilde{\mathcal{D}}_{\lrt}$ & 
 $\widetilde{r}^+_{\lrt}$ & $\mathcal{M}_{\lrt},\widetilde{\mathcal{D}}_{\lrt}$ &
 $F_{k, \widetilde{\df}(\mathcal{r}_m,k)}$ & $F_{k, \widehat{\df}(\mathcal{r}_m,k)}$ &
MR92\footnote{It is a trivial modification of the original proposal in MR92 by replacing $\widetilde{r}_{\rm L}$ with $\widetilde r_{\rm L}^{+}=\max\{0, \widetilde{r}_{\rm L}\}$.} & 
$-$ & $-$ & $+$ & $-$ & $+$ & $-$ & $+$ & $\theta$ \\
 & LRT-5 & 
 $\widehat{D}_{\lrt}(\mathcal{r}_m)$ & $\mathcal{D}_{\lrt}$ &
 $\widehat{r}_{\lrt}$ & $\mathcal{D}_{\lrt}$ &
 $F_{k, \widetilde{\df}(\mathcal{r}_m,k)}$ & $F_{k, \widehat{\df}(\mathcal{r}_m,k)}$ &
Proposal & 
$+$ & $-$ & $-$ & $-$ & $+$ & $+$ & $-$
& $\theta$ \\
\bottomrule
\end{tabular}
\label{table:prosCons_long}
\end{minipage}
\end{table}
\end{landscape}

\subsection{Supplement for Section~\ref{sec:eg_lrt_normal}}\label{sec:eg_lrt_normal_MI}
Let $\overline{X}_{\obs}$ and $S_{\obs}$ be the 
sample mean and sample covariance matrix based on ${X}_{\obs}$.
Then, the $\ell$th imputed missing data set 
can be produced by the following procedure, 
for $\ell=1,\ldots,m$.
\begin{enumerate}[noitemsep]
	\item Draw $(\Sigma^{(\ell)})^{-1}$ from 
			a Wishart distribution 
			with $(n_{\obs}-1)$ degrees of freedom and scale matrix $S_{\obs}$.
	\item Draw $\mu^{(\ell)}$ from  
			$\mathcal{N}_p(\overline{X}_{\obs}, \Sigma^{(\ell)}/n_{\obs})$.
	\item 
			Draw $(n-n_{\obs})$ imputed missing values
			$\{X^{(\ell)}_{i}:i=n_{\obs}+1,\ldots,n\}$ from 
			$\mathcal{N}_p(\mu^{(\ell)}, \Sigma^{(\ell)})$ independently.
\end{enumerate}
Also, denote $X^{(\ell)}_{i} = X_i$ for $i=1,\ldots,n_{\obs}$.
With the $\ell$th completed data set,
the unconstrained MLEs for $\mu$ and $\Sigma$ are 
\begin{equation*}
	\widehat{\mu}^{(\ell)} 
		= \frac{1}{n} \sum_{i=1}^n X_{i}^{(\ell)},\qquad 
	\widehat{\Sigma}^{(\ell)} 
		= \frac{1}{n} \sum_{i=1}^{n} \left(X_{i}^{(\ell)}-\widehat{\mu}^{(\ell)}\right)\left(X_{i}^{(\ell)}-\widehat{\mu}^{(\ell)}\right)^{\T}.
\end{equation*}
Whereas we generate data using a covariance matrix with common variance and correlation, 
our model does not assume any structure for $\Sigma$. The only restriction we can impose is the common-mean assumption under the null, 
for which the constrained MLEs are 
\begin{equation*}
	\widehat{\mu}^{(\ell)}_0
    	= \left\{
				\frac{\mathbf{1}_p^{\T}(\widehat{\Sigma}^{(\ell)})^{-1}\widehat{\mu}^{(\ell)}}
        				{\mathbf{1}_p^{\T}(\widehat{\Sigma}^{(\ell)})^{-1}\mathbf{1}_p} 
			\right\} \mathbf{1}_p ,\qquad 
	\widehat{\Sigma}^{(\ell)}_0
    	= \widehat{\Sigma}^{(\ell)} 
        	+ 	\left( \widehat{\mu}^{(\ell)} - \widehat{\mu}^{(\ell)}_0 \right)
				\left( \widehat{\mu}^{(\ell)} - \widehat{\mu}^{(\ell)}_0 \right)^{\T}.
\end{equation*}

\begin{table*}[t!]
\begin{center}
\renewcommand{\arraystretch}{0.5}
\caption{\small The values of parameters used in the simulation experiment in Section~\ref{sec:eg_lrt_normal}.}
\begin{tabular}{cc C{.9cm}C{.9cm}C{.9cm} ccccc}\toprule
\multicolumn{2}{c}{\bf Experiment}	& 
\multicolumn{3}{c}{\bf Fixed Parameters} 	& 
\multicolumn{5}{c}{\bf Variable Parameter}  \\
\cmidrule(r){1-2}\cmidrule(r){3-5}\cmidrule(r){6-10}
No. & Variable Parameter	& $\rho$ & $p$ & $\mathcal{f}$ & Case 1 & Case 2 & Case 3 & Case 4 & Case 5 	\\
\cmidrule(r){1-10}
	I & Correlation	$\rho$ 	&	-- 		& $2$ 	& $0.5$	& $-0.8$ & $-0.4$ & $0$ 	& $0.4$ 	& $0.8$ 	 \\
	II & Dimension $p$ 			&	$0.4$ 	& -- 	& $0.5$ 	& $2$ 	& $3$ 	& $4$		& $5$ 	& $6$		 \\
	III & FMI $\mathcal{f}$  	&	$0.4$ 	& $2$ 	& --	 	& $0.1$ 	& $0.3$ 	& $0.5$ 	& $0.7$ 	& $0.9$ 	 	\\
\bottomrule
\end{tabular}
\label{table:experimentParaVal}
\end{center}
\end{table*}
\noindent

We first study the distribution of 
$p$-values of each test under $H_0$.
We use $n=100$, $m=3$, $\sigma^2=5$ and $\mu=\mathbf{1}_p$, 
with various values of $\rho$, $p$ and $\mathcal{f}$ specified in Table \ref{table:experimentParaVal}.
The results under parametrizations (i), (ii) and (iii) are shown in Figures \ref{fig:pVal1}, 
\ref{fig:pVal2} and \ref{fig:pVal3}, respectively. 
Note that, for Wald tests under parametrization (ii), 
the matrix $U^{(\ell)}$ is singular in $0.25\%$ of the replications, and those cases are removed from the analysis (which should favor the Wald tests).

\begin{figure}[h!]
\begin{center}
\includegraphics[width=\textwidth]{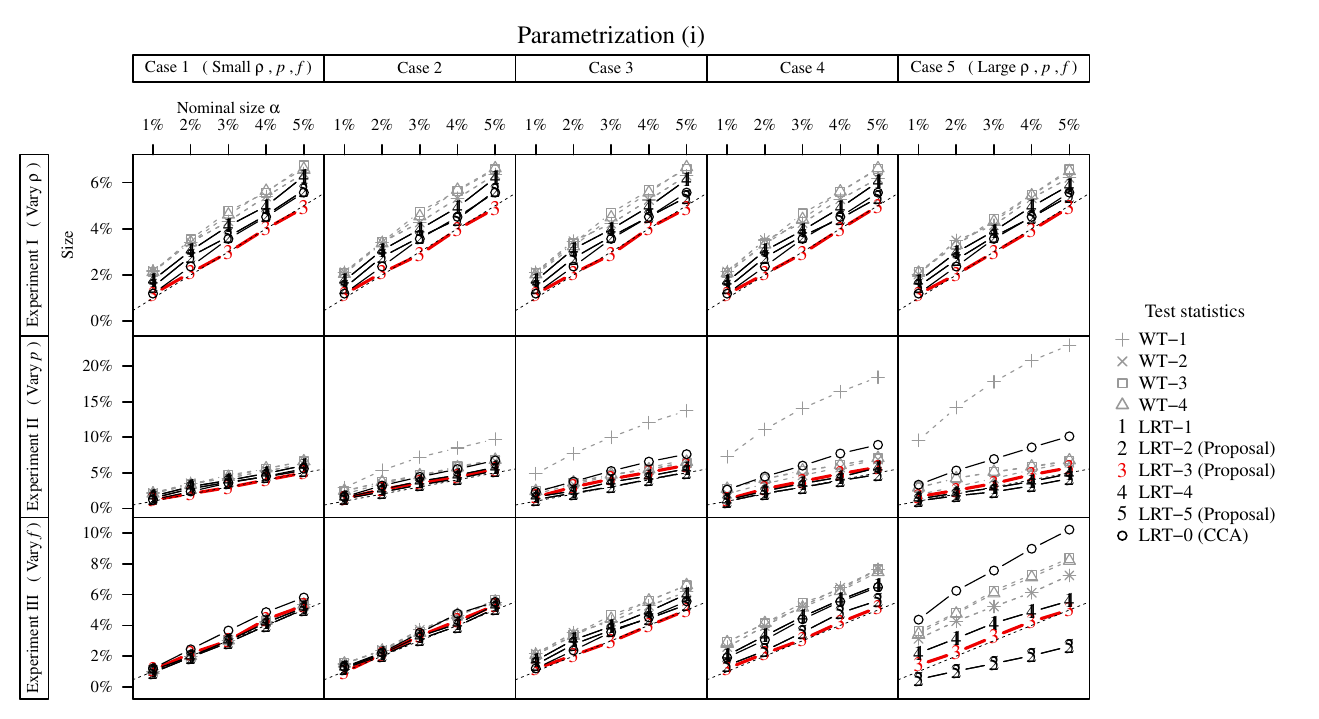}
\caption{\small 
The comparison between empirical size and nominal size $\alpha$
under parametrization (ii) for $\alpha\in (0,5\%]$.
Our most recommended proposal is LRT-3, which is highlighted red. 
}
\label{fig:pVal1}
\end{center} 
\end{figure}

\begin{figure}[h!]
\begin{center}
\includegraphics[width=\textwidth]{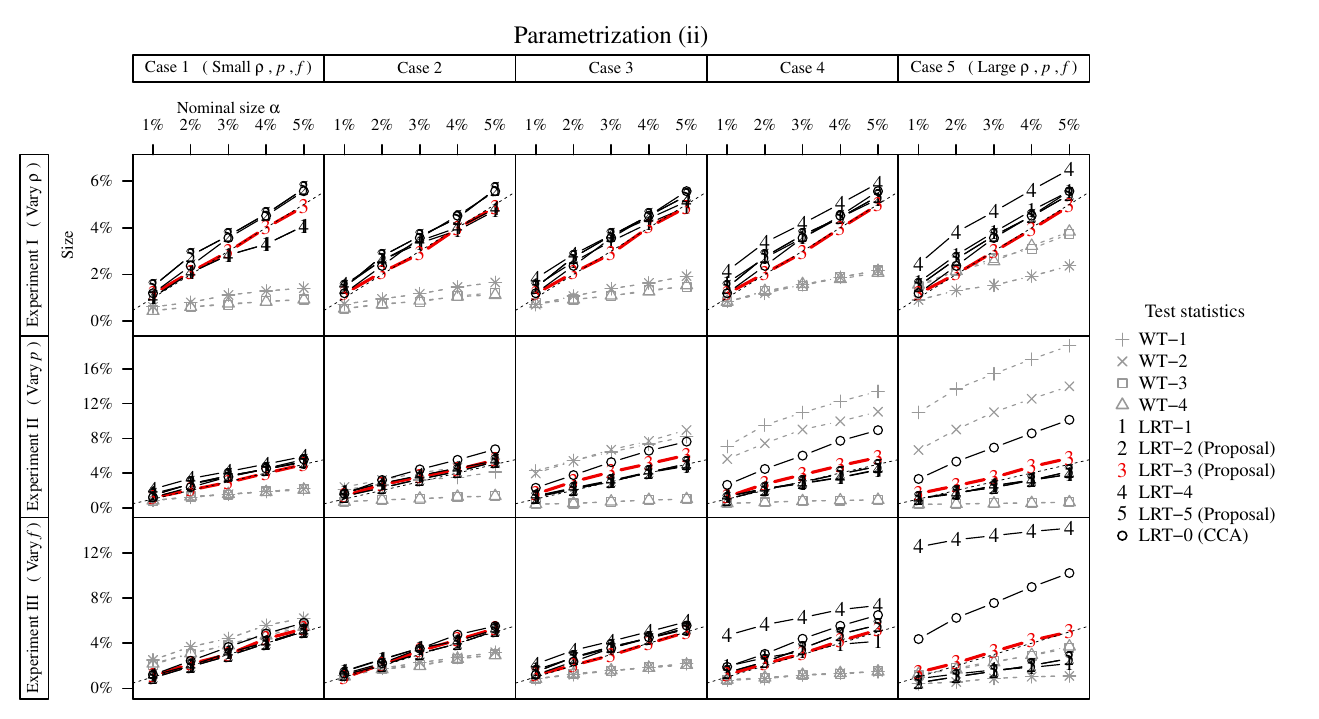}
\caption{\small 
The comparison between empirical size and nominal size $\alpha$
under parametrization (i) for $\alpha\in (0,5\%]$.
The legend in Figure \ref{fig:pVal1} also applies here. 
}
\label{fig:pVal2}
\end{center} 
\end{figure}

\begin{figure}[h!]
\begin{center}
\includegraphics[width=\textwidth]{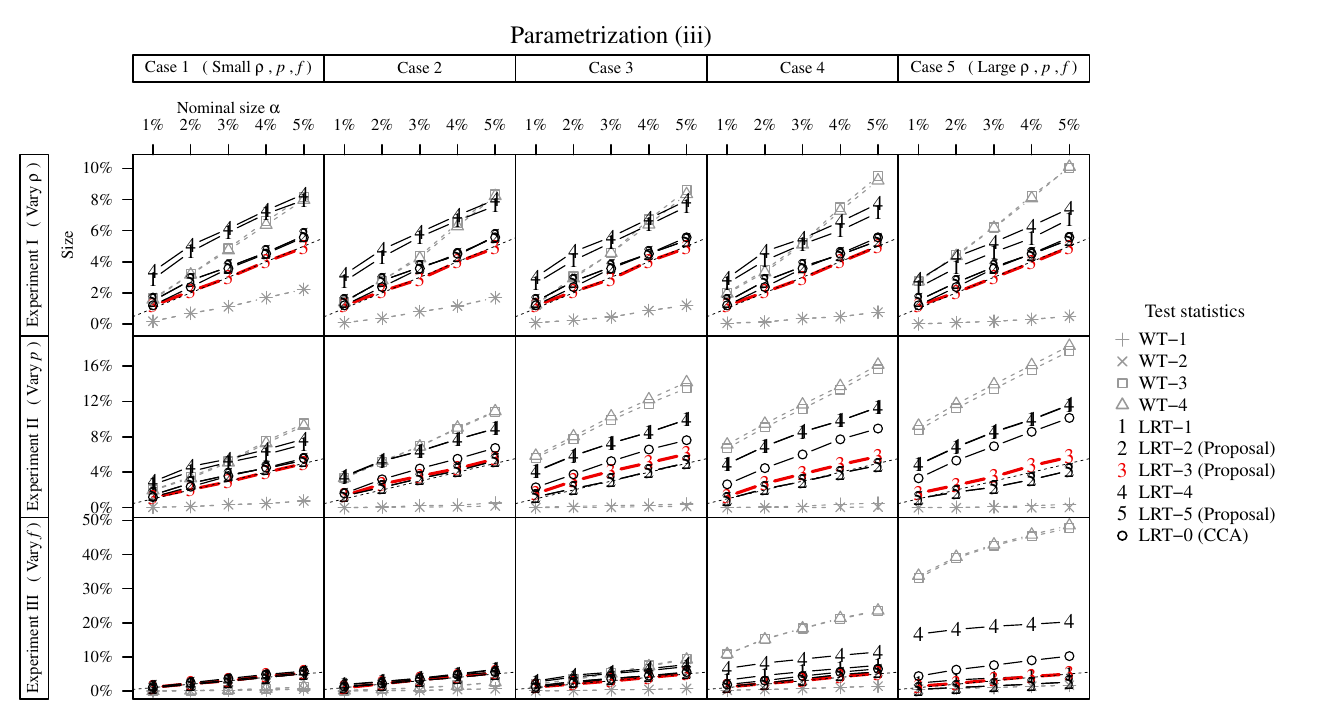}
\caption{\small 
The comparison between empirical size and nominal size $\alpha$
under parametrization (iii) for $\alpha\in (0,5\%]$.
The legend in Figure \ref{fig:pVal2} also applies here. 
}
\label{fig:pVal3}
\end{center} 
\end{figure}

The empirical sizes (i.e., type-I errors) of the MI Wald tests generally deviate from the nominal size $\alpha$ under parametrization (ii). In contrast, the sizes of all LRTs are closer to $\alpha$.
However, the original L-1 and its trivial modification L-2 
do not have accurate sizes when $|\rho|$ or $\mathcal{f}$ is large.
They can be over-sized or under-sized depending on which parametrization is used. 
Moreover, the trivial modification L-2 does not help to correct the size, 
and it may even worsen the test. 
For our test statistics 
L-3 and L-4, 
they are invariant to parametrizations
and have quite accurate sizes, although they are under-sized 
in challenging cases where both $p$ and $\mathcal{f}$ are large.
For our recommended statistic L-5, 
it gives the most satisfactory overall results. 
It generally has very accurate size, except that it is slightly over-sized for large $p$, a problem that should diminish when we use $m$ beyond the smallest recommended $m=3$. 

Interestingly, as seen clearly in Figure~\ref{fig:pVal2},  the benchmark L-0 performs very badly for large $p$ and $\mathcal{f}$. This is because the sample size per parameter, $n/h$, is small; for $p\ge 4$, $n/h\le 100/14<8$. The asymptotic null distribution $\chi^2_k/k$ then can fail badly under arbitrary or even all parametrizations; (ii) apparently falls into this category. An $F$ approximation would be more appropriate \citep[see][]{BarnardRubin1999}. But this is exactly what is being used for MI tests, albeit with different
choices of the denominator degrees of freedom. Note also that, in some cases, nearly half of the simulated values of 
$\widetilde{r}_{\lrt}$ and $\widetilde{D}_{\lrt}$ are negative;
see Table \ref{table:propNeg}. 
In contrast, $\widehat{r}_{\Stack}$
is always non-negative in our simulation,
despite the fact that it can be negative in theory.

\begin{table*}[t!]
\renewcommand{\arraystretch}{0.5}
\begin{center}
\caption{\small The empirical proportions 
of negative $\widetilde{r}_{\lrt}$ and $\widetilde{D}_{\lrt}$.
The results under parametrizations (ii) and (iii) are shown.
For parametrization (i), 
$\widetilde{r}_{\lrt} \geq 0$ and $\widetilde{D}_{\lrt}\geq 0$ in the experiments.}
\begin{tabular}{cccccccccccc}\toprule
&& \multicolumn{10}{c}{\bf Case}\\
\cmidrule(r){3-12}
&&1 & 2 & 3 & 4 & 5 & 1 & 2 & 3 & 4 & 5 \\
\cmidrule(r){3-7}\cmidrule(r){8-12}
{\bf Experiment} & {\bf Parametrization}  & \multicolumn{5}{c}{\% of $\widetilde{r}_{\lrt}< 0$} 
		& \multicolumn{5}{c}{\% of $\widetilde{D}_{\lrt}< 0$} \\[0.5ex]
\cmidrule(r){1-1}\cmidrule(r){2-2}\cmidrule(r){3-7}\cmidrule(r){8-12}
I&	(ii)	&	$ 1$	&	$ 2$	&	$ 3$	&	$ 4$	&	$ 5$	&	$26$	&	$16$	&	$13$	&	$12$	&	$12$	\\[0.5ex]
&	(iii)	&	$ 6$	&	$ 6$	&	$ 7$	&	$ 7$	&	$ 7$	&	$ 1$	&	$ 1$	&	$ 1$	&	$ 1$	&	$ 2$	\\[0.5ex]
\cmidrule(r){1-1}\cmidrule(r){2-2}\cmidrule(r){3-7}\cmidrule(r){8-12}
II	&(ii)	&	$ 4$	&	$ 1$	&	$ 0$	&	$ 0$	&	$ 0$	&	$12$	&	$ 5$	&	$ 3$	&	$ 4$	&	$ 3$	\\[0.5ex]
	&(iii)	&	$ 7$	&	$ 3$	&	$ 1$	&	$ 1$	&	$ 1$	&	$ 1$	&	$ 0$	&	$ 0$	&	$ 0$	&	$ 0$	\\[0.5ex]
\cmidrule(r){1-1}\cmidrule(r){2-2}\cmidrule(r){3-7}\cmidrule(r){8-12}
III	&(ii)	&	$13$	&	$ 6$	&	$ 4$	&	$ 4$	&	$ 3$	&	$55$	&	$25$	&	$12$	&	$ 5$	&	$ 2$	\\[0.5ex]
	&(iii)	&	$18$	&	$ 9$	&	$ 7$	&	$ 5$	&	$ 4$	&	$20$	&	$ 5$	&	$ 1$	&	$ 1$	&	$ 0$	\\[0.5ex]
\bottomrule
\end{tabular}
\label{table:propNeg}
\end{center}
\end{table*}

\begin{figure}[h!]
\begin{center}
\includegraphics[width=\textwidth]{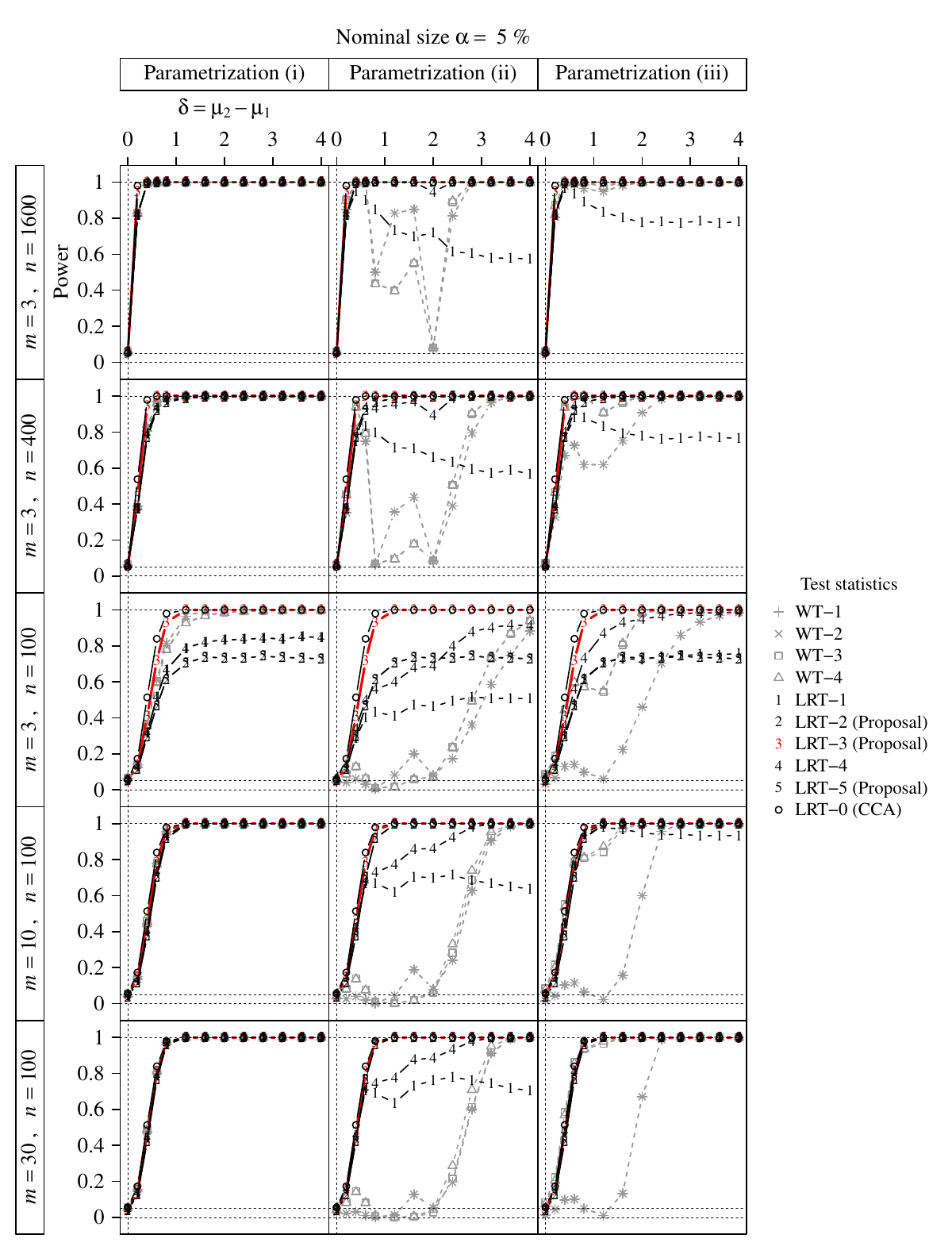}
\caption{\small 
The power curves under different parametrizations.
The nominal size is $\alpha=5\%$.
In each plot, the vertical axis denotes the power, 
whereas the horizontal axis denotes the value of $\delta=\mu_2-\mu_1$.
The legend in Figure \ref{fig:pVal2} also applies here. 
}
\label{fig:power_5pc}
\end{center} 
\end{figure}

The power curves under nominal size $0.5\%$ and $5\%$ are shown in Figure \ref{fig:power_05pc} of the 
main text
and Figure \ref{fig:power_5pc}, respectively. 
Note that 
the trivial modifications LRT-2 of LRT-1 cannot retrieve all the power it should have. 
Tables \ref{table:sizeError_05} and \ref{table:sizeError_5} show the minimum and maximum of the empirical sizes over the three parametrizations considered in each test --- and only one value is needed for those tests that are invariant to parametrization --- 
when the nominal size is $0.5\%$ and $5\%$, respectively.
We see the deviations from the nominal $\alpha$ can be noticeable, especially when $m=3$. To take that into account,  we report the empirical size adjusted power, that is, $O={\rm power}/\widehat{\alpha}$, which also has the interpretation as (an approximated) posterior odds of $H_1$ to $H_0$ \citep{BBBS2015}. 
Figures \ref{fig:powerOverSize_05pc} and \ref{fig:powerOverSize_5pc} plot 
the result for nominal size $0.5\%$ and $5\%$, respectively. 
Compared with the benchmark L-0,
the odds $O$ of the proposed robust MI test (L-5) 
is closer to the nominal value $1/\alpha$ as $\delta\rightarrow \infty$. 
Nevertheless, the performances of all size $0.5\%$ tests are less
satisfactory than those for size $5\%$ tests because 
larger sample sizes $n$ are required to approximate the tail behavior well.

\begin{figure}[t!]
\begin{center}
\includegraphics[width=\textwidth]{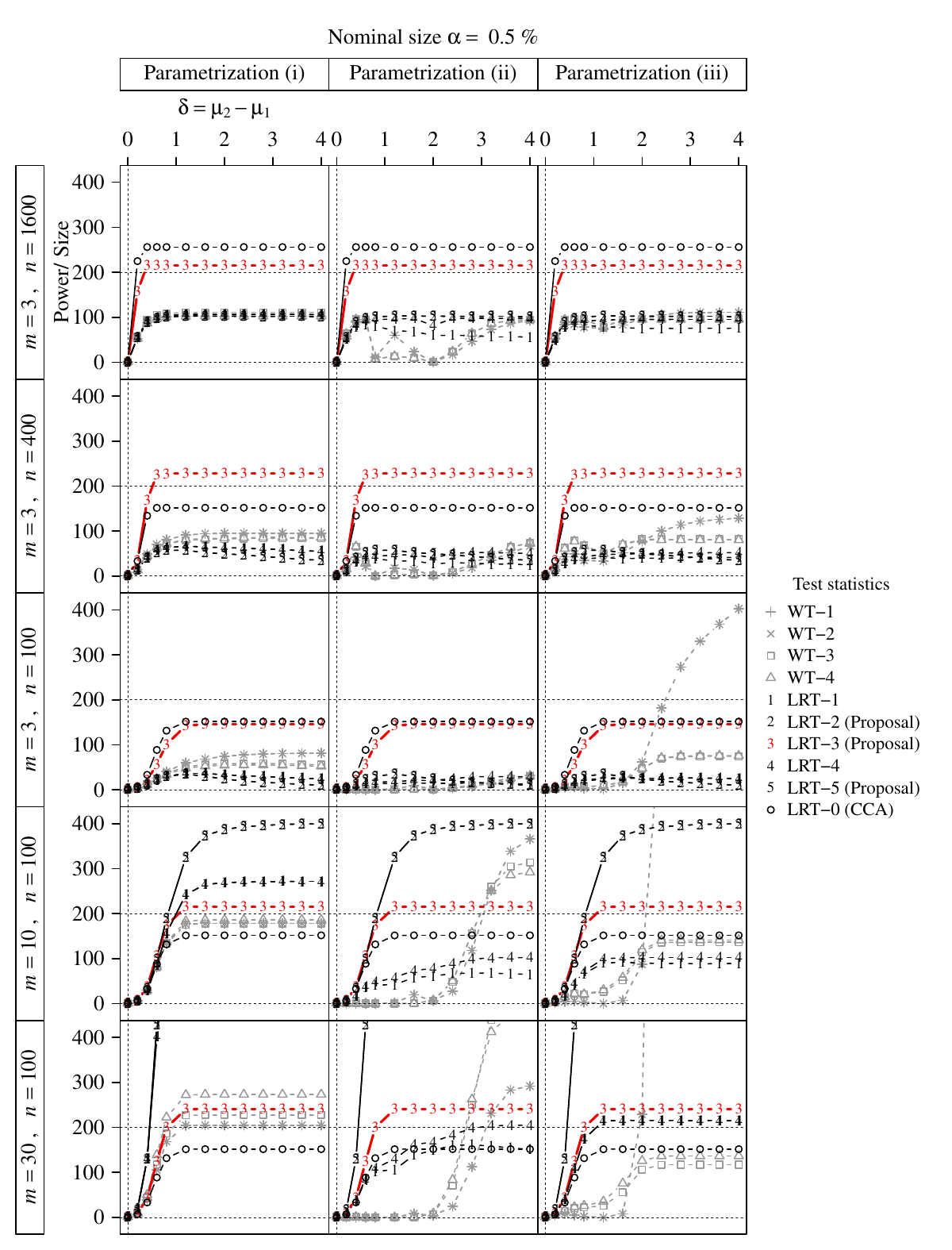}
\caption{\small 
The ratios of empirical power to empirical size under different parametrizations.
The nominal size is $\alpha=0.5\%$.
In each plot, the vertical axis denotes the ratio, 
and the horizontal axis denotes $\delta=\mu_2-\mu_1$.
The legend in Figure \ref{fig:pVal2} also applies here. 
The results under nominal size $5\%$ are shown in Figure \ref{fig:powerOverSize_5pc}.
}
\label{fig:powerOverSize_05pc}
\end{center} 
\end{figure}

\begin{figure}[h!]
\begin{center}
\includegraphics[width=\textwidth]{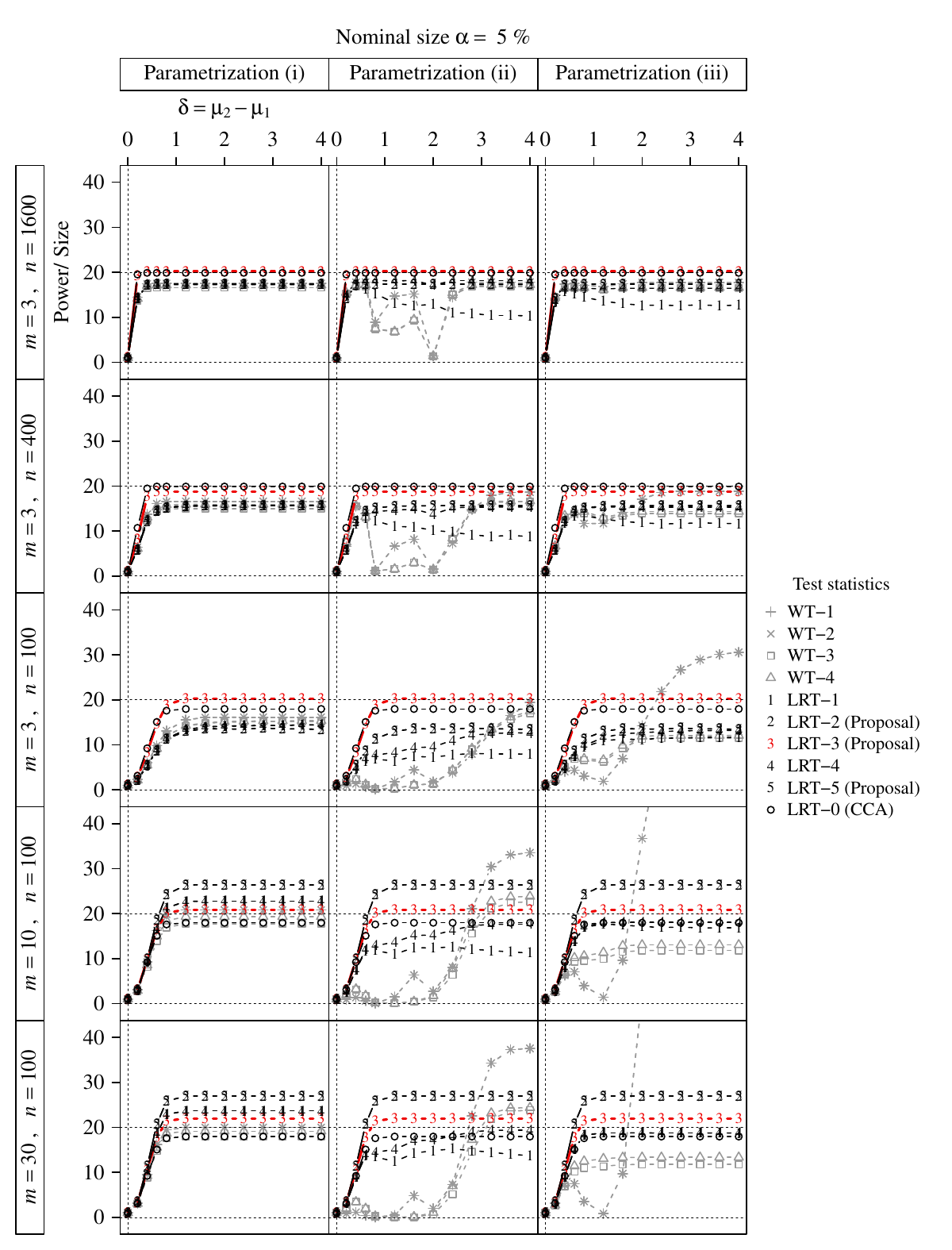} 
\caption{\small 
The ratios of empirical power to empirical size under different parametrizations.
The nominal size is $\alpha=5\%$.
In each plot, the vertical axis denotes the ratio, 
whereas the horizontal axis denotes $\delta=\mu_2-\mu_1$.
The legend in Figure \ref{fig:pVal2} also applies here. 
}
\label{fig:powerOverSize_5pc}
\end{center} 
\end{figure}

\begin{table*}[t]
\begin{center}
\renewcommand{\arraystretch}{0.5}
\caption{\small The range of empirical size $[\min \widehat{\alpha}, \max \widehat{\alpha}]$ in percentage,
where $\max$ and $\min$ are taken over the three parametrizations. 
Only one value is recorded for 
parametrization-invariant tests. 
The nominal size is $\alpha=0.5\%$.
The results under nominal size $\alpha=5\%$ are shown in Figure \ref{table:sizeError_5}.
}
\begin{tabular}{cccccc}\toprule
	&\multicolumn{5}{c}{\bf Range of empirical size: $[\min \widehat{\alpha}, \max \widehat{\alpha}]/\%$}\\
	\cmidrule(r){2-6}
$(n,m)$	 & $(1600,3)$ & $(400,3)$ & $(100,3)$ & $(100,10)$ & $(100,30)$  \\[0.5ex]
	\cmidrule(r){1-6}
	W-1	&	$[0.90,1.05] $	&	$ [0.76,1.05]$	&	$ [0.20,1.22]$	&	$ [0.07,0.56]$	&	$ [0.02,0.49]$	\\[0.5ex]
	W-2	&	$[0.90,1.05] $	&	$ [0.98,1.22]$	&	$ [0.93,1.25]$	&	$ [0.32,0.73]$	&	$ [0.20,0.85]$	\\[0.5ex]
	W-3	&	$[0.98,1.05] $	&	$ [0.98,1.25]$	&	$ [0.90,1.29]$	&	$ [0.34,0.71]$	&	$ [0.22,0.73]$	\\[0.5ex]
	W-4	&	$[0.90,1.05] $	&	$ [0.76,1.05]$	&	$ [0.20,1.22]$	&	$ [0.07,0.56]$	&	$ [0.02,0.49]$	\\[0.5ex]
	\cmidrule(r){1-6}
	L-1	&	$[0.90,1.03] $	&	$ [1.10,1.64]$	&	$ [1.15,1.49]$	&	$ [0.37,1.05]$	&	$ [0.10,0.46]$	\\[0.5ex]
	L-2	&	$[0.90,1.05] $	&	$ [1.10,1.76]$	&	$ [1.15,2.37]$	&	$ [0.37,0.98]$	&	$ [0.10,0.49]$	\\[0.5ex]
	L-3	&	$0.90$	&	$1.10$	&	$0.83$	&	$0.24$	&	$0.07$	\\[0.5ex]
	L-4	&	$0.90$	&	$1.10$	&	$0.83$	&	$0.24$	&	$0.07$	\\[0.5ex]
	L-5	&	$0.46$	&	$0.44$	&	$0.68$	&	$0.46$	&	$0.42$	\\[0.5ex]
	L-0	&	$0.39$	&	$0.66$	&	$0.66$	&	$0.66$	&	$0.66$	\\[0.5ex]
	\bottomrule
\end{tabular}
\label{table:sizeError_05}
\end{center}
\end{table*}

We also compare the performance of estimators of $\mathcal{r}_m$
for different $\delta$ and parametrizations.
In our experiment, we have $\mathcal{r}_m=1+1/m$ because we have set $\mathcal{r}=1$.
The MSEs of estimators 
$\widehat{f} = \widehat{r}/(1+\widehat{r})$
of $\mathcal{f}_m = \mathcal{r}_m/(1+\mathcal{r}_m)$
are shown in Figure \ref{fig:r_convMea}, in log scale. 
Clearly, the only estimator that is consistent, 
invariant to parametrization and robust against $\delta$ is 
our proposal $\widehat{f}_{\lrt}^{\rob} = \widehat{r}_{\lrt}^{\rob}/(1+\widehat{r}_{\lrt}^{\rob})$. 
It concentrates at the true value $\mathcal{f}_m$ 
quite closely even for small $m$ and $n$.
It verifies why L-5 has the greatest power.
On the other hand, the estimator $\widetilde{f}_{\lrt} = \widetilde{r}_{\lrt}/(1+\widetilde{r}_{\lrt})$ 
has a large MSE when $\delta \not{=} 0$.
It explains why L-1 is not powerful.

\begin{figure}[t!]
\begin{center}
\includegraphics[width=\textwidth]{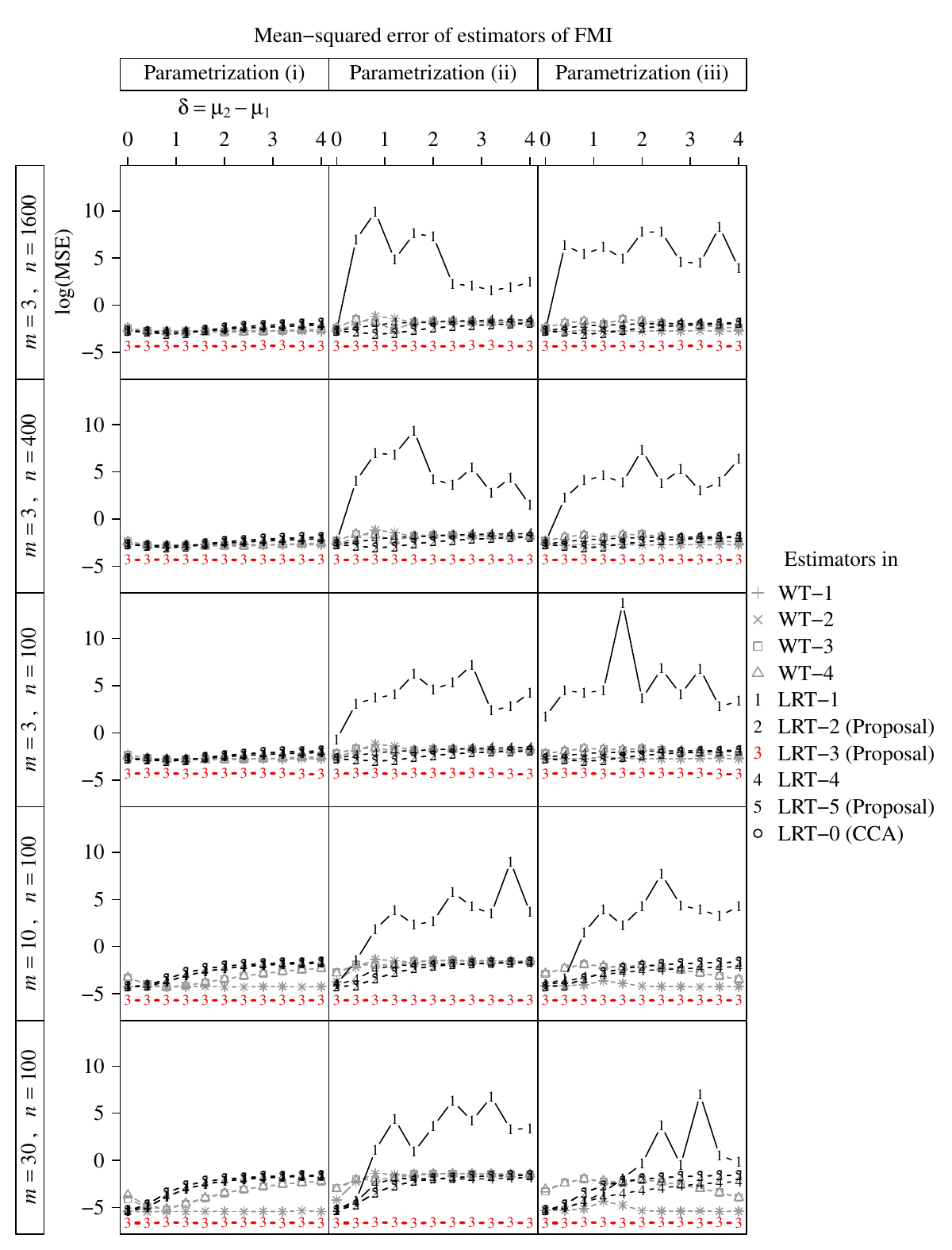}
\caption{\small The MSEs of
estimators of $\mathcal{f}_m$ used in the test statistics.
The vertical axis denotes the log of MSE, 
whereas the horizontal axis denotes the value of $\delta=\mu_2-\mu_1$.
The legend in Figure \ref{fig:pVal2} also applies here. 
}
\label{fig:r_convMea}
\end{center}
\end{figure}

\begin{table*}[t!]
\begin{center}
\caption{\small The range of empirical size $[\min \widehat{\alpha}, \max \widehat{\alpha}]$ in percentage,
where $\max$ and $\min$ are taken over the three parametrizations.
Only one value is recorded for 
parametrization-invariant tests. 
The nominal size is $\alpha=5\%$.
}
\renewcommand{\arraystretch}{0.5}
\begin{tabular}{cccccc}\toprule
	&\multicolumn{5}{c}{\bf Range of empirical size: $[\min \widehat{\alpha}, \max \widehat{\alpha}]/\%$}\\
	\cmidrule(r){2-6}
$(n,m)$	 & $(1600,3)$ & $(400,3)$ & $(100,3)$ & $(100,10)$ & $(100,30)$  \\[0.5ex]
	\cmidrule(r){1-6}
	W-1	&	$[5.62,5.71] $	&	$ [5.30,6.03]$	&	$ [3.22,6.20]$	&	$ [1.64,4.81]$	&	$ [1.37,5.00]$	\\[0.5ex]
	W-2	&	$[5.93,6.05] $	&	$ [6.08,7.18]$	&	$ [5.52,8.69]$	&	$ [4.42,8.47]$	&	$ [4.20,8.50]$	\\[0.5ex]
	W-3	&	$[5.81,6.03] $	&	$ [6.01,6.98]$	&	$ [5.37,8.28]$	&	$ [4.20,7.67]$	&	$ [4.10,7.50]$	\\[0.5ex]
	W-4	&	$[5.62,5.71] $	&	$ [5.30,6.03]$	&	$ [3.22,6.20]$	&	$ [1.64,4.81]$	&	$ [1.37,5.00]$	\\[0.5ex]
    \cmidrule(r){1-6}
	L-1	&	$[5.57,6.15] $	&	$ [6.37,6.57]$	&	$ [5.88,6.47]$	&	$ [4.39,5.66]$	&	$ [4.22,5.32]$	\\[0.5ex]
	L-2	&	$[5.52,6.10] $	&	$ [6.37,6.52]$	&	$ [5.88,7.47]$	&	$ [4.39,5.66]$	&	$ [4.22,5.32]$	\\[0.5ex]
	L-3	&	$5.76$	&	$6.37$	&	$5.42$	&	$3.78$	&	$3.71$	\\[0.5ex]
	L-4	&	$5.76$	&	$6.37$	&	$5.42$	&	$3.78$	&	$3.71$	\\[0.5ex]
	L-5	&	$4.96$	&	$5.32$	&	$4.93$	&	$4.79$	&	$4.54$	\\[0.5ex]
	L-0	&	$5.03$	&	$5.03$	&	$5.57$	&	$5.57$	&	$5.57$	\\[0.5ex]
		\bottomrule
\end{tabular}
\label{table:sizeError_5}
\end{center}
\end{table*}

\subsection{Supplements for Section~\ref{sec:MCexpUFMI}}\label{sec:MCexpUFMI_MI}
Let $n_j = \sum_{i=1}^n R_{ij}$ be the number of observed $j$th component.
Without loss of generality, 
assume $X_{\obs}$ is arranged in such a way that $R_{ij}\geq R_{i'j}$ for all $i<i'$ and $j$.
To impute the missing data, it is useful to represent $X_{i}$ by 
\begin{eqnarray*}
	\left[X_{i1} \mid \beta_1 , \tau_1^2\right] \sim \mathcal{N}(\beta_1, \tau^2_1) 
	\qquad\text{and} \qquad
	\left[X_{ij} \mid X_{i,1:(j-1)}, \beta_j, \tau_j^2\right] \sim \mathcal{N}(\beta_j^{\T} Z_{ij}, \tau^2_j),
\end{eqnarray*}
for $j=2,\ldots,p$,
where 
$\tau_1^2, \ldots, \tau_p^2 \in \mathbb{R}^+$, 
$\beta_j \in\mathbb{R}^{j}$,
$X_{i,1:(j-1)} = (X_{i1}, \ldots, X_{i,j-1})^{\T}$
and 
$Z_{ij} =(1, X_{i,1:(j-1)}^{\T})^{\T}$
for $j\ge 2$.
Denote the (complete-case) least squares estimators of $\beta_j$ and $\tau^2_j$ respectively by 
$$\widehat{\beta}_j = (Z_j^{\T}Z_j)^{-1}Z_j^{\T} W_j \quad {\rm and}\quad  \widehat{\tau}^2_j = \frac{(W_j - Z_j\widehat{\beta}_j)^{\T}(W_j - Z_j\widehat{\beta}_j)}{n_j-j},$$ 
where $Z_j = (Z_{1j}, \ldots, Z_{n_jj})^{\T}$ 
and $W_j = (X_{1j}, \ldots, X_{n_jj})^{\T}$.

We assume a Bayesian imputation model with the non-informative prior 
$
	f(\beta_1, \ldots, \beta_p, \tau^2_1, \ldots, \tau^2_p)
	\;\propto\; 1/(\tau^2_1\cdots \tau^2_p).
$
For $\ell=1,\ldots,m$, denote the $\ell$th imputed data set by $X^{(\ell)}$, 
whose $(i,j)$th element is $X_{ij}^{(\ell)}$. 
If $1\leq j\leq p$ and $i\leq n_j$, 
then $X_{ij}^{(\ell)} = X_{ij}$, otherwise $X_{ij}^{(\ell)}$ is filled in 
by recursing the following steps for $j=2,\ldots,p$.

\begin{enumerate}[noitemsep]
	\item Draw a sample $(\tau_j^{(\ell)})^{2}$ from 
			$\widehat{\tau}^2_j (n_j-j)/ \chi^2_{n_j-j}$.
	\item Draw a sample $\beta_j^{(\ell)}$
			from $\mathcal{N}_j(\widehat{\beta}_j, (\tau_j^{(\ell)})^{2} (Z_j^{\T}Z_j)^{-1} )$.
	\item Draw a sample $X_{ij}^{(\ell)}$
			from 
			$\mathcal{N}( (\beta_j^{(\ell)})^{\T}Z_{ij}^{(\ell)},  (\tau_j^{(\ell)})^{2})$
			for $i = n_j+1, \ldots, n$, 
			where $Z_{ij}^{(\ell)} =(1, (X_{i,1:(j-1)}^{(\ell)})^{\T})^{\T}$.
\end{enumerate}
With the $\ell$th imputed data set, 
the $H_0$-constrained MLEs of $\mu$ and $\Sigma$ are
$\widehat{\mu}_0^{(\ell)} = \bm{0}_p$ and 
$\widehat{\Sigma}_0^{(\ell)} = (X^{(\ell)})^{\T}(X^{(\ell)})/n$; 
whereas the unconstrained counterparts are 
$\widehat{\mu}^{(\ell)} = \bm{1}_n^{\T}X^{(\ell)}/n$ and
$\widehat{\Sigma}^{(\ell)} = (X^{(\ell)}-\widehat{\mu}^{(\ell)})^{\T}(X^{(\ell)}-\widehat{\mu}^{(\ell)})/n$.

The partial result is shown in Figure \ref{fig:jmi_3Jun2018_eg4a2} of the main text, whereas 
the full version is shown in Figure \ref{fig:jmi_3Jun2018_eg4a2_full}.

\begin{figure}[t!]
\begin{center}
\includegraphics[width=1\textwidth]{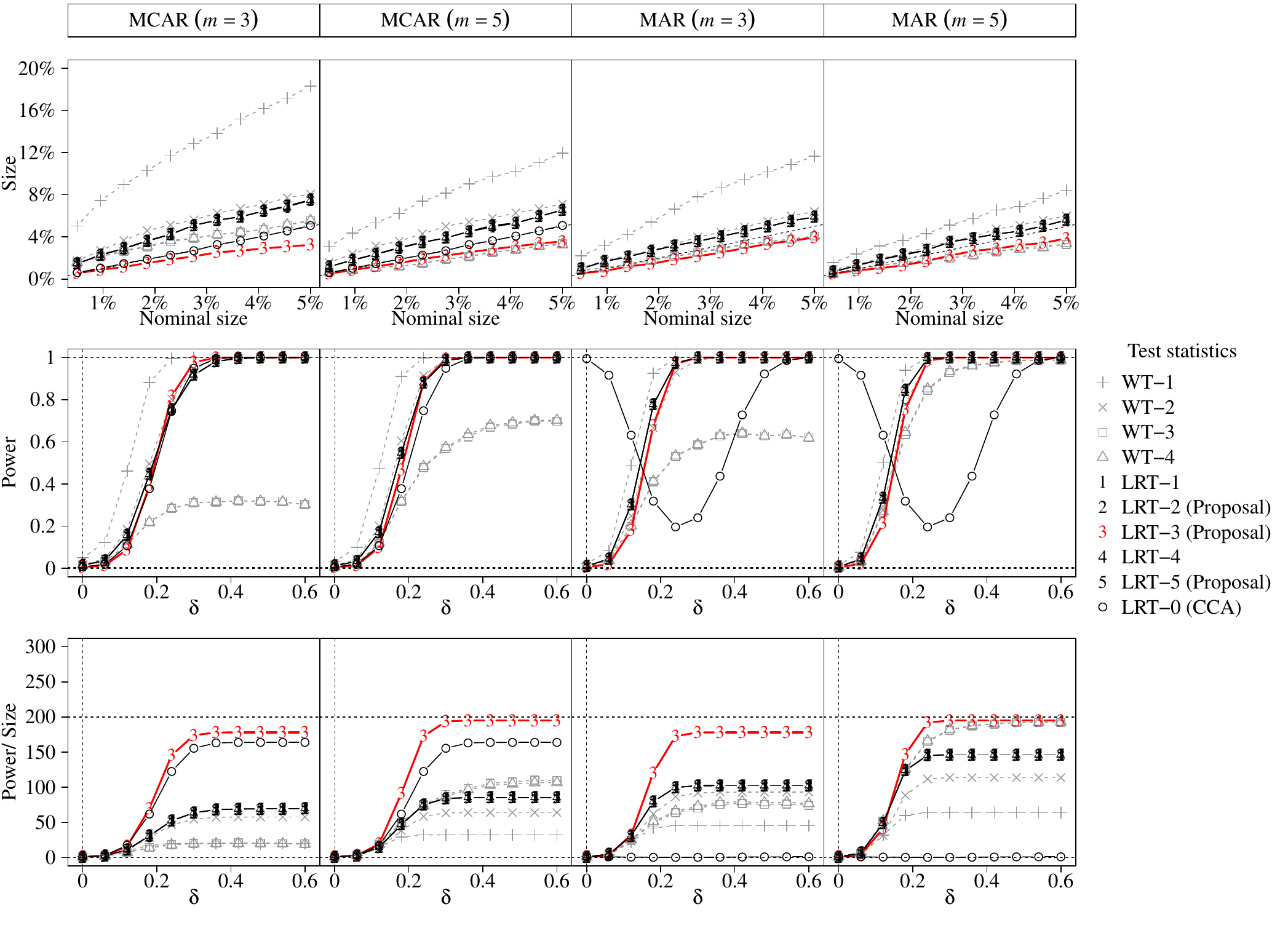}
\vspace{-0.9cm}
\caption{\small The empirical size, empirical power, and their ratio. 
The first row of plots show the empirical sizes.
The size of the complete-case test (C2) under MAR is off the chart (always equals to one) because it is invalid. 
The second and third rows of plots show the powers 
and the power-to-size ratios, respectively, 
where the nominal size is $0.5\%$.}
\label{fig:jmi_3Jun2018_eg4a2_full}
\end{center}
\end{figure}

\subsection{Applications to a Care-Survival Data}\label{sec:contTable_proc}
\cite{mengRubin92} considered the data given in Table \ref{table:realData_data set}, 
where $i$, $j$ and $k$ index, respectively, 
amount of parental care (less or more, corresponding to $i=1,2$), 
and survival status (died or survived, corresponding to $j=1,2$), and  
clinic (A or B, corresponding to $k=1,2$). 
The label $k$ is missing for some observations.
The missing mechanism was assumed to be ignorable.   
We consider two null hypotheses: 
($H_0$) the clinic and parental care are conditionally independent given the survival status, and 
($H_0'$) all three variables are independent. 
It is remarked that testing the conditional independence model (i.e., $H_0$) is useful from a modeling perceptive.
If $H_0$ cannot be rejected, then one may be tempted to adopt the more parsimonious null model (for the cell probabilities). 
The same model is also suggested in \cite{rubinLittle2002} and \cite{mengRubin92}.

Our aim is to investigate the impact 
on $\{\widetilde{D}_{\Stack},\widehat{D}_{\Stack}^{+},\widehat{D}^{\rob}_{\Stack}\}$ 
by the parametrization of the cell probabilities 
\[
	\pi_{ijk} = \mathsf{P}(\text{parental care}=i, \text{survival status}=j, \text{clinic label}=k) 
\]
for $i,j,k\in\{1,2\}$;
and the impact on  
$\{\widetilde{r}_{\lrt},\widehat{r}^{+}_{\Stack},\widehat{r}^{\rob}_{\Stack}\}$ 
under different null hypotheses. 
Here the full model parameter vector can be 
expressed as $\psi = (\pi_{111}, \pi_{112}, \pi_{121}, \pi_{122}, \pi_{211}, \pi_{212},\pi_{221})^{\T}$.
Since the restrictions imposed by $H_0$ are 
$\pi_{ijk} = (\pi_{1jk}+\pi_{2jk})(\pi_{ij1}+\pi_{ij2})$ for $j=1,2$,  
one may express the parameter of interest as $\theta = (\theta_1, \theta_2)^{\T}$, where 
$\theta_j = \pi_{ijk} - (\pi_{1jk}+\pi_{2jk})(\pi_{ij1}+\pi_{ij2})$ for $j=1,2$.
Then $H_0$ can be equivalently stated as $\theta = \theta_0$, where $\theta_0 = (0,0)^{\T}$. 
Similarly, the parameter of interest under $H_0'$ can be defined.

\begin{table}[h!]
\renewcommand{\arraystretch}{0.5}
\begin{center}
  \caption{\small Data from \cite{mengRubin92}. 
				The notation ``?'' indicates missing label. 
				}
  \begin{tabular}{@{} lcc|ccccc @{}}
    \toprule
\multicolumn{2}{l}{{\bf Parental care} ($i$)} && \multicolumn{2}{c}{{\bf Less}} & \multicolumn{2}{c}{{\bf More}}\\
\cmidrule(r){4-5}	\cmidrule(r){6-7}	
   \multicolumn{2}{l}{{\bf Survival Status} ($j$)} && Died & Survived & Died & Survived \\ 
    \midrule
 \multicolumn{2}{l}{{\bf Clinic Label} ($k$)}&    A &  3 & 176  & 4 & 293 \\ 
    &&B & 17 & 197  & 2 & 23  \\ 
    &&? & 10 & 150  & 5 & 90  \\ 
    \bottomrule
  \end{tabular} 
  \label{table:realData_data set}
\end{center}
\end{table}

The computation of the stacked MI estimators of $\{\pi_{ijk}\}$ is presented in \ref{sec:contTable_proc} of the Appendix. We consider three parametrizations:
(i) $\psi_{ijk} = \pi_{ijk}$;  
(ii) $\psi_{ijk} = \log\{\pi_{ijk}/(1-\pi_{ijk})\}$; and   
(iii) $\psi_{ij1} = \pi_{ij1}$ and $\psi_{ij2} = \pi_{ij2}/\pi_{ij1}$. 
Denote the $p$-values of tests
$\{\widetilde{D}_{\lrt}, \widehat{D}_{\Stack}^{+},\widehat{D}^{\rob}_{\Stack}\}$
by 
$\{\widetilde{p}_{\lrt},\widehat{p}_{\Stack}^{+},\widehat{p}^{\rob}_{\Stack}\}$, respectively.
The results are summarized in Table \ref{table:lrt_contingencyTable_parametrization123_full}.
Clearly, only $\widehat{r}_{\Stack},\widehat{r}_{\rob},\widehat{D}_{\Stack}^{+},\widehat{D}^{\rob}_{\Stack}$
are always non-negative and parametrization-invariant.
Some of the values of $\widetilde{r}_{\lrt}$ and $\widetilde{D}_{\lrt}$ are negative, 
leading to the meaningless $\widetilde{p}_{\lrt}=1$. 
For testing $H_0$, 
we have $\widehat{D}_{\Stack}^{+} \approx \widehat{D}_{\Stack}^{\rob}$. 
For testing $H_0'$, 
$\widehat{D}_{\Stack}^{+}$ and $\widehat{D}_{\Stack}^{\rob}$ 
are not very close to each other,
but they both lead to essentially zero $p$-value.
These results reconfirm  
the conclusions in \citet{mengRubin92}.
Moreover,
only $\widehat{r}_{\Stack}^{\rob}$ does not change under 
different null hypotheses.

The MI data sets are generated from a Bayesian model in Section~4.2 of \cite{mengRubin92}.
The $\ell$th imputed log-likelihood function is
$\log f(X^{(\ell)}\mid\pi) = \sum_c n_{c}^{(\ell)} \log \pi_c$,
where 
$X^{(\ell)}$ are the cell counts $n_{c}^{(\ell)}$ in the $\ell$th 
imputed data set. 
Hence the unconstrained MLE of $\pi_c$ is 
$\widehat{\pi}_{c}^{(\ell)}= {n_{c}^{(\ell)}}/{n_{+}^{(\ell)}}$,
where $n_{+}^{(\ell)} = \sum_c n_{c}^{(\ell)}$.
Let $n_{c}^{+} = \sum_{\ell=1}^m n_{c}^{(\ell)}$. 
Consequently, the joint log-likelihood based on the stacked data is 
\begin{equation}\label{eqt:egMeng_jointLike}
	\log f(X^{\Stack}\mid\pi)
		= \sum_{\ell=1}^m \sum_c n_{c}^{(\ell)} \log \pi_c 
		= \sum_c n_{c}^{+} \log \pi_c ,  
\end{equation}
Thus the unconstrained MLE with respect to (\ref{eqt:egMeng_jointLike}) is
$\widehat{\pi}_{c}^{\Stack} = {n_{c}^{+}}/{n_{+}^{+}}$,  
where $n_{+}^{+} = \sum_c n_{c}^{+}$.
Similarly, we can find the constrained MLEs under a given null.

\begin{table*}[h!]
\begin{center}
\caption{\small The LRTs using 
$\widetilde{D}_{\lrt}$, $\widehat{D}_{\Stack}^{+}$ and $\widehat{D}^{\rob}_{\Stack}$
under different parametrizations in Section~\ref{sec:contTable_proc}.
}
\setlength{\tabcolsep}{2pt}
\renewcommand{\arraystretch}{0.5}
\begin{tabular}{c|ccc|ccc}\toprule
&\multicolumn{5}{c}{\bf Parametrization (i): identity map}\\  
\cmidrule(r){2-7}
&\multicolumn{3}{c}{$H_0$: Conditional independence}	& \multicolumn{2}{c}{$H_0$: Full 
independence}\\
\cmidrule(r){2-4}\cmidrule(r){5-7}
$m$	 & $\widetilde{r}_{\lrt},\widehat{r}^{+}_{\Stack},\widehat{r}^{\rob}_{\Stack}$ & $\widetilde{D}_{\lrt},\widehat{D}_{\Stack}^{+},\widehat{D}_{\Stack}^{\rob}$ & $\widetilde{p}_{\lrt},\widehat{p}_{\Stack}^{+},\widehat{p}_{\Stack}^{\rob}$ & $\widetilde{r}_{\lrt},\widehat{r}^{+}_{\Stack},\widehat{r}^{\rob}_{\Stack}$ & $\widetilde{D}_{\lrt},\widehat{D}_{\Stack}^{+},\widehat{D}_{\Stack}^{\rob}$ & $\widetilde{p}_{\lrt},\widehat{p}_{\Stack}^{+},\widehat{p}_{\Stack}^{\rob}$  \\[0.5ex]
\cmidrule(r){1-1}\cmidrule(r){2-2}\cmidrule(r){3-3}\cmidrule(r){4-4}\cmidrule(r){5-5}\cmidrule(r){6-6}\cmidrule(r){7-7}
	$2$	&	$ 0.63, 0.64, 0.83$	&	$ 0.14, 0.14, 0.12$	&	$ 0.87, 0.87, 0.89$	&	$ 0.53, 0.53, 0.83$	&	$44.4,44.4,37.1$	&	$ 0,0,0$	\\[0.5ex]
	$3$	&	$ 0.54, 0.54, 0.38$	&	$ 0.08, 0.08, 0.09$	&	$ 0.93, 0.93, 0.92$	&	$ 0.31, 0.31, 0.38$	&	$54.2,54.2,51.4$	&	$ 0,0,0$	\\[0.5ex]
	$5$	&	$ 0.49, 0.48, 0.89$	&	$ 0.12, 0.12, 0.10$	&	$ 0.89, 0.89, 0.91$	&	$ 0.72, 0.72, 0.89$	&	$40.8,40.8,37.1$	&	$ 0,0,0$	\\[0.5ex]
	$7$	&	$ 0.23, 0.23, 0.47$	&	$ 0.06, 0.06, 0.05$	&	$ 0.94, 0.94, 0.95$	&	$ 0.31, 0.31, 0.47$	&	$53.2,53.2,47.6$	&	$ 0,0,0$	\\[0.5ex]
	$10$	&	$ 0.50, 0.50, 0.70$	&	$ 0.14, 0.14, 0.12$	&	$ 0.87, 0.87, 0.88$	&	$ 0.56, 0.56, 0.70$	&	$45.4,45.4,41.7$	&	$ 0,0,0$	\\[0.5ex]
	$25$	&	$ 0.35, 0.35, 0.47$	&	$ 0.06, 0.06, 0.06$	&	$ 0.94, 0.94, 0.95$	&	$ 0.35, 0.35, 0.47$	&	$51.4,51.4,47.0$	&	$ 0,0,0$	\\[0.5ex]
	$50$	&	$ 0.31, 0.31, 0.45$	&	$ 0.11, 0.11, 0.10$	&	$ 0.90, 0.90, 0.91$	&	$ 0.33, 0.33, 0.45$	&	$51.5,51.5,47.3$	&	$ 0,0,0$	\\[0.5ex]
\cmidrule(r){1-7}
&\multicolumn{5}{c}{\bf Parametrization (ii): logit transformation }\\  
\cmidrule(r){2-7}
&\multicolumn{3}{c}{$H_0$: Conditional independence}	& \multicolumn{2}{c}{$H_0$: Full 
independence}\\
\cmidrule(r){2-4}\cmidrule(r){5-7}
$m$	 & $\widetilde{r}_{\lrt},\widehat{r}^{+}_{\Stack},\widehat{r}^{\rob}_{\Stack}$ & $\widetilde{D}_{\lrt},\widehat{D}_{\Stack}^{+},\widehat{D}_{\Stack}^{\rob}$ & $\widetilde{p}_{\lrt},\widehat{p}_{\Stack}^{+},\widehat{p}_{\Stack}^{\rob}$ & $\widetilde{r}_{\lrt},\widehat{r}^{+}_{\Stack},\widehat{r}^{\rob}_{\Stack}$ & $\widetilde{D}_{\lrt},\widehat{D}_{\Stack}^{+},\widehat{D}_{\Stack}^{\rob}$ & $\widetilde{p}_{\lrt},\widehat{p}_{\Stack}^{+},\widehat{p}_{\Stack}^{\rob}$  \\[0.5ex]
\cmidrule(r){1-1}\cmidrule(r){2-2}\cmidrule(r){3-3}\cmidrule(r){4-4}\cmidrule(r){5-5}\cmidrule(r){6-6}\cmidrule(r){7-7}
	$2$	&	$ 1.23, 0.64, 0.83$	&	$ 0.01, 0.14, 0.12$	&	$ 0.99, 0.87, 0.89$	&	$ 0.98, 0.53, 0.83$	&	$34.2,44.4,37.1$	&	$ 0,0,0$	\\[0.5ex]
	$3$	&	$ 1.08, 0.54, 0.38$	&	$-0.07, 0.08, 0.09$	&	$ 1.00, 0.93, 0.92$	&	$ 0.61, 0.31, 0.38$	&	$43.9,54.2,51.4$	&	$ 0,0,0$	\\[0.5ex]
	$5$	&	$ 1.02, 0.48, 0.89$	&	$-0.09, 0.12, 0.10$	&	$ 1.00, 0.89, 0.91$	&	$ 1.40, 0.72, 0.89$	&	$29.0,40.8,37.1$	&	$ 0,0,0$	\\[0.5ex]
	$7$	&	$ 0.45, 0.23, 0.47$	&	$-0.07, 0.06, 0.05$	&	$ 1.00, 0.94, 0.95$	&	$ 0.58, 0.31, 0.47$	&	$43.9,53.2,47.6$	&	$ 0,0,0$	\\[0.5ex]
	$10$	&	$ 0.99, 0.50, 0.70$	&	$-0.10, 0.14, 0.12$	&	$ 1.00, 0.87, 0.88$	&	$ 1.09, 0.56, 0.70$	&	$33.7,45.4,41.7$	&	$ 0,0,0$	\\[0.5ex]
	$25$	&	$ 0.71, 0.35, 0.47$	&	$-0.14, 0.06, 0.06$	&	$ 1.00, 0.94, 0.95$	&	$ 0.68, 0.35, 0.47$	&	$41.0,51.4,47.0$	&	$ 0,0,0$	\\[0.5ex]
	$50$	&	$ 0.63, 0.31, 0.45$	&	$-0.10, 0.11, 0.10$	&	$ 1.00, 0.90, 0.91$	&	$ 0.65, 0.33, 0.45$	&	$41.3,51.5,47.3$	&	$ 0,0,0$	\\[0.5ex]
\cmidrule(r){1-7}
&\multicolumn{5}{c}{\bf Parametrization (iii): ratios of probabilities}\\   
\cmidrule(r){2-7}
&\multicolumn{3}{c}{$H_0$: Conditional independence}	& \multicolumn{2}{c}{$H_0$: Full 
independence}\\
\cmidrule(r){2-4}\cmidrule(r){5-7}
$m$	 & $\widetilde{r}_{\lrt},\widehat{r}^{+}_{\Stack},\widehat{r}^{\rob}_{\Stack}$ & $\widetilde{D}_{\lrt},\widehat{D}_{\Stack}^{+},\widehat{D}_{\Stack}^{\rob}$ & $\widetilde{p}_{\lrt},\widehat{p}_{\Stack}^{+},\widehat{p}_{\Stack}^{\rob}$ & $\widetilde{r}_{\lrt},\widehat{r}^{+}_{\Stack},\widehat{r}^{\rob}_{\Stack}$ & $\widetilde{D}_{\lrt},\widehat{D}_{\Stack}^{+},\widehat{D}_{\Stack}^{\rob}$ & $\widetilde{p}_{\lrt},\widehat{p}_{\Stack}^{+},\widehat{p}_{\Stack}^{\rob}$  \\[0.5ex]
\cmidrule(r){1-1}\cmidrule(r){2-2}\cmidrule(r){3-3}\cmidrule(r){4-4}\cmidrule(r){5-5}\cmidrule(r){6-6}\cmidrule(r){7-7}
	$2$	&	$    1.06,    0.64,    0.83 $	&	$     0.04,    0.14,    0.12 $	&	$     0.96,    0.87,    0.88 $	&	$    -0.38,    0.53,    0.83 $	&	$   109,   44.4,   37.1 $	&	$     0,0,0$	\\[0.5ex]
	$3$	&	$   -2.35,    0.54,    0.38 $	&	$    -1.16,    0.08,    0.09 $	&	$     1.00,    0.93,    0.92 $	&	$    -1.22,    0.31,    0.38 $	&	$  -321,   54.2,   51.4 $	&	$     1,0,0$	\\[0.5ex]
	$5$	&	$   -2.64,    0.48,    0.89 $	&	$    -1.38,    0.12,    0.10 $	&	$     1.00,    0.89,    0.91 $	&	$    -2.24,    0.72,    0.89 $	&	$   -58,   40.8,   37.1 $	&	$     1,0,0$	\\[0.5ex]
	$7$	&	$   -0.01,    0.23,    0.47 $	&	$     0.25,    0.06,    0.05 $	&	$     0.78,    0.94,    0.95 $	&	$    -0.34,    0.31,    0.47 $	&	$   107,   53.2,   47.6 $	&	$     0,0,0$	\\[0.5ex]
	$10$	&	$   -2.04,    0.50,    0.70 $	&	$    -2.20,    0.14,    0.12 $	&	$     1.00,    0.87,    0.88 $	&	$    -1.85,    0.56,    0.70 $	&	$   -86,   45.4,   41.7 $	&	$     1,0,0$	\\[0.5ex]
	$25$	&	$   -1.39,    0.35,    0.47 $	&	$    -4.30,    0.06,    0.06 $	&	$     1.00,    0.94,    0.95 $	&	$    -1.12,    0.35,    0.47 $	&	$  -603,   51.4,   47.0 $	&	$     1,0,0$	\\[0.5ex]
	$50$	&	$   -1.22,    0.31,    0.45 $	&	$    -7.39,    0.11,    0.10 $	&	$     1.00,    0.90,    0.91 $	&	$    -1.06,    0.33,    0.45 $	&	$ -1136,   51.5,   47.3 $	&	$     1,0,0$	\\[0.5ex]
\bottomrule
\end{tabular}
\label{table:lrt_contingencyTable_parametrization123_full}
\end{center}
\end{table*}

\section{Proofs}\label{sec:proof}
\begin{proof}[Proof of Theorem~\ref{thm:AsyEq_dBar_LL}]
(i, ii) From (\ref{eqt:bardd_combinedLike}), 
we know $\widehat{d}_{\lrt} \geq 0$ is invariant to parametrization $\psi$.
(iii) Since $\widehat{d}_{\lrt}$ is invariant to transformation of $\psi$,
we assume, without loss of generality, 
that $\psi$ admits a parameterization such that 
$\Cov(\widehat{\theta}^{(\ell)}, \widehat{\eta}^{(\ell)}) \bumpeq \mathbf{0}$
by taking suitable linear transformation of $\psi$.
Also write $U^{(\ell)}_{\eta}$
as an efficient estimator of $\Var(\widehat{\eta})$ based on $X^{(\ell)}$;
and recall that $U^{(\ell)}_{\theta}= U^{(\ell)}$ is an efficient estimator of 
$\Var(\widehat{\theta})$ based on $X^{(\ell)}$.

Using Taylor's expansion on 
$\psi\mapsto\overline{\loglik}(\psi)= m^{-1}\sum_{\ell=1}^m \log f(X^{(\ell)}\mid\psi)$ 
around  
$\widehat{\psi}^{*} = ((\widehat{\theta}^{*})^{\T},(\widehat{\eta}^{*})^{\T})^{\T}$, 
we know that for $\psi\bumpeq \widehat{\psi}^{*}$,
\begin{equation}\label{eqt:Taylor}
	\overline{\loglik}(\psi)
		\bumpeq \overline{\loglik}(\widehat{\psi}^{*})
			- \frac{1}{2} \left( \psi -\widehat{\psi}^{*} \right)^{\T} 
				\overline{I}(\widehat{\psi}^{*}) \left( \psi - \widehat{\psi}^{*} \right), 
\end{equation}
where 
$\overline{I}(\psi) = - {\partial^2 \overline{\loglik}(\psi)}/{\partial \psi\partial \psi^{\T}}$, 
which satisfies 
\begin{equation}\label{eqt:blockFormFishInfo}
	\overline{I} (\widehat{\psi}^{*}) 
		\bumpeq 
	\left( \begin{array}{cc}
	 	\overline{U}^{-1}_{\theta} & \mathbf{0} \\
		\mathbf{0} & \overline{U}^{-1}_{\eta}
	 \end{array}\right)
\end{equation}
with $\overline{U}_{\eta} = m^{-1}\sum_{i=1}^m U^{(\ell)}_{\eta}$.
Under the null, $\widehat{\psi}^{*}\bumpeq\widehat{\psi}_{0}^{*}$.
So, using (\ref{eqt:Taylor}), we have
\begin{eqnarray}\label{eqt:proof_bard_decomp}
	\widehat{d}_{\lrt}
		&\bumpeq& \left( \widehat{\psi}_{0}^{*}- \widehat{\psi}^{*}\right)^{\T} 
					 \overline{I} (\widehat{\psi}^{*})
	 			\left( \widehat{\psi}_{0}^{*}-\widehat{\psi}^{*} \right), \nonumber \\
		&\bumpeq& \left( \begin{array}{c}
						\theta_0 - \widehat{\theta}^{*} \\
						\widehat{\eta}(\theta_0) - \widehat{\eta}(\widehat{\theta}^{*})
					\end{array} \right)^{\T} 
					\left( \begin{array}{cc}
					 	\overline{U}^{-1}_{\theta} & \mathbf{0} \\
						\mathbf{0} & \overline{U}^{-1}_{\eta}
					 \end{array}\right)
	 			\left( \begin{array}{c}
						\theta_0 - \widehat{\theta}^{*} \\
						\widehat{\eta}(\theta_0) - \widehat{\eta}(\widehat{\theta}^{*})
					\end{array} \right)  \nonumber\\
		&\bumpeq& (\overline{\theta}^{\T}-\theta_0) \overline{U}^{-1}_{\theta} (\overline{\theta}^{\T}-\theta_0)
		= \widetilde{d}_{\wt}^{\prime}, 
\end{eqnarray}
where we have used 
(a) $\widehat{\theta}^{*}\bumpeq \overline{\theta}$; see, e.g., Lemma 1 of \citet{wangRobins1998},
and (b) $\widehat{\eta}(\theta_0) - \widehat{\eta}(\widehat{\theta}^{*}) = O_p(1/n)$
if $\theta_0 - \widehat{\theta}^{*}=O_p(1/\sqrt{n})$; see \cite{coxReid1987}.
Since $\widetilde{d}_{\wt}^{\prime}\bumpeq \widetilde{d}_{\lrt}$
\citep*{mengRubin92}, we have $\widehat{d}_{\lrt}\bumpeq \widetilde{d}_{\lrt}$. 
\end{proof}

\begin{proof}[Proof of Proposition \ref{prop:condForPositiveR}]
The given condition implies that 
\begin{gather*}
	\widehat{\psi}^{(\ell)} 
		= ( (\widehat{\theta}^{(\ell)})^{\T}, (\widehat{\eta}^{(\ell)})^{\T} )^{\T},
	\qquad 	
	\widehat{\psi}_{0}^{(\ell)} 
		= (\theta_0^{\T}, (\widehat{\eta}^{(\ell)})^{\T})^{\T} ,\\
	\widehat{\psi}^{*} 
		= ( (\widehat{\theta}^{*})^{\T}, (\widehat{\eta}^{*})^{\T} )^{\T},
	\qquad
	\widehat{\psi}_{0}^{*} 
		= (\theta_0^{\T}, (\widehat{\eta}^{*})^{\T})^{\T} .
\end{gather*}
Clearly, we also have the decomposition:
$L^{(\ell)}(\psi) = L^{(\ell)}_{\dag}(\theta) + L^{(\ell)}_{\ddag}(\eta)$ for all $\ell$, 
where $L^{(\ell)}_{\dag}(\theta) = L_{\dag}(\theta\mid X^{(\ell)})$
and $L^{(\ell)}_{\ddag}(\eta) = L_{\ddag}(\eta\mid X^{(\ell)})$.
Then, 
\begin{eqnarray*}
	\overline{d}_{\lrt} - \widehat{d}_{\lrt}
		&=& \frac{2}{m}\sum_{\ell=1}^m \left\{ 
				  L^{(\ell)}(\widehat{\psi}^{(\ell)}) 
				- L^{(\ell)}(\widehat{\psi}_0^{(\ell)}) 
				- L^{(\ell)}(\widehat{\psi}^{*}) 
				+ L^{(\ell)}(\widehat{\psi}_0^{*}) \right\}\\
		&= &\frac{2}{m}\sum_{\ell=1}^m \left\{ 
				  L^{(\ell)}_{\dag}(\widehat{\theta}^{(\ell)}) 
				- L_{\dag}^{(\ell)}(\widehat{\theta}^{*}) \right\} \geq 0
\end{eqnarray*}
since $L^{(\ell)}_{\dag}(\widehat{\theta}^{(\ell)}) \geq L_{\dag}^{(\ell)}(\widehat{\theta}^{*})$ for all $\ell$.  
\end{proof}

\begin{proof}[Proof of Corollary \ref{coro:finitenessOfr}]
Applying Taylor's expansion on $\psi\mapsto\loglik^{(\ell)}(\psi)$, 
we can find $\widecheck{\psi}{}^{(\ell)}$
lying on the line segment joining $\widehat{\psi}^{(\ell)}$ and $\widehat{\psi}_{0}^{(\ell)}$
such that 
\[
	\loglik^{(\ell)}(\widehat{\psi}_{0}^{(\ell)}) 
		= \loglik^{(\ell)}(\widehat{\psi}^{(\ell)}) 
				- \frac{1}{2} \left( \widehat{\psi}_{0}^{(\ell)} - \widehat{\psi}^{(\ell)} \right)^{\T} 
				 I^{(\ell)}(\widecheck{\psi}{}^{(\ell)}) \left( \widehat{\psi}_{0}^{(\ell)} - \widehat{\psi}^{(\ell)} \right),
\]
where $I^{(\ell)}(\psi) = - \partial^2 \loglik^{(\ell)}(\psi) /\partial \psi\partial\psi^{\T}$.
By the lower order variability of $I^{(\ell)}(\widecheck{\psi}{}^{(\ell)})$,
we can find $\widecheck{\psi}{}^*$ such that 
$I^{(\ell)}(\widecheck{\psi}{}^{(\ell)}) \bumpeq  I^{(\ell)} (\widecheck{\psi}{}^*)$ for all $\ell$.
Then, using similar techniques as in (\ref{eqt:blockFormFishInfo}) and (\ref{eqt:proof_bard_decomp}), we have
\begin{eqnarray}\label{eqt:diffLL_H01}
	\loglik^{(\ell)}(\widehat{\psi}^{(\ell)}) - \loglik^{(\ell)}(\widehat{\psi}_{0}^{(\ell)}) 
		&\bumpeq&  \frac{1}{2} \left( \widehat{\psi}_{0}^{(\ell)} - \widehat{\psi}^{(\ell)} \right)^{\T} 
				 		I^{(\ell)} (\widecheck{\psi}{}^*) 
				 		\left( \widehat{\psi}_{0}^{(\ell)} - \widehat{\psi}^{(\ell)} \right) \nonumber \\
		&\bumpeq& \frac{1}{2} \left( \theta_0 - \widehat{\theta}^{(\ell)} \right)^{\T} 
						\widecheck{U}^{-1} \left( \theta_0 - \widehat{\theta}^{(\ell)} \right) 
\end{eqnarray}
for some matrix $\widecheck{U}$.
Similarly, we have 
\begin{equation}\label{eqt:diffLL_H01_ave}
	\loglik^{(\ell)}(\widehat{\psi}^{*}) - \loglik^{(\ell)}(\widehat{\psi}_{0}^{*}) 
		\bumpeq \frac{1}{2} \left( \theta_0 - \widehat{\theta}^{*} \right)^{\T} 
				\widecheck{U}^{-1} \left( \theta_0 - \widehat{\theta}^{*} \right) .
\end{equation}
Write $A^{\otimes 2} = AA^{\T}$ for any appropriate matrix $A$.
Using (\ref{eqt:diffLL_H01}), (\ref{eqt:diffLL_H01_ave})
and the cyclic property of trace, we have 
\begin{eqnarray*}
	\overline{d}_{\lrt} - \widehat{d}_{\lrt} \nonumber
		&\bumpeq& \frac{1}{m}\sum_{\ell=1}^m \left\{ \left( \theta_0 - \widehat{\theta}^{(\ell)} \right)^{\T} 
				\widecheck{U}^{-1} \left( \theta_0 - \widehat{\theta}^{(\ell)} \right)
				- \left( \theta_0 - \widehat{\theta}^{*} \right)^{\T} 
				\widecheck{U}^{-1} \left( \theta_0 - \widehat{\theta}^{*} \right) \right\} \nonumber\\
		&=& \tr\left[ 
						\widecheck{U}^{-1} \left\{
						\frac{1}{m} \sum_{\ell=1}^m 
						\left( \theta_0 - \widehat{\theta}^{(\ell)} \right)^{\otimes 2}
						- \left( \theta_0 - \widehat{\theta}^{*} \right)^{\otimes 2} \right\}
				\right] \nonumber\\
		&\bumpeq& \tr\left[
					\widecheck{U}^{-1} \frac{1}{m} \sum_{\ell=1}^m 
						\left\{ (\widehat{\theta}^{(\ell)})^{\otimes 2} 
						- \overline{\theta}^{\otimes 2} \right\}
				\right] 
		\bumpeq \tr\left( \widecheck{U}^{-1} B \right) 
		\bumpeq \tr\left( \mathcal{U}_{\theta,0}^{-1} \mathcal{B}_{\theta} \right)\label{eqt:hatDelta_underH1}
\end{eqnarray*}
as $m,n\rightarrow\infty$,
where $\mathcal{U}_{\theta,0}$ is a deterministic matrix that depends on both $\theta_0$ and 
the true value of $\theta$, and satisfies $n(\widecheck{U} - \mathcal{U}_{\theta,0})\inP 0 $.
Note that $\tr( \mathcal{U}_{\theta,0}^{-1} \mathcal{B}_{\theta} ) = k\mathcal{r}_0$, 
for some finite $\mathcal{r}_0$ by Assumption \ref{ass:likelihood}.
Then $\widehat{r}_{\lrt} \inP \mathcal{r}_0 = \tr( \mathcal{U}_{\theta,0}^{-1} \mathcal{B}_{\theta} )/k$,
proving (ii).
(But $\mathcal{U}_{\theta,0}$ may not equal to $\mathcal{U}_{\theta}$, and hence 
$\widehat{r}_{\lrt}$ may not be consistent for $\mathcal{r}_m$.)

If $H_0$ is true, 
then $\overline{\theta}\inP \theta_0$
and $\widecheck{U} \bumpeq \overline{U} \bumpeq \mathcal{U}_{\theta} = \mathcal{U}_{\theta,0}$.
Then, $\widehat{r}_{\lrt} \inP \mathcal{r}$ as $m,n\rightarrow\infty$.
So, (i) follows.
\end{proof}

\begin{proof}[Proof of Theorem \ref{thm:robEstofR}]
(i, ii) It is trivial by the definition of $\widehat{r}_{\lrt}^{\rob}$.
(iii) Applying Taylor's expansion to $\psi\mapsto\loglik^{(\ell)}(\psi)$ again, 
we know there is $\widecheck{\psi}{}^{(\ell)}$ 
lying on the line segment joining $\widehat{\psi}^{(\ell)}$ and $\widehat{\psi}^{*}$ such that 
\begin{equation}\label{eqt:proof_logLR_robr}
	\loglik^{(\ell)}(\widehat{\psi}^{*}) 
		= \loglik^{(\ell)}(\widehat{\psi}^{(\ell)}) 
				- \frac{1}{2} \left( \widehat{\psi}^{*} - \widehat{\psi}^{(\ell)} \right)^{\T} 
				 I^{(\ell)}(\widecheck{\psi}{}^{(\ell)}) \left( \widehat{\psi}^{*} - \widehat{\psi}^{(\ell)} \right).
\end{equation}
By the lower order variability of $I^{(\ell)}(\widecheck{\psi}{}^{(\ell)})$, 
we know that 
$I^{(\ell)}(\widecheck{\psi}{}^{(\ell)}) \bumpeq  \overline{I}(\widehat{\psi}^{*})$ for all $\ell$, 
where $\overline{I}(\psi)= m^{-1}\sum_{\ell=1}^m I^{(\ell)}(\psi)$.
We also know that $\widehat{\psi}^* \bumpeq \overline{\psi}$.
Thus
\begin{eqnarray}
	\overline{\delta}_{\lrt} - \widehat{\delta}_{\lrt}
		&\bumpeq& \frac{1}{m}\sum_{\ell=1}^m 
			\left( \widehat{\psi}^{*} - \widehat{\psi}^{(\ell)} \right)^{\T} 
				 \overline{I} (\widehat{\psi}^{*}) \left( \widehat{\psi}^{*} - \widehat{\psi}^{(\ell)} \right)\nonumber \\
		&=& \tr \left\{ 
				 \overline{I}(\widehat{\psi}^{*})
				\frac{1}{m}\sum_{\ell=1}^m 
				\left( \widehat{\psi}^{*} - \widehat{\psi}^{(\ell)} \right)^{\otimes 2}
			\right\} \nonumber\\
		&\bumpeq& \tr \left\{ 
				 \overline{I}(\widehat{\psi}^{*})
				\frac{1}{m}\sum_{\ell=1}^m 
				\left( \widehat{\psi}^{(\ell)} - \overline{\psi}\right)^{\otimes 2}
			\right\} 
		\bumpeq \tr\left( \mathcal{U}_{\psi}^{-1} \mathcal{B}_{\psi} \right)\label{eqt:diffDeltaL_traceForm}
\end{eqnarray}
as $m,n\rightarrow\infty$.
By the assumption of EFMI of $\psi$, 
we have $\widehat{r}_{\lrt}^{\rob} \inP \mathcal{r}$.
\end{proof}

\begin{proof}[Proof of Lemma~\ref{thm:distHatrPos}]
First, recall that, as $n\rightarrow\infty$, 
the observed data MLE $\widehat{\theta}_{\obs}$ of $\theta$ satisfies 
(\ref{eqt:psiObsHatNormal}), which can be written as 
$[\widehat{\theta}_{\obs}  \mid \theta ] \simApprox \mathcal{N}_k(\theta ,\mathcal{T}_{\theta}), $
where 
$A_{1,n} \simApprox A_{2,n}$ means that
$A_{1,n}$ and $A_{2,n}$ have the same asymptotic distribution, 
i.e., there exist deterministic sequences $\mu_n$ and $\Sigma_n$ such that 
$(A_{1,n}-\mu_n)\Sigma_n^{-1/2} \inD A$ and 
$(A_{2,n}-\mu_n)\Sigma_n^{-1/2}\inD A$ 
for some non-degenerate random variable $A$.
From Assumption \ref{ass:properImpModel}, a proper imputation model is used. So,
we have (\ref{eqt:psiellHatNormal}), which is equivalent to say that, 
as $n\rightarrow\infty$, 
\begin{equation}
	\left[\widehat{\theta}^{(\ell)}  \mid X_{\obs} \right] 
		\simApprox \mathcal{N}_k(\widehat{\theta}_{\obs}, \mathcal{B}_{\theta}), 
\end{equation}
independently for for $\ell=1,\ldots,m$.
Therefore we can represent 
\begin{eqnarray}
	\widehat{\theta}_{\obs} 
		&\simApprox& \theta + \mathcal{T}_{\theta}^{1/2} W  , \label{eqt:first_representation_thetas1}\\
	\widehat{\theta}^{(\ell)} 
		&\simApprox& \widehat{\theta}_{\obs} + \mathcal{B}_{\theta}^{1/2} Z_{\ell}, \qquad \ell=1,\ldots,m \label{eqt:first_representation_thetas2}
\end{eqnarray}
where $Z_{1},\ldots, Z_m, W \simIID\mathcal{N}_k(0,I_k)$.
Also write $Z_{\ell}=(Z_{1\ell},\ldots,Z_{k\ell})^{\T}$, for $\ell=1,2,\ldots,m$,
and $W=(W_1,\ldots,W_k)^{\T}$.
Averaging (\ref{eqt:first_representation_thetas2}) over $\ell$, 
we have $\overline{\theta} \simApprox \widehat{\theta}_{\obs} + \mathcal{B}_{\theta}^{1/2} \overline{Z}_{\bullet}$,
where $\overline{Z}_{\bullet} = m^{-1}\sum_{\ell=1}^m Z_{\ell}$.
Since $\mathcal{B}_{\theta} = \mathcal{r}\mathcal{U}_{\theta}$,
we have 
\begin{eqnarray*}
	\mathcal{U}^{-1/2}_{\theta} (\widehat{\theta}^{(\ell)} - \theta )
		&\simApprox& (1+\mathcal{r})^{1/2}W + \mathcal{r}^{1/2}Z_{\ell} ,\\
	\mathcal{U}^{-1/2}_{\theta} (\overline{\theta} - \theta )
		&\simApprox& (1+\mathcal{r})^{1/2}W + \mathcal{r}^{1/2}\overline{Z}_{\bullet}.
\end{eqnarray*}
Note that (\ref{eqt:lowerOrderVarAss}) implies $\mathcal{U}_{\theta} \bumpeq \overline{U}$.
Under $H_0$, we have $\theta = \theta_0$ and 
\begin{eqnarray*}
	\overline{d}_{\lrt} &\bumpeq& \overline{d}_{\wt}^{\prime}
		\;\;\simApprox\;\; \sum_{i=1}^k \left\{ (1+\mathcal{r})^{1/2}W_i + \mathcal{r}^{1/2}Z_{i\ell} \right\}^2, \\
	\widehat{d}_{\lrt} &\bumpeq& \widetilde{d}_{\lrt} 
		\;\;\bumpeq\;\;  \widetilde{d}_{\wt}^{\prime}
		\;\;\simApprox\;\; \sum_{i=1}^k \left\{ (1+\mathcal{r})^{1/2}W_i + \mathcal{r}^{1/2}\overline{Z}_{i} \right\}^2. 
\end{eqnarray*}
After some simple algebra, we obtain
\begin{equation*}
	\widehat{r}_{\lrt}^+ 
		\simApprox \frac{(m+1)\mathcal{r}}{mk}\sum_{i=1}^k  s_{Z_i}^2 
		\qquad\text{and}\qquad
	\widehat{D}_{\lrt}^+
		\simApprox \frac{m\sum_{i=1}^k \left\{ (1+\mathcal{r})^{1/2}W_i + \mathcal{r}^{1/2}\overline{Z}_{i\bullet} \right\}^2}{mk + (m+1)\mathcal{r}\sum_{i=1}^k  s_{Z_i}^2 },
\end{equation*}
where $s_{Z_i}^2 = (m-1)^{-1} \sum_{\ell=1}^m (Z_{i\ell} - \overline{Z}_{i\bullet})^2$ is the 
sample variance of $\{Z_{i\ell}\}_{\ell=1}^m$.
Since $W_i$, $\overline{Z}_{i\bullet}$ and $s_{Z_i}^2$ are mutually independent for each fixed $i$, 
we can simplify the representation of $\widehat{D}_{\lrt}^+$ to 
\begin{equation*}
	\widehat{r}_{\lrt}^+ 
		\simApprox \frac{(m+1)\mathcal{r}}{m(m-1)k}\sum_{i=1}^k  H_i^2
		\qquad\text{and}\qquad
\widehat{D}_{\lrt}^+
		\simApprox \frac{(m-1)\{m+(m+1)\mathcal{r}\}\sum_{i=1}^k G^2_{i} }{m(m-1)k + (m+1)\mathcal{r}\sum_{i=1}^k  H^2_{i} },  
\end{equation*}
where $G^2_{i} \simIID \chi^2_1$ and $H^2_i \simIID \chi^2_{m-1}$, for $i=1,\ldots,k$,
are all mutually independent.
Clearly, they can be further simplified to (\ref{eqt:asyDistrPos}).
\end{proof}

\begin{proof}[Proof of Theorem~\ref{thm:distHatrRob}]
Similar to (\ref{eqt:first_representation_thetas1}) and (\ref{eqt:first_representation_thetas2}), 
we can have a more general representation: 
\begin{eqnarray*}
	\widehat{\psi}_{\obs} \simApprox \psi + \mathcal{T}_{\psi}^{1/2} W; \quad
	\widehat{\psi}^{(\ell)} \simApprox\widehat{\psi}_{\obs} + \mathcal{B}_{\psi}^{1/2} Z_{\ell}, 
	\qquad \ell=1,\ldots,m ,   
\end{eqnarray*}
where $Z_1, \ldots, Z_h, W\simIID\mathcal{N}_h(0,I_h)$.
Also write $Z_{\ell}=(Z_{1\ell},\ldots,Z_{h\ell})^{\T}$, for $\ell=1,2,\ldots,m$,
and $W=(W_1,\ldots,W_h)^{\T}$.
Using (\ref{eqt:diffDeltaL_traceForm}), we have
\begin{eqnarray*}
	\overline{\delta}_{\lrt}	- \widehat{\delta}_{\lrt}
		&\bumpeq&  \tr \left\{ 
				 \overline{I}(\widehat{\psi}^{*})
				\frac{1}{m}\sum_{\ell=1}^m 
				\left( \widehat{\psi}^{(\ell)} - \overline{\psi}\right)
				\left( \widehat{\psi}^{(\ell)} - \overline{\psi}\right)^{\T}
			\right\} \\
		&\simApprox& \tr \left\{
				 \mathcal{U}_{\psi}^{-1}
				\frac{1}{m}\sum_{\ell=1}^m 
				\left[\left( \mathcal{T}_{\psi} - \mathcal{U}_{\psi}\right)^{1/2}
				\left(Z_{\ell} - \overline{Z}_{\bullet} \right) \right]^{\otimes 2}
			\right\} \\
		&=& \frac{1}{m}\sum_{\ell=1}^m\tr\left\{
			\mathcal{r} I_h \left( Z_{\ell} - \overline{Z}_{\bullet}\right)^{\otimes 2}
			\right\} 
		= \frac{\mathcal{r}}{m}\sum_{\ell=1}^m \sum_{i=1}^h (Z_{i\ell} - \overline{Z}_{i\bullet})^2.
\end{eqnarray*}
Equivalently, we can say 
$\overline{\delta}_{\lrt}	- \widehat{\delta}_{\lrt}\inD \mathcal{r} \chi^2_{h(m-1)} /m$
as $n\rightarrow\infty$.
Hence
\[	
	\widehat{r}_{\lrt}^{\rob} \inD \mathcal{r} \cdot \frac{ m+1}{hm(m-1)} \cdot \chi^2_{h(m-1)} ,
\]
which is equivalent to (\ref{eqt:asyDistrRob}).
Note that it is true under both $H_0$ and $H_1$.
\end{proof}

\begin{proof}[Proof of Theorem~\ref{thm:exactNullDistRrob}]
From the representations of $\widehat{d}_{\lrt}^{\rob}$ and $\widehat{r}_{\lrt}^{\rob}$
in Lemma \ref{thm:distHatrPos} and Theorem \ref{thm:distHatrRob}, 
we know that they are asymptotically ($n\rightarrow\infty$) independent. The proof then follows the derivation for Lemma \ref{thm:distHatrPos}. 
\end{proof}

\begin{proof}[Proof of Theorem~\ref{thm:AsyEq_r_LL_Delta}]
(i) Using the representation (\ref{eqt:alternativeDeltad}), 
we can easily see that $\widehat{r}_{\lrt}^{\pert}\geq 0$.
(ii) It suffices to show 
\[
	m^{-1}\sum_{\ell=1}^m d_{\lrt}(\widehat{\psi}_{0}^{(\ell)}+\Delta_{m},\widehat{\psi}^{(\ell)}\mid X^{(\ell)}) 
\bumpeq \overline{d}_{\lrt} - \widetilde{d}_{\lrt}, 
\]
where $\Delta_{m} = \widehat{\psi}^{*} - \widehat{\psi}_{0}^{*}$.
Under $H_0$, $\Delta_{m}\bumpeq 0$ and 
$\widehat{\psi}_{0}^{(\ell)} \bumpeq \widehat{\psi}^{(\ell)}$,
so $\widehat{\psi}_{0}^{(\ell)}+\Delta_{m} \bumpeq \widehat{\psi}^{(\ell)}$.
Using Taylor's expansion on 
$\psi\mapsto\loglik^{(\ell)}(\psi)$ around 
its maximizer $\widehat{\psi}^{(\ell)}$, we have
for $\psi\bumpeq \widehat{\psi}^{(\ell)}$ that 
\[
	\loglik^{(\ell)}(\psi) 
		\bumpeq \loglik^{(\ell)}(\widehat{\psi}^{(\ell)}) 
				- \frac{1}{2} \left( \psi - \widehat{\psi}^{(\ell)} \right)^{\T} 
				 I^{(\ell)}(\widehat{\psi}^{(\ell)}) \left( \psi - \widehat{\psi}^{(\ell)} \right).
\]
Under the parametrization of $\psi$ in the proof of Theorem~\ref{thm:AsyEq_dBar_LL}, 
we know that the upper $k\times k$ sub-matrix of $ I^{(\ell)}(\widehat{\psi}^{(\ell)})$ is $\left( U^{(\ell)} \right)^{-1}$.
Using the lower order variability of $U^{(\ell)}$, 
we have $\left( U^{(\ell)} \right)^{-1} \bumpeq \overline{U}^{-1}$ and 
\begin{eqnarray*}
	\frac{1}{m}\sum_{\ell=1}^m d_{\lrt}(\widehat{\psi}_{0}^{(\ell)}+\Delta_{m},\widehat{\psi}^{(\ell)}\mid X^{(\ell)})\nonumber
	&\bumpeq& \frac{1}{m}\sum_{\ell=1}^m  \left( \widehat{\psi}_{0}^{(\ell)}+\Delta_{m} - \widehat{\psi}^{(\ell)} \right)^{\T} 
				 I^{(\ell)}(\widehat{\psi}^{(\ell)}) \left( \widehat{\psi}_{0}^{(\ell)}+\Delta_{m} - \widehat{\psi}^{(\ell)} \right) \nonumber\\
	&\bumpeq& \frac{1}{m}\sum_{\ell=1}^m (\widehat{\theta}^{(\ell)}-\overline{\theta})^{\T} \overline{U}^{-1} (\widehat{\theta}^{(\ell)}-\overline{\theta}) = \overline{d}_{\wt}^{\prime} - \widetilde{d}_{\wt}^{\prime} 
    \bumpeq \overline{d}_{\lrt} - \widehat{d}_{\lrt} . 
\end{eqnarray*}
Therefore, the desired result follows. 
\end{proof}

\begin{proof}[Proof of Theorem \ref{prop:approxMLE}]
Throughout this proof, conditions (a), (b) and (c) 
refer to the list given in Assumption \ref{ass:for_propApproxMLE}. 
(i, ii) It trivially follows from the definitions of $\widehat{d}_{\Stack}$
and $\widehat{r}_{\Stack}$.
(iii) First, by the definition of maximizer and condition (a), we have 
\begin{eqnarray}
		\underline{\overline{\loglik}}(\widehat{\psi}^{*}) - \underline{\overline{\loglik}}(\widehat{\psi}^{\Stack}) 
		&=& \underline{\overline{\loglik}}(\widehat{\psi}^{*}) - \underline{\overline{\loglik}}^{\Stack}(\widehat{\psi}^{\Stack}) 
			+	\underline{\overline{\loglik}}^{\Stack}(\widehat{\psi}^{\Stack}) - 	\underline{\overline{\loglik}}(\widehat{\psi}^{\Stack})  \nonumber\\
		&\leq& \underline{\overline{\loglik}}(\widehat{\psi}^{*}) - \underline{\overline{\loglik}}^{\Stack}(\widehat{\psi}^{*}) 
			+	\underline{\overline{\loglik}}^{\Stack}(\widehat{\psi}^{\Stack}) - 	\underline{\overline{\loglik}}(\widehat{\psi}^{\Stack})  \nonumber\\
		&\leq& 2\sup_{\psi\in\Psi} \left\vert \underline{\overline{\loglik}}(\psi) 
				- \underline{\overline{\loglik}}^{\Stack}(\psi)\right\vert \nonumber 
		= O_p(1/n),  
		\label{eqt:diffLL_Opm}
\end{eqnarray}
which, together with condition (b), imply that
\begin{eqnarray}
		\overline{\underline{\mathcal{L}}}(\psi_{}^{*}) - \overline{\underline{\mathcal{L}}}(\widehat{\psi}^{\Stack}) 
		&=& \left\{\overline{\underline{\mathcal{L}}}(\psi_{}^{*}) - \underline{\overline{\loglik}}(\psi_{}^{*})\right\}
			+ \left\{\underline{\overline{\loglik}}(\psi_{}^{*}) - \underline{\overline{\loglik}}(\widehat{\psi}^{\Stack}) \right\} + \left\{ \underline{\overline{\loglik}}(\widehat{\psi}^{\Stack}) - \overline{\underline{\mathcal{L}}}(\widehat{\psi}^{\Stack}) \right\} \nonumber \\ 
		& \leq& 2\sup_{\psi\in\Psi} \left\vert \underline{\overline{\loglik}}(\psi) - \overline{\underline{\mathcal{L}}}(\psi) \right\vert 
			+ \left\{\underline{\overline{\loglik}}(\widehat{\psi}^{*}) - \underline{\overline{\loglik}}(\widehat{\psi}^{\Stack}) \right\} = o_p(1).\label{eqt:diffELL}
\end{eqnarray}
Using (\ref{eqt:diffELL}) and (c), 
we have $\widehat{\psi}^{\Stack}\inP\psi_{}^{*}$. 
By (b) and (c), we also have $\widehat{\psi}^{*}\inP\psi_{}^{*}$.
So,
$\left\vert \widehat{\psi}^{\Stack}-\widehat{\psi}^{*} \right\vert \inP \mathbf{0}$
as $n\rightarrow\infty$.
By the definition of maximizer, 
\begin{equation}\label{eqt:defMax}
	\mathbf{0}=\nabla\underline{\overline{\loglik}}^{\Stack}(\widehat{\psi}^{\Stack})
	= \nabla\underline{\overline{\loglik}}(\widehat{\psi}^{\Stack})
		+ \nabla R(\widehat{\psi}^{\Stack}), 
\end{equation}
where $\nabla g(\psi) = \partial g(\psi)/\partial\psi$ is the gradient of $\psi\mapsto g(\psi)$.
By condition (a), we know $\nabla R(\widehat{\psi}^{\Stack})=O_p(1/n)$.
Thus, together with (\ref{eqt:defMax}), 
we have $\nabla\underline{\overline{\loglik}}(\widehat{\psi}^{\Stack})=O_p(1/n)$.
Also, by the definition of MLE, 
we have $\nabla\underline{\overline{\loglik}}(\widehat{\psi}^{*})=\mathbf{0}$.

By Taylor's expansion, 
there exists $\widecheck{\psi}$ 
such that 
\begin{eqnarray}\label{eqt:diffLbarunderline}
	\underline{\overline{\loglik}}(\widehat{\psi}^{*})
		-\underline{\overline{\loglik}}(\widehat{\psi}^{\Stack}) 
		&=& \left\{ \nabla\underline{\overline{\loglik}}(\widecheck{\psi}) \right\}^{\T}
			\left( \widehat{\psi}^{*}- \widehat{\psi}^{\Stack} \right) = o_p(1/n),
\end{eqnarray}
where we have used the continuity of $\psi \mapsto \nabla\underline{\overline{\loglik}}(\psi)$
to yield $\nabla\underline{\overline{\loglik}}(\widecheck{\psi}) = O_p(1/n)$.
Rewriting (\ref{eqt:diffLbarunderline}), we have
\begin{equation}\label{eqt:diffLL1}
	\underline{\overline{\loglik}}(\widehat{\psi}^{*})
		-\underline{\overline{\loglik}}^{\Stack}(\widehat{\psi}^{\Stack})
		= R(\widehat{\psi}^{\Stack}) + o_p(1/n).
\end{equation}
Similar to (\ref{eqt:diffLL1}), we have 
\begin{equation}\label{eqt:diffLL2}
	\underline{\overline{\loglik}}(\widehat{\psi}_{0}^{*})
		-\underline{\overline{\loglik}}^{\Stack}(\widehat{\psi}_{0}^{\Stack})
		= R(\widehat{\psi}_{0}^{\Stack}) + o_p(1/n).
\end{equation}
Then, using (\ref{eqt:diffLL1}) and (\ref{eqt:diffLL2}), 
we have 
\begin{eqnarray*}
	\left\vert \widehat{d}_{\lrt}- \widehat{d}_{\Stack} \right\vert 
	 &=& 2n\left\vert
	\left\{ \underline{\overline{\loglik}}(\widehat{\psi}^{*}) 
		- \underline{\overline{\loglik}}^{\Stack}(\widehat{\psi}^{\Stack}) 
		 \right\}
	- \left\{ 
		\underline{\overline{\loglik}}(\widehat{\psi}_{0}^{*})
		- \underline{\overline{\loglik}}^{\Stack}(\widehat{\psi}_{0}^{\Stack}) \right\}\right\vert \\
	&=& 2n \left\vert R(\widehat{\psi}^{\Stack}) - R(\widehat{\psi}_{0}^{\Stack}) + o_p(1/n)\right\vert .
\end{eqnarray*}

Now consider two cases.
\begin{itemize}[noitemsep]
	\item[(i)] Under $H_0$, 
			we have $\widehat{d}_{\lrt}=O_p(1)$ and 
			$\widehat{\psi}_{0}^{\Stack}\bumpeq \widehat{\psi}^{\Stack}$.
			Thus condition (a) 
			implies $R(\widehat{\psi}^{\Stack}) - R(\widehat{\psi}_{0}^{\Stack})=o_p(1/n)$.
			Then,
			we have $\left\vert \widehat{d}_{\lrt}- \widehat{d}_{\Stack} \right\vert = o_p(\widehat{d}_{\lrt})$.
	\item[(ii)] Under $H_1$, we have 
			$\widehat{d}_{\lrt}\inP\infty$. 
			Condition (a) and (\ref{eqt:diffLL_Opm}) imply that 
			$\underline{\overline{\loglik}}(\widehat{\psi}^{*}) 
					- \underline{\overline{\loglik}}^{\Stack}(\widehat{\psi}^{\Stack}) = O_p(1/n)$. 
			Similarly, we also have 
			$\underline{\overline{\loglik}}(\widehat{\psi}_{0}^{*})
					- \underline{\overline{\loglik}}^{\Stack}(\widehat{\psi}_{0}^{\Stack})=O_p(1/n)$.
			Hence $\left\vert \widehat{d}_{\lrt}- \widehat{d}_{\Stack} \right\vert=O_p(1)$.
			Thus we have $\left\vert \widehat{d}_{\lrt}- \widehat{d}_{\Stack} \right\vert = o_p(\widehat{d}_{\lrt})$.
\end{itemize}
Therefore, under either $H_0$ or $H_1$, we also have 
$\left\vert \widehat{d}_{\lrt}- \widehat{d}_{\Stack} \right\vert = o_p(\widehat{d}_{\lrt})$.
Since $\widehat{d}_{\lrt} \bumpeq \widehat{d}_{\Stack}$
and $\overline{d}_{\lrt} = \overline{d}_{\Stack}$, 
we know $\widehat{r}_{\lrt} \bumpeq \widehat{r}_{\Stack}$.
\end{proof}

Note that, even under the assumption of this theorem, $\widehat{r}_{\Stack}$ and $\widehat{r}^{\rob}_{\Stack}$ are not equivalent. 
From (\ref{eqt:def_hatrs}) and (\ref{eqt:def_rharrobCompeasy}), 
$\widehat{r}_{\Stack}$ and $\widehat{r}^{\rob}_{\Stack}$
are a ``difference of difference'' estimator 
and a ``difference" estimator, respectively. 
So, the ``bias'' of using $\overline{\loglik}^{\Stack}(\psi)$
cannot be canceled out in $\widehat{r}^{\rob}_{\Stack}$.

\end{document}